\documentclass[a4paper,12pt,reqno]{article}
\usepackage{titlesec}  
\usepackage{authblk}
\usepackage{amsmath}
\usepackage{amsthm}
\usepackage{amssymb}
\usepackage{latexsym}
\usepackage{graphicx}
\usepackage{stmaryrd}
\usepackage{mathrsfs}
\usepackage{enumerate}
\usepackage{color}
\usepackage[colorlinks, linkcolor=red,citecolor=blue]{hyperref}
\font\sy=rsfs10     scaled 1200

\newcommand{\BLACK}{\color{black}}

\definecolor{dGREEN}{rgb}{0.0,0.5,0.5}

\newcommand{\glalign}[2]{\lower.6ex\vbox{
		\baselineskip\lineskip\ialign{$#1\hfil##\hfil$\crcr#2\crcr=\crcr}}}

\newcommand{\delx}{\partial_{xx}}

\newcommand{\scr}[1]{{\mbox{\sy #1}\,}}

\newcommand{\del}{\partial}

\newcommand{\delt}{\partial_{t}}

\renewcommand{\delx}{\partial_{x}}

\newcommand{\goto}{\rightarrow}

\renewcommand{\div}{\mbox{\rm div}\,}

\newcommand{\supp}{\mbox{\rm supp}\,}

\newcommand{\trans}{{}^\top}

\def\eqn#1$$#2$${\begin{equation}\label#1#2\end{equation}}

\numberwithin{equation}{section}

\newtheorem{defi}{Definition}[section]
\newtheorem{thm}[defi]{Theorem}

\newtheorem{prop}[defi]{Proposition}
\newtheorem{lem}[defi]{Lemma}
\newtheorem{rem}[defi]{Remark}
\def\eqn#1$$#2$${\begin{equation}\label#1#2\end{equation}}

\numberwithin{equation}{section}
\numberwithin{equation}{section}
\allowdisplaybreaks[4]

\begin{document}

	%------------------------------------------------------------
	%     Section 1.   
	%------------------------------------------------------------
	
	\title{
		Time periodic problem 
		of compressible Euler equations with damping
		on the whole space }
	\author[a]{Houzhi Tang \thanks{E-mail: houzhitang@ahnu.edu.cn (H.-Z. Tang)}}
	\author[b]{Kazuyuki Tsuda \thanks{E-mail: k-tsuda@ip.kyusan-u.ac.jp (K. Tsuda)}}
	
	\affil[a]{School of Mathematics and Statistics, Anhui Normal University
		\par Wuhu 241002, P. R. China}
	
	\affil[b]{Kyushu Sangyo University, Fukuoka, 813-8503, Japan}
	
	\date{}
	\maketitle
	
	%%%%%%%%%%%%%%%%%%%%%%
	\begin{abstract}
		\noindent
		In this article, time periodic problem  
		of the compressible Euler equations with damping on the whole space 
		is studied. It is well known that in the Euler system, long-time behavior of solutions is a more delicate problem  due to lack of the viscosity.  By virtue of a  damping effect, 
		time global solutions barely exist.  Under such circumstances, existence of a time periodic solution is obtained for sufficiently  small time periodic external force 
		when the space dimension is greater than or equal to $3$. In addition, its stability is also obtained. The solution is 
		asymptotically stable 
		under sufficiently small initial perturbations 
		and the $L^\infty$ norm of the perturbation decays as time goes to infinity.  The potential theoretical estimates work well on a low frequency part of solutions, while a new energy estimate with weights is established to avoid derivative loss.  
	\end{abstract}
	\textbf{2020 MR Subject Classification:}\ 35Q31, 35B10, 35B35, 35B40.
	
	%------------------------------------------------------------
	%     Section 1.   
	%------------------------------------------------------------
	
	\section{Introduction}
	
	We consider time periodic problem of the following Euler equations with damping in $\mathbb{R}^d$ $(d\geq 3)$:
	\begin{equation}\label{Main1}
		\left\{
		\begin{aligned}
			&\partial_{t}\rho  +\div (\rho v) =0,\\
			&\partial_{t}(\rho v) +\div (\rho v\otimes v) +\nabla P(\rho) +\rho v=\rho g,
		\end{aligned}
		\right.
	\end{equation}
	\noindent where $\rho=\rho(x,t)$ denotes the unknown density, and 
	$v=\trans{(v_{1}(x,t),\cdots,v_{d}(x,t))}$ is the corresponding velocity field, respectively, at time $t\geq 0$ and 
	position $x\in\mathbb{R}^d$; 
	The pressure $P=P(\rho)$ is supposed to be a smooth function of $\rho$ satisfying 
	$$
	P'(\rho_*)>0
	$$
	for a given positive constant $\rho_*$;  The function $g=g(x,t)$ is given external force  satisfying the condition
	\begin{equation}
		\begin{array}{rcl}
			g(x,t+T)&=&g(x,t)\quad(x\in \mathbb{R}^d,\,t\in\mathbb{R})
		\end{array}
		\label{assumption-g}
	\end{equation}
	for some constant $T>0$.
	
	Time periodic flow is one of basic phenomena in fluid mechanics, and thus, time periodic problems for 
	fluid dynamical equations have been extensively studied. We refer, e.g., to \cite{Kaniel-Shinbrot, Kozono-Nakao, Serrin, Yamazaki} for the incompressible Navier-Stokes case, and to \cite{Brezina,  Feireisl-Matusu-Necasova-Petzeltova-Straskrava, Feireisl- Mucha-Novotny-Pokorny, Kagei-Tsuda, Ukai, Valli} 
	for the compressible case. 
	In this paper we are interested in time periodic problem for the compressible fluid system without the viscosity  terms on unbounded domains. 
	As well known, solutions to the compressible fluid system without the viscosity  terms, {\rm i.e., } the compressible Euler system blow up. Indeed, classical results, John \cite{John}, Liu \cite{Liu1978} and Li, Zhou and Kong\cite{LZK}  confirmed that when initial data are small smooth perturbations near constant states, blowup in gradient of solutions occurs in finite time. We note that on the incompressible case, the blow up problem has been a big problem in long time, and recently Eligindi \cite{terk} obtained an important progress, {\rm i.e., } some blow up results of solutions.

	Therefore, as one of continuing works, adding some damping terms to the system  has been considered to get  time global solutions. We focus on the compressible Euler system with damping \eqref{Main1} as a typical system. 
	Recently, study of the compressible Euler system with damping has been developed by many researchers.  In multi space dimensions, Wang and Yang \cite{Wang-Yang} first studied time global solutions with time asymptotic behavior of them to \eqref{Main1}. Then 
	Sideris, Thomases, and Wang \cite{Sider-Wang} obtained the global smooth solutions under small data assumption, and proved the singularity formation of solutions for a class of large data. On the initial value problem, here we refer to Ji and Mei \cite{Ji-Mei}. They investigated the influence of the damping effect to time-decay estimates of solutions and showed that if the damping effect depends time and becomes weak as $t\rightarrow \infty$, the solutions may have time growth order estimates.     
	On the other hand, concerning  the time periodic problem, as far as we investigated, there are few results. Tan, Xu and Wang \cite{TXZ} studied in a periodic domain. They showed that  under some smallness and symmetry assumptions on the external force, we have the existence of the periodic solutions. Li and Yao \cite{Li-Yao} proved the existence of time periodic solutions to the full quantum Euler equation on $\mathbb{R}^3$ for a small time periodic external force.  
	%\BLUE(When we consider a periodic domain, we regard the problem as a bounded domain case. This derives exponential decay of solutions. When we consider the full quantum Euler equation, we can use the disspersion term of the velocity in the energy estimates and the density has higher order regularity.)
	\BLACK 
	
	Let us explain why we consider the time periodic problem and difficulty of the problem. The problem is deeply related to time asymptotic behavior of solutions. 
	The difficulty is that we consider the unbounded domains which gives polynomial time-decay rate of solutions,  with lower regularity solution spaces by lack of the viscosity. 
	
	On the compressible Navier-Stokes system, 
	Ma, Ukai, and Yang \cite{Ukai} proved the existence and stability of time periodic solutions on the whole space $\mathbb{R}^d$ with $d \geq 5$. 
	Since time periodic solution is one of time global solutions, we have to use the information of time asymptotic behavior of linear system.   Hence they consider the higher dimension case $d \geq 5$ because 
	the solution has faster decay rates as the hear kernel. We note that 
	if we consider bounded domain case like a periodic domain, linear  solutions have exponential decay rates in each dimension.   
	It was shown in \cite{Kagei-Tsuda} that, 
	for $d\geq 3$, if the external force $g$ satisfies the same symmetry condition 
	as \cite{TXZ}, there exists a time periodic solution. 
	The symmetry condition gives faster decay rate of linear solutions. 
	After that, Tsuda \cite{Tsuda2016} proved the existence of the time periodic solution of the compressible Navier-Stokes equations without symmetric condition for the external force by applying the potential theoretical method. 
	
	For \eqref{Main1}, the long-time behavior is more delicate by lack of the viscosity as above previous works.  Indeed, 
	if there is no damping, solutions to the Euler system blow up in finite time, which can be referred to \cite{Siders}. If the damping effect becomes weaker as $t \rightarrow \infty$, 
	the solutions also may blow up in infinite time.  
	By virtue of a firm damping effect, 
	we barely have time global solutions. This motivate us to consider time periodic problem related to time asymptotic behavior of solutions.  
	
	Furthermore, the lack of viscosity causes another serious problem: we have to consider lower order regularity solution space in the energy estimates. 
	Since we consider the hyperbolic system, we have to take care of derivative loss of  the density part of solutions like the compressible Navier-Stokes system. This technique is called as the Matsumura-Nishida energy estimate. In addition, for the Euler-damping system, derivative loss of the velocity part also has to be taken care of because there is no dispersion term about the velocity.  We note that if we consider the full quantum Euler equation, we can use the dispersion term  in the energy estimates and the density has higher order regularity.

	We consider \eqref{Main1}.	It will be proved that if $d\geq 3$ and if $g$ satisfies \eqref{assumption-g} and 
	\begin{align*}
		\|g\|_{C([0,T];L^1)}
		+\|(1+|x|^d)g\|_{C([0,T];L^\infty)}
		+\|(1+|x|^{d-1})g\|_{L^2(0,T;H^{s-1})}
		\ll 1
	\end{align*}
	with an integer $s\geq [d/2]+2$, then 
	there exists a time periodic solution 
	$(\rho_{per},v_{per})\in C([0,T];H^s)$ with period $T$ for \eqref{Main1}
	and $u_{per}(t)=(\rho_{per}(t)-\rho_{\ast},v_{per}(t))$ satisfies 
	\begin{align}\label{u_{per}esimate}
		&\sup_{t\in [0,T]}(
		\sum_{j=0}^1\|(1+|x|^{d-2+j})\delx^j\rho_{per}(t)\|_{L^\infty}
		+\| (1+|x|^{d-1}) v_{per}(t)\|_{L^{\infty}})\notag\\
		&
		\quad 
		\leq 
		C (\|g\|_{C([0,T];L^1)}+\|(1+|x|^d)g\|_{L^\infty(0,T;L^\infty)}
		+\|(1+|x|^{d-1})g\|_{L^2(0,T;H^{s-1})}).
	\end{align}
	Furthermore, if $g$ satisfies 
	$$
	\|g\|_{C([0,T];L^1)}
	+\|(1+|x|^d)g\|_{C([0,T];L^\infty)}
	+\|(1+|x|^{d-1})g\|_{L^2(0,T;H^{s})}
	\ll 1,
	$$
	then the time periodic solution $(\rho_{per},v_{per})$ is asymptotically stable 
	under sufficiently small initial perturbations, 
	and the perturbation satisfies 
	$$
	\|(\rho(t),v(t))-(\rho_{per}(t),v_{per}(t))\|_{L^\infty}\goto 0
	$$ 
	as $t\goto\infty$. 
	\BLACK
	%We expect that the decay estimate such as \eqref{decay} would also hold for this case 
	%and it would be desirable to derive the optimal decay estimate of $L^2$ norm for the perturbations. 
	The precise statements of our existence and stability results 
	are given in Theorem \ref{Theorem 3.1} and Theorem \ref{Theorem 3.2} below. 
	
	\vspace{2ex}
	Let us explain our method to prove the main results. 
	The existence of a time periodic solution around $(\rho_*,0)$ is shown  by an iteration argument 
	using the time-$T$-map associated with the linearized problem at $(\rho_*,0)$. 
	As in \cite{Kagei-Tsuda, Tsuda2016} we formulate the time periodic problem 
	as a system of equations for low frequency part and high frequency part of the solution. We need to investigate 
	$(I-S_{j}(T))^{-1}$ for $j=1$ implying low frequency part and $j=\infty$ implying the high frequency part respectively. Here $S_j(T)=e^{-TA}$ with $A$ being the linearized operator around $(\rho_{\ast},0)$ which 
	acts on functions whose Fourier transforms have their supports in $\{\xi \in \mathbb{R}^{d};|\xi|\leq r_{1}\}$ for some $r_{1}>0$ or $\{\xi \in \mathbb{R}^{d};|\xi|\geq r_{\infty}\}$  for some $r_{\infty}>0$. Concerning the low frequency part, we apply the potential theoritical estimates, while the weighted energy estimates is used in the high frequency part as in \cite{Tsuda2016}. 
	
	There are two key observations. One is the decay of solution $v$ as $|x|\rightarrow \infty$ is faster than
	\cite{Tsuda2016} by virtue of the hyperbolic system.  In \cite{Tsuda2016}, 
	since the spatial decay at infinity coincides with that of the fundamental solution of the diffusion equation by the viscosity,  we need more delicate analysis to estimate the convective term $v\cdot\nabla v$, {\rm i.e.,} to reformulate to the momentum equation to get the divergence form. However, in our system $v$ is governed by the hyperbolic system due to lack of the viscosity, which derives faster spatial decay at infinity than  \cite{Tsuda2016}. Hence we can omit such a detailed analysis. We note that on the density part, the solution has slow spatial decay at infinity, however, the nonlinear term related to the density, that is, the pressure term $\nabla (P(\rho)\rho)$ has already the derivative form like the divergence form.   This implies possibility that time periodic solution can be constructed easier than 
	the usual compressible Navier-Stokes system although there are few previous results on the Euler system with damping terms so far. 
	
	Another point is  techniques to avoid derivative loss in the energy estimates. 
	As for the high frequency part, we employ the weighted energy estimates 
	established in \cite{Kagei-Tsuda}.  Since the solution belongs to $C([0,T]; H^s)$, especially, there is no dispersion term on $v$ due to Euler system, we have to take care of derivative loss in the energy estimate.  In taking care of the derivative loss, we cannot apply the previous method used in the initial value problem as in  \cite{Wang-Yang}. Because when we use the iteration, the numbers of solutions between linear part and nonlinear part is different, {\rm i.e., } there appears, for example,  
	$$
	(v^{(k-1)}\cdot \nabla v^{(k-1)}, v^{(k)} )_{L^2}
	$$
	in the energy estimate, where $k \in\mathbb{N}$.  Hence the integration by parts does not work for such a term to avoid derivative loss.  To overcome this difficulty, we extend the Matsumura-Nishida energy estimate. We regard such nonlinear terms as linear terms by taking one of solutions as given functions, for instance, $\tilde{v} \cdot \nabla v$ with a given function $\tilde{v}$.  Then we  include the linear terms to the linear system.   See \eqref{Mainsystem} below. We establish weighted energy estimates for \eqref{Mainsystem}. We note that in the weighted energy estimates, some reminder terms related to weights appear, for example, 
	$$
	(\tilde{v} \cdot \nabla v, \nabla(|x|^\ell) v)_{L^2},
	$$
	where $\ell \in \mathbb N$. 
	However, such terms do not have any derivative loss of solutions because derivatives act to the weights.  This implies that weighted energy estimates can be applicable to the Euler system.

	Finally, the asymptotic stability of the time periodic solution $(\rho_{per},v_{per})$, which had not been obtained in previous results,   can be proved as in 
	the argument in Kagei and Kawashima \cite{Kagei-KawashimaCMP} by using the Hardy inequality. 
	%It seems, however, that a perturbation argument for the linearized problem as in 
	%\cite{Kagei-Tsuda, Okita} does not work well to derive the optimal decay estimate 
	%because of the slow decay of  $v_{per}(x,t)$ as $|x|\goto\infty$;  
	%and a more refined perturbation analysis would be needed. 
	
	\vspace{2ex}
	This paper is organized as follows. 
	In section \ref{S2}, we introduce notations and auxiliary lemmas used in this paper. 
	In section \ref{S3}, we state main results of this paper. 
	Section \ref{S4} is devoted to the reformulation of the problem. 
	We will rewrite the system for the low and high frequency parts into a system of integral equations in terms of the time-$T$-map. 
	In section \ref{S5}, we study the low frequency part and derive the necessary estimates 
	for the time-$T$-map of the low frequency part. 
	In section \ref{S6}, we state some spectral properties of 
	the time-$T$-map of the high frequency part. 
	In section \ref{S7}, we estimate nonlinear terms 
	and then give a proof of the existence of a time periodic solution by the iteration argument. In section \ref{S8}, we prove the stability of the time periodic solution.

	%------------------------------------------------------------
	%     Section 2.   
	%------------------------------------------------------------

	\section{Preliminaries}\label{S2}
	In this section we first introduce some notations which will be used throughout this paper. We then introduce some auxiliary lemmas which will be useful in the proof of the main results.\\
	
	For a given Banach space $X,$ the norm on $X$ is denoted by $\|\cdot\|_{X}$.
	
	Let $1\leqq p\leqq \infty.$ We denote by $L^p$ the usual $L^p$ space over $\mathbb{R}^d$. 
	The inner product of $L^2$ is denoted by $(\cdot , \cdot)$. 
	For a nonnegative integer $k$, we denote by $H^k$ the usual $L^2$-Sobolev space of order $k$. 
	(As usual, $H^0=L^2$.) 
	
	We simply denote by $L^p$ the set of all vector fields 
	$ v=\trans(v_1,\cdots,v_d)$ on $\mathbb{R}^d$ 
	with $v_j\in L^p$ $(j=1,\cdots,d)$, i.e., $(L^p)^d$  
	and the norm $\|\cdot\|_{(L^p)^d}$ on it is denoted by $\|\cdot\|_{L^p}$ 
	if no confusion will occur. 
	Similarly, for a function space $X$, 
	the set of all vector fields 
	$v=\trans(v_1,\cdots,v_d)$ on $\mathbb{R}^d$ 
	with $v_j\in X$ $(j=1,\cdots,d)$, i.e., $X^d$, is simply denoted by $X$; 
	and the norm $\|\cdot\|_{X^d}$ on it is denoted by $\|\cdot\|_{X}$ if no confusion will occur. 
	(For example, $(H^k)^d$ is simply denoted by $H^k$ 
	and the norm $\|\cdot\|_{(H^k)^d}$ is denoted by $\|\cdot\|_{H^k}$.)
	
	Let $u=\trans(a,v)$ with $\phi\in H^k$ and $v=\trans(v_1,\cdots,v_d)\in H^m$.  
	we denote the norm of $u$ on $H^k\times H^m$ by $\|u\|_{H^k\times H^m}$:
	$$
	\|u\|_{H^k\times H^m}=\left(\|a\|_{H^k}^2+\|v\|_{H^m}^2\right)^{\frac{1}{2}}.
	$$
	When $m=k$, the space $H^k\times (H^k)^d$ is simply denoted by $H^k$, and, also, 
	the norm $\|u\|_{H^k\times (H^k)^d}$ by $\|u\|_{H^k}$ 
	if no confusion will occur : 
	$$
	H^k:=H^k\times (H^k)^d, 
	\ \ \ 
	\|u\|_{H^k}:=\|u\|_{H^k\times (H^k)^d} 
	\ \ \ (u=\trans(a,v)).
	$$
	Similarly, for $u=\trans(a,v)\in X\times Y$ with $v=\trans(v_1,\cdots,v_d)$ 
	, 
	we denote its norm $\|u\|_{X\times Y}$ by $\|u\|_{X\times Y}$:
	$$
	\|u\|_{X\times Y}=\left(\|a\|_{X}^2+\|v\|_{Y}^2\right)^{\frac{1}{2}}
	\ \ \ (u=\trans(a,v)).
	$$
	If $Y=X^d$, 
	we simply denote $X\times X^d$ by $X$, 
	and, its norm $\|u\|_{X\times X^d}$ by $\|u\|_X$: 
	$$
	X:=X\times X^d, 
	\ \ \ 
	\|u\|_{X}:=\|u\|_{X\times X^d}
	\ \ \ (u=\trans(a,v)).
	$$
	
	We will work on function spaces with spatial weight. 
	For a nonnegative integer $\ell$ and $1\leq p \leq \infty$, we denote by $L^{p}_{\ell}$ the weighted $L^{p}$ space defined by
	$$
	L^{p}_{\ell}=\{u\in L^{p}; \|u\|_{L^{p}_{\ell}}:=\|(1+|x|)^{\ell}u\|_{L^{p}}<\infty\}.
	$$
	
	We denote the Fourier transform of $f$ by $\hat{f}$ or $\mathcal{F}[f]$: 
	\begin{eqnarray}
		\hat{f}(\xi)
		=\mathcal{F}[f](\xi)
		=\int_{\mathbb{R}^d}f(x)e^{-ix\cdot\xi}dx\quad (\xi\in\mathbb{R}^d).\nonumber
	\end{eqnarray}
	The inverse Fourier transform of $f$ is denoted by $\mathcal{F}^{- 1}[f]$: 
	\begin{eqnarray}
		\mathcal{F}^{- 1}[f](x)
		=(2\pi)^{-d}\int_{\mathbb{R}^d}f(\xi)e^{i\xi\cdot x}d\xi\quad(x\in\mathbb{R}^d).\nonumber
	\end{eqnarray}

	Let $k$ be a nonnegative integer and let 
	$r_{1}$ and $r_{\infty}$ be positive constants  satisfying $r_{1}<r_{\infty}.$  We denote by $H_{(\infty)}^{k}$ the set of all $u\in H^k$ satisfying $\mbox{supp }\hat{u}\subset\{|\xi|\geq r_{1}\}$, 
	and by $L_{(1)}^{2}$ the set of all $u\in L^2$ satisfying $\mbox{supp }\hat{f}\subset\{|\xi|\leq r_{\infty}\}$. 
	Note that $H^k \cap L_{(1)}^2=L_{(1)}^2$ for any nonnegative integer $k$. (Cf., Lemma \ref{lemP_1} {\rm (ii)} bellow.)

	Let  $k$ and $\ell$ be nonnegative integers. We define the spaces $H_{\ell}^{k}$ and $H_{(\infty),\ell}^{k}$ by
	\begin{eqnarray}
		H_{\ell}^{k}=\{u\in H^k;\|u\|_{H_{\ell}^{k}}< +\infty\},\nonumber
	\end{eqnarray}
	where
	\begin{eqnarray}
		\|u\|_{H_{\ell}^{k}}&=&\left(\sum_{j=0}^{\ell}|u|_{H^{k}_{j}}^{2}\right)^{\frac{1}{2}},\nonumber\\
		|u|_{H^{k}_{\ell}}&=&\left(\sum_{|\alpha|\leq k} \parallel |x|^{\ell}\partial^{\alpha}_{x} u\parallel _{L^2}^{2} \right)^{\frac{1}{2}},\nonumber
	\end{eqnarray}
	and
	\begin{eqnarray}
		H_{(\infty),\ell}^{k}=\{u\in H_{(\infty)}^k;\|u\|_{H_{\ell}^{k}}< +\infty\}.\nonumber
	\end{eqnarray}
	
	Let $\ell $ be a nonnegative integer. We denote $L^2_{(1),\ell}$ by
	$$
	L^2_{(1),\ell}=\{f\in L^2_{\ell};f\in L^2_{(1)}\}.
	$$
	
	For $-\infty\leq a<b\leq \infty$,
	we denote by $C^k([a,b];X)$ the set of all $C^k$ functions on $[a,b]$ with values in $X$. 
	We denote the Bochner space on $(a,b)$ by $L^p(a,b;X)$ 
	and the $L^2$-Bochner-Sobolev space of order $k$ by $H^k(a,b;X)$.

	The space ${\scr X}_{(1)}$ is defined by 
	\begin{eqnarray}
		{\scr X}_{(1)}=\{a\in L^{\infty}_{d-2},\nabla {a}\in H^{1};\mbox{supp }\hat{a}\subset\{|\xi|\leq r_{\infty}\}, \|a\|_{{\scr X}_{(1)}}<+\infty\},\nonumber
	\end{eqnarray}
	where
	\begin{eqnarray*}
		&&\|a\|_{{\scr X}_{(1)}}:=\|a\|_{{\scr X}_{(1),L^{\infty}}}+\|a\|_{{\scr X}_{(1),L^{2}}},\\
		&&\|a\|_{{\scr X}_{(1),L^{\infty}}}:=\sum_{j=0}^{1}\|(1+|x|)^{d-2+j}{\nabla}^{j}a \|_{L^\infty},\,
		\|a\|_{{\scr X}_{(1),L^{2}}}:=\sum_{j=1}^{2}\|(1+|x|)^{j-1}{\nabla}^{j} a\|_{L^2}.
	\end{eqnarray*}
	
	We define the space ${\scr Y}_{(1)}$ by 
	\begin{eqnarray}
		{\scr Y}_{(1)}=\{v \in L^{\infty}_{d-1} \cap L^2,\nabla v \in L^{2}_1;\mbox{supp }\hat{v}\subset\{|\xi|\leq r_{\infty}\},\|v\|_{{\scr Y}_{(1)}}<+\infty\},\nonumber
	\end{eqnarray}
	where
	\begin{eqnarray*}
		&&\|v\|_{{\scr Y}_{(1)}}:=\|v\|_{{\scr Y}_{(1),L^{\infty}}}+
		\|v\|_{{\scr Y}_{(1),L^{2}}},\\
		&&\|v\|_{{\scr Y}_{(1),L^{\infty}}}:=\|(1+|x|)^{d-1}v\|_{L^\infty},\,
		\|v\|_{{\scr Y}_{(1),L^{2}}}:=\|v\|_{L^2}+\|(1+|x|)\nabla v\|_{L^2}.
	\end{eqnarray*}

	The space $\scr{Z}_{(1)}(a,b)$ is defined by 
	\begin{eqnarray*}
		\scr{Z}_{(1)}(a,b)=C^1([a,b];{\scr X}_{(1)} \times {\scr Y}_{(1)}).
	\end{eqnarray*}
	
	\vspace{2ex}
	
	Let $\ell$ be a nonnegative integer and  let $s$ be a nonnegative integer satisfying $s\geq\left[\frac{d}{2}\right]+2.$ For $k=s-1,s$,
	the space ${\scr Z}^{k}_{(\infty),\ell}(a,b)$ is defined by
	\begin{align*}
		{\scr Z}^{k}_{(\infty),\ell}(a,b)&= \bigl[C([a,b];H_{(\infty),\ell}^{k})\cap C^{1}(a,b;H_{(\infty)}^{k-1})\big]\nonumber\\
		&\qquad\times   C([a,b];H_{(\infty),\ell}^{k}).
	\end{align*}
	
	\vspace{2ex}
	
	\indent Let $s$ be a nonnegative integer satisfying $s\geq\left[\frac{d}{2}\right]+2.$ and let $k=s-1,s$. The space $X^{k}(a,b)$ is defined by
	\begin{align*}
		X^{k}(a,b)\quad \\
		=\bigl\{\{u_{1},u_{\infty}\};\,
		& u_{1}\in \scr{Z}_{(1)}(a,b), u_{\infty}\in {\scr Z}^{k}_{(\infty),d-1}(a,b),\\
		& \del_t \phi_1 \in C([a,b];L^2), \, \del_t \phi_\infty \in C([a,b];H^{k-1}), \,  u_{j} =\trans{(a_{j},v_{j})}\, (j=1,\infty)\bigr\},
	\end{align*}
	equipped with the norm
	\begin{align*}
		\|\{u_{1},u_{\infty}\}\|_{X^{k}(a,b)}=&
		\|u_{1}\|_{\scr{Z}_{(1)}(a,b)}
		+\|u_{\infty}\|_{{\scr Z}^{k}_{(\infty),d-1}(a,b)}\\
		&+
		\|\partial_{t}\phi_1\|_{C([a,b];L^{2})}+\|\partial_{t}\phi_\infty\|_{C([a,b];H^{k-1})}.
	\end{align*}
	
	We also introduce function spaces of $T$-periodic functions in $t$. 
	We denote by $C_{per}(\mathbb{R};X)$ the set of all $T$-periodic continuous functions 
	with values in $X$ equipped with the norm $\|\cdot\|_{C([0,T];X)}$; 
	and we denote by $L^2_{per}(\mathbb{R};X)$ the set of all $T$-periodic locally square integrable functions 
	with values in $X$ equipped with the norm $\|\cdot\|_{L^2(0,T;X)}$. 
	Similarly, $H^1_{per}(\mathbb{R};X)$ and $X_{per}^k(\mathbb{R})$, and so on, are defined. 
	
	For a bounded linear operator $L$ on  a Banach space $X$,  we denote by $r_{X}(L)$ the spectral radius of $L$.

	For operators $L_1$ and $L_2,$ $[L_1,L_2]$ denotes the commutator of $L_1$ and $L_2$: 
	\begin{eqnarray}
		[L_1,L_2]f=L_1(L_2 f)-L_2(L_1 f).\nonumber
	\end{eqnarray}

	We next  state some lemmas which will be used in the proof of the main results. 
	
	\vspace{2ex}
	
	We begin with the well-known Sobolev type inequality.
	
	\begin{lem}\label{lem2.1.} Let $d\geq 3$ and let $s\geq \left[\frac{d}{2}\right]+1.$ Then there holds the inequality 
		\begin{eqnarray}
			\|f\|_{L^{\infty}}\leq C\|\nabla f\|_{H^{s-1}}\nonumber
		\end{eqnarray}
		for $f\in H^{s}.$
	\end{lem}
	
	\vspace{2ex}
	
	We next state some inequalities concerned with composite functions.
	
	\begin{lem}\label{lem2.2.} 
		Assume $d\geq 2$ and let $s$ be an integer satisfying $s\geq \left[\frac{d}{2}\right]+1$.  
		Let $s_{j}$ and $\mu_{j}$ ($j=1,\cdots,\ell$) 
		satisfy $0\leq |\mu_{j}|\leq s_{j}\leq  s+|\mu_{j}|$, 
		$\mu=\mu_{1}+\cdots +\mu_{\ell}$, 
		$s=s_{1}+\cdots+s_{\ell}\leq(\ell-1)s+|\mu|$. 
		Then there holds
		\begin{eqnarray}
			\parallel \partial^{\mu_{1}}_{x}f_{1}\cdots \partial^{\mu_{\ell}}_{x}f_{\ell}\parallel_{L^{2}}
			\leq C\prod_{1\leq j\leq \ell}\parallel f_{j}\parallel_{H^{s_{j}}}.\nonumber
		\end{eqnarray}
	\end{lem}
	
	See, e.g., \cite{Kagei-Kobayashi}, for the proof of Lemma $\ref{lem2.2.}$.
	
	\vspace{2ex}
	\begin{lem}{\rm \cite[Lemma 2.1]{choi}}\label{lem2.3.} 
		Let $f, g \in H^k \cap L^\infty$ with $\nabla^k f \in L^\infty$. Then it holds that for $k \in \mathbb{N}$ 
		$$
		\|\nabla^k (fg) - f \nabla^k g \|_{L^2} \leq C(\|\nabla f\|_{L^\infty}\|\nabla^{k-1}g\|_{L^2}+\|\nabla^k f\|_{L^2}\|g\|_{L^\infty}).
		$$
		
		%Let $d\geq 2$ and let $s$ be an integer satisfying $s\geq \left[\frac{d}{2}\right]+1$. 
		%Suppose that $F$ is a smooth function on $I$, 
		%where $I$ is a compact interval of $\mathbb{R}$.  
		%Then for a multi-index $\alpha$ with $1\leq |\alpha|\leq s$, 
		%there hold the estimates 
		%$$
		%\|[\partial^{\alpha}_{x},F(f_{1})]f_{2}\|_{L^2}
		%\leq 
		%C\|F\|_{C^{|\alpha|}(I)}
		%\left\{1+\|\nabla f_{1}\|_{s-1}^{|\alpha|-1}\right\}
		%\|\nabla f_{1}\|_{H^{s-1}}\|f_2 \|_{H^{|\alpha|}}
		%$$ 
		%for $f_1\in H^s$ with $f_1(x)\in I$ for all $x\in \mathbb{R}^d$ 
		%and $f_2\in H^{|\alpha|}$; and 
		%$$
		%\|[\partial^{\alpha}_{x},F(f_{1})]f_{2}\|_{L^2}
		%\leq 
		%C\|F\|_{C^{|\alpha|}(I)}
		%\left\{1+\|\nabla f_{1}\|_{s-1}^{|\alpha|-1}\right\}
		%\|\nabla f_{1}\|_{H^{s}}\|f_2 \|_{H^{|\alpha|-1}}. 
		%$$
		%for $f_1\in H^{s+1}$ with $f_1(x)\in I$ for all $x\in \mathbb{R}^d$ 
		%and $f_2\in H^{|\alpha|-1}$.  
	\end{lem}

	%------------------------------------------------------------
	%     Section 2.   
	
	%------------------------------------------------------------

	\section{Main results}\label{S3}
	In this section, we introduce the main results on the existence and stability of a time-periodic solution for system \eqref{Main1}. 
	
	Without loss of generality, we assume $\rho_*=1$ and $P'(\rho_*)=1$.  Denote
	\begin{align*}
		\rho=1+\phi, \quad a=\log(1+\phi).
	\end{align*}
	Then system \eqref{Main1} is reformulated as follows,
	\begin{equation}\label{3.1}
		\left\{
		\begin{aligned}
			&\partial_ta+\text{div}v+v\cdot \nabla a=0,\\
			&\partial_tv+\nabla a+v+g_3(\phi)\nabla a+v\cdot\nabla v=g,
		\end{aligned}
		\right.
	\end{equation}
	where the function $g_3(\phi)$ is given by  
	\begin{align*}
		g_3(\phi)=P'(1+\phi)-P'(1)=P'(e^a)-1.
	\end{align*}
	We see that \eqref{Main1} is rewritten as the following symmetric hyperbolic system.
	\eqn{ns2}
	$$
	\delt u+Au=-B[u] u+G(g),
	$$
	where 
	\begin{eqnarray}\label{(a)}
		A=\begin{pmatrix}
			0 &\mathrm{div}\\
			\nabla & I \\
		\end{pmatrix},
	\end{eqnarray}
	\begin{eqnarray}\label{(b)}
		B[\tilde{u}]u=\begin{pmatrix}
			\tilde{v}\cdot\nabla a\\
			\tilde{v}\cdot\nabla v+g_3(\tilde\phi)\nabla a \\
		\end{pmatrix}
		\mbox{ for }u=\trans(a, v),\,\tilde{u}=\trans(\tilde{a},\tilde{v} ),
	\end{eqnarray}
	and
	\begin{eqnarray}
		G(g)&=&\begin{pmatrix}
			0\\
			g\\
		\end{pmatrix}.
		\label{(c)}
	\end{eqnarray}

	\vspace{2ex}
	
	We next introduce operators which decompose a function into 
	its low and high frequency parts. 
	Operators $P_{1}$ and $P_{\infty}$ on $L^{2}$ are defined by
	\begin{eqnarray}
		P_{j}f=\mathcal{F}^{- 1}\big(\hat\chi_{j}\mathcal{F}[f]\big)\quad(f\in L^2 , j=1,\infty),\nonumber
	\end{eqnarray}
	where 
	\begin{align*}
		&\hat{\chi}_{j}(\xi)\in C^{\infty}(\mathbb{R}^{d}) 
		\quad(j=1,\infty),\quad 0\leq \hat\chi_{j}\leq 1 \quad(j=1,\infty),\\
		&\hat\chi_{1}(\xi)=\left\{
		\begin{array}{l}
			1\quad (|\xi|\leq r_{1}),\\
			0\quad (|\xi|\geq r_{\infty}),
		\end{array}
		\right.
		\\
		&\hat\chi_{\infty}(\xi)=1-\hat\chi_{1}(\xi),\\
		&0<r_{1}<r_{\infty}.
	\end{align*}
	We fix $0<r_{1}<r_{\infty}<\frac{1}{2}$ in such a way that the estimate 
	\eqref{Pi} in Lemma \ref{eigenvalue} below holds for $|\xi|\leq r_{\infty}$.

	\vspace{2ex}
	
	Our result on the existence of a time periodic solution is stated as follows.

	\begin{thm}\label{Theorem 3.1} Let $d\geq 3$ and $s$ be an integer satisfying $s\geq \left[\frac{d}{2}\right]+2.$ Assume that $g(x,t)$ satisfies 
		\eqref{assumption-g} 
		and $g(x,t)\in C_{per}(\mathbb{R};L^{1}\cap L^{\infty}_{d}) \cap L^{2}_{per}(\mathbb{R};H^{s-1}_{d-1})$. Set 
		\begin{eqnarray*}
			[g]_{s}=\|g\|_{C([0,T];L^{1}\cap  L^{\infty}_{d})}+\|g\|_{L^{2}(0,T;H^{s-1}_{d-1})}.
		\end{eqnarray*}
		
		Then there exist constants $\delta>0$ and $C>0$ such that if $[g]_{s}\leq \delta,$ then the system $\eqref{3.1}$ has a time-periodic solution $u=u_{1}+u_{\infty}$ satisfying $\{u_{1},u_{\infty}\}\in X^{s}_{per}(\mathbb{R}^{d})$ 
		with $\|\{u_1,u_\infty\}\|_{X^s(0,T)}\leq C[g]_{s}$. 
		Furthermore, the uniqueness of time periodic solutions of \eqref{ns2} holds 
		in the class 
		$\{u=\trans(a,v);
		\{P_1u,P_\infty u\}\in X^{s}_{per}(\mathbb{R}), \, 
		\|\{u_1,u_\infty\}\|_{X^s(0,T)}\leq C\delta \}$. 
	\end{thm}
	
	\vspace{2ex}
	{\rm } \begin{rem}\label{note-a} The time periodic solution obtained in Theorem \ref{Theorem 3.1} 
		satisfies that  
		$$
		\trans(\nabla a,v) \in C_{per}(\mathbb{R};H^{s-1}_{d-1} \times H^{s}_{d-1}),  \ \ 
		\|\nabla a\|_{C_{per}(\mathbb{R};H^{s-1}_{d-1} )} \leq  C\| \nabla \phi \|_{C_{per}(\mathbb{R};H^{s-1}_{d-1})}
		$$
		with $a \in C_{per}(\mathbb{R};L^\infty)$, where 
		$a=\log (1+\phi)$, $\phi=\phi_1 +\phi_\infty$. 
	\end{rem}
	
	\vspace{2ex}
	
	We next consider the stability of the time-periodic solution obtained in Theorem $\ref{Theorem 3.1}$.
	
	\vspace{2ex}
	Let $\trans(a_{per},v_{per})$ be the periodic solution given in Theorem $\ref{Theorem 3.1}$. 
	We denote the perturbation by ${u}=\trans{({\psi},{w})},$ where ${\psi}=a-a_{per},{w}=v-v_{per}$. 
	Substituting $a= {\psi}+a_{per}$ and $v= {w}+v_{per}$ into \eqref{3.1}, we see that the perturbation $ {u}=\trans{( {\psi}, {w})}$ is governed by 
	\begin{eqnarray}
		\left\{
		\begin{array}{ll}
			\partial_{t} {\psi}+\div w + v_{per}\cdot\nabla {\psi} + w \cdot \nabla a_{per}= {f^1(u)},
			\\
			\partial_{t} {w}+ w + \nabla \psi +v_{per}\cdot\nabla {w}+ {w}\cdot\nabla v_{per}
			\\
			\quad
			+g_4 (a_{per}, \psi) \nabla \psi + g_5(a_{per}, \psi) \nabla a_{per}
			= f^2 (u),
			\label{stability}
		\end{array}
		\right.
	\end{eqnarray}
	where 
	\begin{align*}
		f^1(u) 
		&=-w \cdot \nabla \psi,\\
		f^2 (u)
		&=
		- {w}\cdot\nabla {w},\\ 
		%- g_3(\psi) \nabla \psi, \\
		g_4 (a_{per}, \psi)& = P'(e^{a_{per}+\psi})-1, \\
		g_5(a_{per}, \psi) &= P'(e^{\psi+a_{per}})- P'(e^{a_{per}}) \\
		&= e^{a_{per}} (e^\psi-1) \int_0^1 P''(a_{per}+ \theta (e^{\psi+a_{per}}))d\theta.\BLACK  
	\end{align*}

	We consider the initial value problem for \eqref{stability} under the initial condition 
	\begin{eqnarray}
		u|_{t=0}=u_{0}=\trans{(\psi_{0},w_{0})}.\label{stability2}
	\end{eqnarray}
	
	Our result on the stability of the time-periodic solution is stated as follows.
	
	\begin{thm}\label{Theorem 3.2} Let $d\geq 3$ and $s$ be an integer satisfying $s\geq \left[\frac{d}{2}\right]+2.$ Assume that $g(x,t)$ satisfies 
		\eqref{assumption-g} 
		and $g(x,t)\in C_{per}(\mathbb{R};L^{1} \cap L^{\infty}_{d}) \cap L^{2}_{per}(\mathbb{R};H^{s+1}_{d-1})$. 
		Let $\trans(a_{per},v_{per})$ be the time-periodic solution obtained in Theorem $\ref{Theorem 3.1}$,    %satisfying 
		%$$
		%\trans(\nabla a_{per},v_{per}) \in C_{per}(\mathbb{R};H^{s}_{d-1} \times H^{s+1}_{d-1}) 
		%$$
		%with $a_{per} \in C_{per}(\mathbb{R};L^\infty)$ 
		and let $u_{0}\in H^{s}$.
		Then there exist constants $\epsilon_{1}>0$ and $\epsilon_{2}>0$ such that if 
		\begin{eqnarray}
			[g]_{s+1}+ \|u_{0}\|_{H^{s}}\leq \delta_0 \leq 1,\nonumber
		\end{eqnarray}
		there exists a unique global solution $u=\trans{(\psi,w)}$ of \eqref{stability}-\eqref{stability2} satisfying 
		\begin{align}\label{energy-stability}
			&u\in C([0,\infty);H^{s}),\\
			&\|(\psi,w)(t)\|_{H^{s}}^2+\int_{0}^{t}(\|\nabla \psi(\tau)\|_{H^{s-1}}^{2}+\|w(\tau)\|_{H^s}^{2})d\tau \leq C\|u_{0}\|_{H^{s}}^{2} \ \ (t\in [0,\infty)),\\
			&\|u(t)\|_{L^{\infty}}\rightarrow 0 \ \ (t\rightarrow \infty).
		\end{align}
	\end{thm}
	
	\vspace{2ex}

	\section{Reformulation of the problem}\label{S4}
	
	In this section, we reformulate problem \eqref{3.1}.
	Inspired by \cite{Kagei-Tsuda}, to solve the time periodic problem for \eqref{3.1}, we decompose $u$ into a low frequency part $u_{1}$ and 
	a high frequency part $u_{\infty}$, and then, we rewrite the problem 
	into a system of equations for $u_{1}$ and $u_{\infty}$.
	
	\vspace{2ex}
	First of all, we set 
	\begin{eqnarray*}
		u_{1}=P_{1}u,\quad u_{\infty}=P_{\infty}u.\nonumber
	\end{eqnarray*}
	Applying the operators $P_{1}$ and $P_{\infty}$ to \eqref{ns2}, we obtain,
	\begin{eqnarray}
		\partial_{t}u_{1}+Au_{1}=F_{1}(u_{1}+u_{\infty},g),\label{eq:(4.2.2)}\\
		\partial_{t}u_{\infty}+Au_{\infty}+P_\infty(B[u_1+u_\infty]u_\infty)
		=F_{\infty}(u_{1}+u_{\infty},g)\label{eq:(4.2.3)}.
	\end{eqnarray}
	Here
	\begin{eqnarray*}
		F_1(u_{1}+u_{\infty},g)
		&=&
		P_1[-B[u_{1}+u_{\infty}] (u_{1}+u_{\infty})+G(g)],
		\\
		F_\infty(u_{1}+u_{\infty},g)
		&=&
		P_\infty[- B[u_{1}+u_{\infty}] u_1+G(g)].
	\end{eqnarray*}
	More precisely, one has 
	\begin{align}\label{010301}
		F_1(u,g)=P_1[-B[u]u+G(g)],\quad F_\infty(u,g)=P_\infty[-B[u]u_1+G(g)].
	\end{align}
	Suppose that $(\ref{eq:(4.2.2)})$ and $(\ref{eq:(4.2.3)})$ are satisfied by some functions $u_{1}$ and $u_{\infty}$. 
	Then by adding $(\ref{eq:(4.2.2)})$ to $(\ref{eq:(4.2.3)})$, we obtain
	\begin{eqnarray*}
		\partial_{t}(u_{1}+u_{\infty})+A(u_{1}+u_{\infty})
		&=&
		-P_\infty(B[u_1+u_\infty]u_\infty)
		+(P_{1}+P_{\infty})F(u_{1}+u_{\infty},g)
		\\
		&=&
		-B[u_{1}+u_{\infty}](u_{1}+u_{\infty})
		+G(g).
	\end{eqnarray*}
	Set $u=u_{1}+u_{\infty}$, then we have
	\begin{eqnarray}
		\partial_{t}u+Au+B[u]u=G(g).\nonumber
	\end{eqnarray}
	Consequently, if we show the existence of  a pair of functions 
	$\{u_{1}, u_{\infty}\}$ satisfying 
	$(\ref{eq:(4.2.2)})$-$(\ref{eq:(4.2.3)})$, 
	then we can obtain a solution $u$ of \eqref{ns2}.

	\vspace{2ex}
	
	For later use,  we  prepare some inequalities for the low frequency part. 
	We first derive some properties of $P_1$. 
	
	\begin{lem}\label{lemP_1}
		{\rm (i)} 
		Let $k$ be a nonnegative integer. 
		Then $P_{1}$ is a bounded linear operator from $L^2$ to $H^{k}$. 
		In fact, it holds that 
		\begin{eqnarray}
			\|\nabla^{k}P_{1}f\|_{L^2}\leq C\|f\|_{L^2}\qquad (f\in L^{2}).\nonumber
		\end{eqnarray}
		As a result, for any $2\leq p\leq \infty$, $P_1$ is bounded from $L^2$ to $L^p$. 
		
		\vspace{1ex}
		{\rm (ii)} 
		Let $k$ be a nonnegative integer. 
		Then there hold the estimates 
		$$
		\|\nabla^k f_1\|_{L^2}+\|f_1\|_{L^p}
		\leq 
		C\|f_1\|_{L^2} 
		\ \ 
		(f\in L^2_{(1)}), 
		$$
		where $2\leq p\leq \infty$.
		
	\end{lem}
	
	\vspace{2ex}
	
	The proofs of estimates {\rm (i)} and {\rm (ii)} are given in \cite[Lemma 4.3]{Kagei-Tsuda}.

	\vspace{2ex}
	
	The following inequality is concerned with the estimates of the weighted $L^{p}$ norm for the low frequency part.
	
	\begin{lem}\label{lemP_1 to weightedLinfty} 
		Let $\chi$ be a function which belongs to the Schwartz space 
		on $\mathbb{R}^d$. 
		Then for a nonnegative integer $\ell$ and $1\leq p\leq \infty$, there holds 
		\begin{eqnarray}
			\| |x|^{\ell}(\chi\ast f)\|_{L^p}
			\leq 
			C\{\| |x|^{\ell}\chi\|_{L^1}\|f\|_{L^p}
			+\|\chi\|_{L^1}\| |x|^{\ell}f\|_{L^p}\}\qquad (f\in L^{p}_{\ell}).\nonumber
		\end{eqnarray}
		Here $C$ is a positive constant depending only on $\ell$.
	\end{lem}
	
	\noindent
	\textbf{Proof.} 
	Let $\chi$ be a function which belongs to the Schwartz space 
	on $\mathbb{R}^d$. 
	Then
	\begin{eqnarray}
		| |x|^{\ell}(\chi\ast f)|&\leq& |x|^{\ell}\int_{\mathbb{R}^d}|\chi(x-y)f(y)|dy
		\nonumber\\
		&\leq& 
		C\int_{\mathbb{R}^d}|x-y|^{\ell}|\chi(x-y)||f(y)|dy+C\int_{\mathbb{R}^d}|\chi(x-y)||y|^{\ell}|f(y)|dy.\nonumber
	\end{eqnarray}
	Therefore, the Young inequality gives 
	\begin{eqnarray}
		\| |x|^{\ell}(\chi\ast f)\|_{L^p}\leq C\{\| |x|^{\ell}\chi\|_{L^1}\|f\|_{L^p}+\|\chi\|_{L^1}\| |x|^{\ell}f\|_{L^p}\}\qquad (f\in L^{p}_{\ell}).\nonumber
	\end{eqnarray}
	This completes the proof.$\hfill\square$\\
	
	\vspace{2ex}
	
	Applying Lemma \ref{lemP_1 to weightedLinfty}, we have the following inequality for the weighted $L^{p}$ norm of the low frequency part.
	\begin{lem}\label{lemP_1 weightedLinfty}
		Let $k$ and $\ell$ be nonnegative integers and let $1\leq p\leq \infty$. 
		Then there holds the estimate
		$$
		\||x|^{\ell}\nabla^{k}f_1\|_{L^p}\leq C\||x|^{\ell}f_1\|_{L^p}
		\ \ 
		(f\in L^2_{(1)}\cap L^p_{\ell}).
		$$
	\end{lem}
	
	\vspace{2ex}
	\noindent
	\textbf{Proof.} 
	We define a cut-off function $\chi_0=\mathcal{F}^{-1}\big(\hat{\chi}_0\big)$ with $\hat{\chi}_0$ satisfying 
	\begin{eqnarray}
		\hat{\chi}_0\in C^{\infty}(\mathbb{R}^{d}),
		\ \ 0\leq \hat{\chi}_0\leq 1, 
		\ \ \hat{\chi}_0=1 \ \ \mbox{on} \ \ \{|\xi| \leq  r_{\infty}\}
		\ \ \mbox{and} \ \  \supp\hat{\chi}_0 \subset \{|\xi| \leq 2 r_{\infty}\}.\label{subcutofflowpart}
	\end{eqnarray}
	We note that 
	\begin{eqnarray}
		|\partial_{x}^{\alpha}\chi_0(x)| \leq C(1+|x|)^{-(d+|\alpha|)}\ \ \mbox{for} \ \  |\alpha| \geq 0.\label{E_0estimate}
	\end{eqnarray}
	
	Since $f_{1}\in L^2_{(1)}$, we see that $\nabla^{k}f_1=(\nabla^{k}\chi_0 )\ast f_1$ $(k\geq 0)$. 
	Therefore, by Lemma \ref{lemP_1 to weightedLinfty}, we obtain the desired estimate. 
	This completes the proof.$\hfill\square$\\
	
	\vspace{2ex}

	%	Since $d \geq 3$, applying the Hardy inequality and Lemma $\ref{lemP_1 weightedLinfty}$, we have the following inequality for the weighted $L^2$ norm of the low frequency part. 
	%\begin{lem}\label{lowpartL2esiamatelemma}
	%	Let $a \in {\scr X}_{(1)}$ and $v_1 \in {\scr Y}_{(1)}$. Then, it holds that 
	%	$$
	%		\|P_1(a v_1)\|_{{\scr Y}_{(1),L^2}} \leq C\|a\|_{L^{\infty}_{d-1}}\|\nabla v_1\|_{L^2}.
	%		$$
	%		Here $C>0$ is a constant depending only on  $d$.
	%	\end{lem}

%	\vspace{2ex}

%	\noindent{\bf Proof.} By Lemma $\ref{lemP_1 weightedLinfty}$, we see that 
%	\begin{eqnarray}
	%		\|P_1(a v_1)\|_{{\scr Y}_{(1),L^2}}\leq C\|a v_1\|_{L^2_1}.\label{Hardy1}
	%	\end{eqnarray}
%	Since $d \geq 3$, by the Hardy inequality, we find that
%	\begin{eqnarray}
	%		\|a v_1\|_{L^2_1}\leq C\|a\|_{L^{\infty}_{d-1}}\|\nabla v_1\|_{L^2}.\label{Hardy2}
	%	\end{eqnarray}
%	By \eqref{Hardy1} and \eqref{Hardy2}, we obtain the desired estimate. This completes the proof. $\hfill\square$

%	\vspace{2ex}

We look for a time periodic solution $u$ for the system $(\ref{eq:(4.2.2)})$ and $(\ref{eq:(4.2.3)})$. To solve the time periodic problem, we introduce solution operators for the following linear problems:
\begin{eqnarray}
	\left\{
	\begin{array}{ll}
		\partial_{t}u_{1}+Au_{1}=F_{1},\\
		u_{1}|_{t=0}=u_{01},
	\end{array}
	\right.\label{eq:(1)}
\end{eqnarray}
and
\begin{eqnarray}
	\left\{
	\begin{array}{ll}
		\partial_{t}u_{\infty}+Au_{\infty}+P_\infty (B[\tilde u]u_{\infty})=F_{\infty},\\
		u_{\infty}|_{t=0}=u_{0\infty},
	\end{array}
	\right.\label{eq:(2)}
\end{eqnarray}
where $\tilde{u}=\trans{(\tilde{a},\tilde{v})},u_{01},u_{0\infty},F_{1}$ and $F_{\infty}$ are given functions.
\vspace{2ex}

To formulate the time periodic problem, 
we denote by $S_{1}(t)$ the solution operator for $(\ref{eq:(1)})$ with $F_{1}=0$, 
and by ${\scr S}_{1}(t)$ the solution operator for $(\ref{eq:(1)})$ with $u_{01}=0$.
We also denote by $S_{{\infty},\tilde{u}}(t)$ the solution operator for $(\ref{eq:(2)})$ with $F_{\infty}=0$ 
and by $\scr{S}_{\infty,\tilde{u}}(t)$ the solution operator for $(\ref{eq:(2)})$ with $u_{0\infty}=0$. 
(The precise definition of these operators will be given later.)

Our main goal is to look for a $\{u_{1},u_{\infty}\}$ satisfying 
\begin{eqnarray}
	\left\{
	\begin{array}{ll}
		u_{1}(t)=S_{1}(t)u_{01}+{\scr S}_{1}(t)[F_{1}(u,g)],\label{eq:(3)}\\
		u_{\infty}(t)=S_{\infty,u}(t)u_{0\infty}+\scr{S}_{\infty,u}(t)[F_{\infty}(u,g)],
	\end{array}
	\right.
\end{eqnarray}
where
\begin{eqnarray}
	\left\{
	\begin{array}{lll}
		u_{01}=(I-S_{1}(T))^{-1}{\scr S}_{1}(T)[F_{1}(u,g)],\label{eq:(4)}\\
		u_{0\infty}=(I-S_{\infty,u}(T))^{-1}\scr{S}_{\infty,u}(T)[F_{\infty}(u,g)],\\
	\end{array}
	\right.
\end{eqnarray}
$u=\trans(a,v)$ is a function given by $u_{1}=\trans(a_1,v_1)$ and $u_{\infty}=\trans(a_\infty,v_{\infty})$ through the relation 
\begin{align*}
	a=a_1+a_\infty,\quad  v=v_1+v_{\infty}.
\end{align*}

Therefore if $(I-S_{1}(T))$ and $(I-S_{\infty,u}(T))$ are invertible in a suitable sense, 
then one obtains $(\ref{eq:(3)})$-$(\ref{eq:(4)})$. 
So, to obtain a  $T$-time periodic solution of \eqref{eq:(4.2.3)} and \eqref{eq:(4.2.2)}, we look for a pair of functions $\{u_{1},u_{\infty}\}$ 
satisfying \eqref{eq:(3)}-\eqref{eq:(4)}. 
We will investigate the solution operators $S_{1}(t)$ and ${\scr S}_{1}(t)$ in section \ref{S5}; and we state some properties of $S_{{\infty},{u}}(t)$ 
and $\scr{S}_{\infty,u}(t)$ in section \ref{S6}.

\vspace{2ex}
In the remaining of this section we introduce some lemmas which will be used in the proof of Theorem \ref{Theorem 3.1}. 

\vspace{2ex}
For the analysis of the low frequency part, we will use the following well-known inequalities.

\vspace{2ex}
\begin{lem}\label{Kuroda}
	Let $\alpha$ and $\beta$ be positive numbers satisfying $d<\alpha+\beta$. 
	Then there holds the following estimate.
	\begin{eqnarray*}
		\int_{\mathbb{R}^{d}}(1+|x-y|)^{-\alpha}(1+|y|^{2})^{-\frac{\beta}{2}}dy
		\leq C
		\left \{
		\begin{array}{ll}
			(1+|x|)^{d-(\alpha+\beta)} \ \ (\max \{\alpha,\beta\}<d),\\
			(1+|x|)^{-\min\{\alpha,\beta\}}\log |x| \ \ (\max \{\alpha,\beta\}=d),\\
			(1+|x|)^{-\min\{\alpha,\beta\}} \ \ (\max \{\alpha,\beta\}>d)
		\end{array}
		\right.
	\end{eqnarray*}
	for $x\in \mathbb{R}^{d}$
\end{lem}

\vspace{2ex}

The following lemma is related to the estimates for the integral kernels which will appear in the analysis of  the low frequency part.

\vspace{2ex}
\begin{lem}\label{sekibunkaku}
	Let $\ell$ be a nonnegative integer and 
	let $E(x)={\scr F}^{-1}\hat{\Phi}_{\ell}$ $(x\in \mathbb{R}^{d})$, 
	where $\hat{\Phi}_{\ell}\in C^{\infty}(\mathbb{R}^d-\{0\})$ is a function satisfying 
	\begin{eqnarray*}
		\partial_{\xi}^{\alpha}\hat{\Phi}_{\ell}\in L^{1}\quad (|\alpha|\leq d-3+\ell), \\
		|\partial_{\xi}^{\beta}\hat{\Phi}_{\ell}| \leq C|\xi|^{-2-|\beta|+\ell}\quad (\xi \neq 0,\,|\beta|\geq 0).
		%&&\Longrightarrow |E(x)|\leq C|x|^{-(n-2+\ell)} \ \ \mbox{for} \ \ \ell \geq 0.
	\end{eqnarray*}
	Then the following estimate holds for $x \not=0$.
	\begin{eqnarray*}
		|E(x)|\leq C|x|^{-(d-2+\ell)}.
	\end{eqnarray*}
\end{lem}

\vspace{2ex}
Lemma \ref{sekibunkaku} easily follows from a direct application of \cite[Theorem 2.3]{Shibata-Shimizu}; and we omit the proof.

\vspace{2ex}

We will also use the following lemma for the analysis of the low frequency part.

\vspace{2ex}

\begin{lem}\label{convolution}
	{\rm (i)}
	Let $E(x)$ $(x\in \mathbb{R}^{d})$ be a scalar function satisfying 
	\begin{eqnarray}
		|\partial^{\alpha}_{x}E(x)|\leq \frac{C}{(1+|x|)^{|\alpha|+d-2}} \ \ (|\alpha|=0,1,2)\label{Shibata-Tanaka katei}.
	\end{eqnarray}
	Assume that $f$ is a scalar function satisfying $\|f\|_{L^{\infty}_d \cap L^1}<\infty$.  
	Then there holds the following estimate for $|\alpha|=0,1$.
	\begin{eqnarray*}
		|[\partial^{\alpha}_{x}E\ast f](x)|\leq \frac{C}{(1+|x|)^{|\alpha|+d-2}}\|f\|_{L^{\infty}_d \cap L^1}.
	\end{eqnarray*}
	
	\vspace{2ex}
	{\rm (ii)} 
	Let $E(x)$ $(x\in \mathbb{R}^{d})$ be a scalar function satisfying $(\ref{Shibata-Tanaka katei})$. 
	Assume that $f$ is a scalar function of the form: $f=\partial_{x_j}f_1$ for some $1 \leq j \leq d$ satisfying $\|\partial_{x_j}f_1\|_{L^{\infty}_d}+\|f_1\|_{L^{\infty}_{d-1}}<\infty$. 
	Then there holds the following estimate for $|\alpha|=0,1$.
	\begin{eqnarray*}
		|[\partial^{\alpha}_{x}E\ast f](x)|\leq \frac{C}{(1+|x|)^{|\alpha|+d-2}}(\|\partial_{x_j}f_1\|_{L^{\infty}_d}+\|f_1\|_{L^{\infty}_{d-1}}).
	\end{eqnarray*}

	\vspace{1ex}
	{\rm (iii)} 
	Let $E(x)$ $(x\in \mathbb{R}^{d})$ be a scalar function satisfying 
	\begin{eqnarray*}
		|\partial^{\alpha}_{x}E(x)|\leq \frac{C}{(1+|x|)^{|\alpha|+d-1}} \ \ (|\alpha|=0,1).
	\end{eqnarray*}
	Assume that $f$ is a scalar function satisfying $\|f\|_{L^{\infty}_d}< \infty$.
	Then there holds the following estimate for $|\alpha|=0,1$.
	\begin{eqnarray*}
		|[\partial^{\alpha}_{x}E\ast f](x)|\leq \frac{C\log{|x|} }{(1+|x|)^{|\alpha|+d-1}}\|f\|_{L^{\infty}_d}.
	\end{eqnarray*}
\end{lem}

\vspace{2ex}

Lemma $\ref{convolution}$  $\rm (i)$ and $\rm (ii)$ is given in \cite[Lemma 2.5]{Shibata-Tanaka} for $d=3$ and the case $d\geq 4$ can be proved similarly;
the assertion $\rm (iii)$ can be proved by a direct computation based on Lemma $\ref{Kuroda}$; and so 
the details are omitted here. 
\vspace{2ex}

The following inequalities will be used to estimate the low frequency part of nonlinear terms. 

\vspace{2ex} 

\begin{lem}\label{nonlinearestimatelemma}
	Let $\ell$ be a nonnegative integer satisfying $\ell \geq d-1$ and $E(x)$ be a scalar function satisfying that
	$$
	|E(x)| \leq \frac{C}{(1+|x|)^{\ell}}
	\ \ \mbox{for} \ \ x \in \mathbb{R}^{d}.
	$$
	Then for $f \in L^{2}_{d-1}$, it holds that
	$$
	\|E\ast f\|_{L^{\infty}_{d-1}} \leq C\{\|(1+|y|)^{-\ell}\|_{L^2}\|f\|_{L^2_{d-1}}+\|f\|_{L^2_{d-1}}\}.
	$$ 
	\vspace{1ex}
	%{\rm (ii)} Let $E(x)$ be a scalar function satisfying that
	%$$
	%|E(x)| \leq \frac{C}{(1+|x|)^{d-2}}
	%\ \ \mbox{for} \ \ x \in \mathbb{R}^{d}.
	%$$
	%5Then for $f \in L^{1}_{d-1}$, %it holds that
	%$$
	%\|E\ast f\|_{L^{\infty}_{d-1}} \leq C\|f\|_{L^1_{d-1}}.
	%$$ 
\end{lem}

Lemma \ref{nonlinearestimatelemma} easily follows from direct computations; and we omit the proof.

\vspace{2ex}

%The following Lemma is related to the weighted $L^{\infty}$ estimate for the low frequency part.

%\begin{lem}\label{P_1estimateincludehigh}
%	$$
%\|F_1\|_{\scr{Y}_{(1),L\infty}}\leq C\|F_1\|_{L^2_{(1), d-1}}.
%$$
%for $F_1\in L^2_{(1), d-1}$.
%\end{lem}

%\vspace{2ex}

%\noindent
%\textbf{Proof.}
%We see that $\tilde{F}_1=\chi_0\ast F_1$, where $\chi_0=\mathcal{F}^{-1}\hat{\chi}_{0}$, $\hat{\chi}_0$ is the cut-off function defined by $(\ref{subcutofflowpart})$. 
%Since $\hat{\chi}_0 \in \scr{S}$, we find that
\begin{eqnarray}
	|\partial_{x}^{\alpha}\chi_0(x)| \leq C(1+|x|)^{-(d+|\alpha|)}\ \ \mbox{for} \ \  |\alpha| \geq 0.\label{E_0estimate}
\end{eqnarray}
%Therefore, applying Lemma \ref{lemP_1 weightedLinfty} and Lemma \ref{nonlinearestimatelemma}, we obtain  the desired estimate.
%This completes the proof. 
%$\hfill\square$

\vspace{2ex}

As for the high frequency part, we have the following inequalities given in \cite[Lemma 4.4]{Kagei-Tsuda}.

\vspace{2ex}
\begin{lem}\label{lemPinfty}
	{\rm (i)} 
	Let $k$ be a nonnegative integer. 
	Then $P_\infty$ is a bounded linear operator on $H^k$. 
	
	\vspace{1ex}
	{\rm (ii)} 
	There hold the inequalities 
	\begin{eqnarray*}
		\|P_\infty f\|_{L^2}
		& \leq &
		C\|\nabla f\|_{L^2} 
		\ \ (f\in H^1),
		\\[2ex]
		\|f_{\infty}\|_{L^2}
		& \leq & 
		C\|\nabla f_{\infty}\|_{L^2} 
		\ \ (f_{\infty}\in H^{1}_{(\infty)}).
	\end{eqnarray*}
\end{lem}

\vspace{2ex}

\begin{lem}\label{lemPinftyweight}
	Let $\ell\in \mathbb{N}$. 
	Then there exists a positive constant $C$ depending only on $\ell$ 
	such that 
	\begin{eqnarray*}
		\| P_{\infty}f\|_{L^2_{\ell}}\leq C\|\nabla f\|_{L^2_{\ell}}.
	\end{eqnarray*}
\end{lem}

\vspace{2ex}

Lemma \ref{lemPinftyweight} follows from the inequalities 
$$
\||x|^k \nabla f_{\infty}\|_{L^2}^2 \geq \frac{r_1^2}{2}\||x|^k f_{\infty}\|_{L^2}^2-C\||x|^{k-1}f_{\infty}\|_{L^2}^2 \ \ (k=1,\cdots,\ell)
$$
for $f_{\infty}\in H^1_{(\infty),\ell}$ which are proved in 
\cite[Lemma 4.7]{Kagei-Tsuda} 
by using the Plancherel theorem.

\vspace{2ex}

	\section{Properties of $S_{1}(t)$ and ${\scr S}_{1}(t)$}\label{S5}
	
	In this section we investigate $S_{1}(t)$ and ${\scr S}_{1}(t)$ and establish estimates for a solution $u_{1}$ of 
	\begin{eqnarray}
		\partial_{t}u_{1}+Au_{1}=F_{1}\label{lowparttimeeq}
	\end{eqnarray}
	satisfying $u_1(0)=u_1(T)$ where $F_1 =\trans(F_1^1,{F}_1^2)$.
	\vspace{2ex}
	
	We denote by $A_{1}$ the restriction of $A$ on ${\scr X}_{(1)}\times {\scr Y}_{(1)}$.
	
	\begin{prop}\label{S1}
		{\rm (i)} 
		$A_{1}$ is a bounded linear operator on ${\scr X}_{(1)}\times {\scr Y}_{(1)}$ 
		and $S_1(t)=e^{-tA_1}$ is a uniformly continuous semigroup on ${\scr X}_{(1)}\times {\scr Y}_{(1)}$. 
		Furthermore, $S_{1}(t)$ satisfies 
		$$
		S_{1}(t)u_{1} \in C^{1}([0,T'];{\scr X}_{(1)}\times {\scr Y}_{(1)}), \ \  
		\partial_{t} S_1(\cdot)u_1\in C([0,T'];L^{2})
		$$ 
		for each $u\in {\scr X}_{(1)}\times {\scr Y}_{(1)}$ 
		and all $T'>0$, 
		$$
		\partial_{t} S_1(t)u_1=-A_1 S_1(t)u_1\,(=-A S_1(t)u_1), 
		\ S_1(0)u_1=u_1  
		\ \ \mbox{for \ $u_1\in {\scr X}_{(1)}\times {\scr Y}_{(1)}$,} 
		$$
		$$
		\|\partial_{t}^k S_1(\cdot)u_1\|_{C([0,T'];{\scr X}_{(1),L^{\infty}}\times {\scr Y}_{(1),L^{\infty}})}\leq 
		C\|u_1\|_{{\scr X}_{(1),L^{\infty}}\times {\scr Y}_{(1),L^{\infty}}},
		$$
		$$
		\|\partial_{t}^k S_1(\cdot)u_1\|_{C([0,T'];{\scr X}_{(1),L^{2}}\times {\scr Y}_{(1),L^{2}})}\leq 
		C\|u_1\|_{{\scr X}_{(1),L^{2}}\times {\scr Y}_{(1),L^{2}}}
		$$
		for $u_1\in {\scr X}_{(1)}\times {\scr Y}_{(1)}$, $k=0,1$, 
		$$
		\|\partial_{t} S_1(t)u_1\|_{C([0,T'];L^{2})}\leq C\|u_1\|_{{\scr X}_{(1)}\times {\scr Y}_{(1)}}. 
		$$
		%	and
		%	$$
		%	\|\partial_{t} \nabla S_1(t)u_1\|_{C([0,T'];L^{2}_1)}\leq C\|u_1\|_{{\scr X}_{(1)}\times {\scr Y}_{(1)}}
		%	$$
		%	for $u_1\in {\scr X}_{(1)}\times {\scr Y}_{(1)}$, where $T'>0$ is any given positive number and $C$ is a positive constant depending on $T'$. 
		
		\vspace{1ex}
		{\rm (ii)} 
		Let the operator $\scr S_1(t)$ be defined by 
		$$
		\scr S_1(t)F_1=\int_0^t S_1(t-\tau)F_1(\tau)\,d\tau
		$$
		for $F_1\in C([0,T];{\scr X}_{(1)}{\scr Y}_{(1)})$. 
		Then 
		$$
		\scr S_1(\cdot)F_1 \in C^1([0,T];{\scr X}_{(1)}\times{\scr Y}_{(1)})
		$$
		for each $F_1\in C([0,T];{\scr X}_{(1)}\times {\scr Y}_{(1)})$ and
		$$
		\del_t \scr S_1(t)F_1+A_1 \scr S_1(t)F_1=F_1(t), 
		\  \scr S_1(0)F_1=0, 
		$$
		$$
		\| \scr S_1(\cdot)F_1\|_{C([0,T];{\scr X}_{(1),L^{p}}\times {\scr Y}_{(1),L^{p}} )}
		\leq C\|F_1\|_{C([0,T];{\scr X}_{(1),L^{p}} \times {\scr Y}_{(1)},L^{p})},
		$$
		$$
		\|\del_t \scr S_1(\cdot)F_1\|_{C([0,T];{\scr X}_{(1),L^{p}}\times {\scr Y}_{(1),L^{p}} )}
		\leq C\|F_1\|_{C([0,T];{\scr X}_{(1),L^{p}} \times {\scr Y}_{(1)},L^{p})},
		$$
		for $p=2,\infty$,
		where $C$ is a positive constant depending on $T$.
		If, in addition, $F_{1}\in C([0,T];L^2_1)$, then $\del_t \scr S_1(\cdot)F_1 \in C([0,T];L^{2})$, %$\del_t \nabla \scr S_1(\cdot)F_1 \in C([0,T];L^{2}_1)$,
		$$
		\|\del_t \scr S_1(\cdot)F_{1}\|_{C([0,T];L^{2})}\leq C\|F_1\|_{C([0,T];L^{2})},
		$$
		%	and
		%	$$
		%	\|\del_t \nabla \scr S_1(\cdot)F_{1}\|_{C([0,T];L^{2}_1)}\leq C\|F_1\|_{C([0,T];L^{2}_1)},
		%	$$
		where $C$ is a positive constant depending on $T$.
		
		\vspace{1ex}
		{\rm (iii)}
		It holds that 
		$$
		S_1(t)\scr S_1(t')F_1=\scr S_1(t')[S_1(t)F_1]
		$$
		for any $t\geq 0$, $t'\in [0,T]$ and $F_1\in C([0,T];{\scr X}_{(1)} \times {\scr Y}_{(1)})$.
	\end{prop}
	
	\vspace{2ex}
	\noindent\textbf{Proof of Proposition} {\bf \ref{S1}.} 
	Let 
	\begin{eqnarray}
		\hat{A}_{\xi}=\begin{pmatrix}
			0 &i\trans{\xi}\\
			i\xi &I
		\end{pmatrix}
		\ \ \ (\xi\in \mathbb{R}^d). 
		\nonumber
	\end{eqnarray}
	Then, ${\cal F}(Au_1)=\hat A_{\xi}\hat u_1$. 
	Hence, if ${\rm supp}\,\hat{u}_{1}\subset \{\xi;|\xi|\leq r_\infty\},$ then 
	${\rm supp}\,\hat{A}_{\xi}\hat{u}_{1}\subset \{\xi;|\xi|\leq r_\infty\}.$ 
	Furthermore, 
	we see from Lemma \ref{lemP_1 weightedLinfty} that
	$$
	\|Au_{1}\|_{{\scr X}_{(1),L^p}\times {\scr Y}_{(1),L^p}}\leq C\|u_{1}\|_{{\scr X}_{(1),L^p}\times {\scr Y}_{(1),L^p}}
	$$
	for $p=2,\infty$. 
	Therefore, $A_{1}$ is a bounded linear operator on ${\scr X}_{(1)}\times {\scr Y}_{(1)}$. 
	It then follows that 
	$-A_1$ generates a uniformly continuous semigroup $S_1(t)=e^{-tA_1}$ 
	that is given by 
	$$
	S_{1}(t)u_1=\mathcal{F}^{-1}e^{-t \hat{A}_{\xi}}\mathcal{F}u_1
	\ \ \ (u_1\in {\scr X}_{(1)}\times {\scr Y}_{(1)}).
	$$
	Furthermore, $S_{1}(t)$ satisfies $S_{1}(\cdot)u_{1} \in C^{1}([0,T'];{\scr X}_{(1)}\times {\scr Y}_{(1)})$ 
	%$\partial_{t} S_1(\cdot)u_1\in L^{2}(0,T';L^{2})$ 
	for each $u\in {\scr X}_{(1)}\times {\scr Y}_{(1)}$ 
	, and
	$$
	\partial_{t} S_1(t)u_1=-A_1 S_1(t)u_1\,(=-A S_1(t)u_1), 
	\ S_1(0)u_1=u_1  
	\ \ \mbox{for \ $u_1\in {\scr X}_{(1)}\times {\scr Y}_{(1)}$.} 
	$$
	
	It easily follows from the definition of $S_{1}(t)$ that
	$$
	\|S_1(\cdot)u_1\|_{C([0,T'];{\scr X}_{(1),L^p} \times {\scr Y}_{(1),L^p})}\leq 
	C\|u_1\|_{{\scr X}_{(1),L^p}\times {\scr Y}_{(1),L^p}}
	\ \ 
	(p=2,\infty)
	\ \ 
	\mbox{for \ $u_1\in {\scr X}_{(1)}\times {\scr Y}_{(1)}$,}
	$$
	and hence, by the relation that $\partial_{t} S_1(t)u_1=-A_1 S_1(t)u_1$ and Lemma \ref{lemP_1 weightedLinfty},
	$$
	\|\partial_{t}S_1(\cdot)u_1\|_{C([0,T'];{\scr X}_{(1),L^{p}} \times {\scr Y}_{(1),L^{p}})}\leq 
	C\|u_1\|_{{\scr X}_{(1),L^{p}}\times {\scr Y}_{(1),L^{p}}}
	\ \ 
	(p=2,\infty)
	\ \ 
	\mbox{for \ $u_1\in {\scr X}_{(1)}\times {\scr Y}_{(1)}$, }
	$$
	where $T'>0$ is any given positive number and $C$ is a positive constant depending on $T'$. 
	In addition, we see from the relation $\del_t S_1(t)u_1=-A_1S_1(t)u_1$ that 
	$\del_t S_1(\cdot)u_1\in C([0,T'];L^2)$, %$\del_t \nabla S_1(\cdot)u_1\in C([0,T'];L^2_1)$,   
	$$
	\|\del_t S_1(\cdot)u_1\|_{C([0,T'];L^2)}
	\leq C\|u_{1}\|_{{\scr X}_{(1)}\times {\scr Y}_{(1)}}.
	$$
	%and 
	%$$
	%\|\del_t \nabla S_1(\cdot)u_1\|_{C([0,T'];L^2_1)}
	%\leq C\|u_{1}\|_{{\scr X}_{(1)}\times {\scr Y}_{(1)}}.
	%$$
	
	The assertion (ii) follows from Lemma \ref{lemP_1 weightedLinfty}, the assertion (i) and the relation 
	$\del_t \scr S_1(t)[F_1]=-A_1 \scr S_1(t)[F_1]+F_1(t)$. 
	The assertion (iii) easily follows from the definitions of $S_1(t)$ and $\scr S_1(t)$. 
	This completes the proof. 
	$\hfill\square$

	\vspace{2ex}
	We next investigate invertibility of $I-S_{1}(T)$.
	
	\vspace{2ex}
	\begin{prop}\label{inverseprop}
		If ${F}_1$ satisfies the conditions given in ${\rm (i)}$-${\rm (iii)}$, then, 
		there uniquely exists $u\in {\scr X}_{(1)}\times {\scr Y}_{(1)}$ that satisfies $(I-S_{1}(T))u={F}_1$ and $u$ satisfies 
		the estimates in each case of ${\rm (i)}$-${\rm (iii)}$. 
		
		\vspace{1ex}
		{\rm (i)} ${F}_1 \in L^2_{(1),1}\cap L^{\infty}\cap L^{1}$; 
		\begin{eqnarray}
			&&\|u\|_{{\scr X}_{(1),L^{\infty}}\times {\scr Y}_{(1),L^{\infty}}}\leq C \{\|F_1\|_{L^{\infty}_d}+\|F_1\|_{L^1}\},\label{inveeseLinftyestimate1}\\
			&&\|u\|_{{\scr X}_{(1),L^{2}}\times {\scr Y}_{(1),L^{2}}}\leq C (\|F_1\|_{L^1}+\|F_1\|_{L^2_1})\label{inverseL2estimate1}.
		\end{eqnarray}
		
		{\rm (ii)} ${F}_1 =\partial_{x}^{\alpha} F^{(1)}_1 \in L^{\infty}_d \cap L^2_{(1),1}$ with $F^{(1)}_1 \in L^2_{(1)}\cap L^{\infty}_{d-1}$ 
		for some $\alpha$ satisfying $|\alpha|=1$;
		\begin{eqnarray*}
			&&\|u\|_{{\scr X}_{(1),L^{\infty}}\times {\scr Y}_{(1),L^{\infty}}}\leq C \{\|F_1\|_{L^{\infty}_d}+\|F^{(1)}_1\|_{L^{\infty}_{d-1}}\},\\
			&&\|u\|_{{\scr X}_{(1),L^{2}}\times {\scr Y}_{(1),L^{2}}}\leq C (\|F^{(1)}_1\|_{L^2}+\| F_1\|_{L^2_1}).
		\end{eqnarray*}
		
		{\rm (iii)} 
		${F}_1 =\partial_{x}^{\alpha} F^{(1)}_1 \in L^2_{(1)}$ with $F^{(1)}_1 \in L^2_{(1),1}\cap L^{\infty}_{d}$ 
		for some $\alpha$ satisfying $|\alpha|\geq 1$;
		\begin{eqnarray}
			&&\|u\|_{{\scr X}_{(1),L^{\infty}}\times {\scr Y}_{(1),L^{\infty}}}\leq C \|F^{(1)}_1\|_{L^{\infty}_{d}},\label{inveeseLinftyestimate3}\\
			&&\|u\|_{{\scr X}_{(1),L^{2}}\times {\scr Y}_{(1),L^{2}}}\leq C \|F^{(1)}_1\|_{L^2_1}\label{inverseL2estimate3}.
		\end{eqnarray}
		
	\end{prop}
	
	\vspace{2ex}
	
	To prove  Proposition \ref{inverseprop}, we prepare some lemmas.
	
	\vspace{2ex}
	\begin{lem}\label{eigenvalue} {\rm (\cite{Wang-Yang})} 
		{\rm (i)} 
		The set of all eigenvalues of $-\hat{A}_{\xi}$ consists of $\lambda_{j}(\xi)\,(j=0,\pm)$, 
		where
		\begin{eqnarray*}
			\lambda_0=-1,\quad \lambda_{\pm}(\xi)=\frac{1}{2}(-1\pm \sqrt{1-4|\xi|^2}).
		\end{eqnarray*}
		We note that 
		$$
		\lambda_+ = - |\xi|^2 +O(|\xi|^4),  \ \ 
		\lambda_- = -1 +|\xi|^2 + O(|\xi|^4)
		$$
		as $|\xi|\rightarrow 0$. 
		If $|\xi|>\frac{1}{2}$, then 
		$$
		{\rm Re}\,\lambda_{\pm}
		=-\frac{1}{2},\quad  
		{\rm Im}\, \lambda_{\pm}
		=\pm|\xi|\sqrt{1-\frac{1}{4|\xi|^2}}.
		$$
		
		{\rm (ii)} For $|\xi|<\frac{1}{2}$, 
		$e^{-t\hat{A}_{\xi}}$ has the spectral resolution 
		\begin{eqnarray*}
			e^{-t\hat{A}_{\xi}}=\sum_{j=\pm}e^{t\lambda_{j}(\xi)}{\Pi}_{j}(\xi),
		\end{eqnarray*}
		where ${\Pi}_{j}(\xi)$ is eigenprojections for $\lambda_{j}(\xi)\,(\pm)$, and $\Pi_{j}(\xi)\,(\pm)$ satisfy
		\begin{align*}
			\Pi_{0}(\xi)&=\begin{pmatrix}
				0 &0\\
				0 &I_{d}-\frac{\xi\trans{\xi}}{|\xi|^{2}}
			\end{pmatrix}
			,\\
			\Pi_{\pm}(\xi)&=\pm\frac{1}{\lambda_{+}-\lambda_{-}}\begin{pmatrix}
				-\lambda_{\mp} &-i\trans{\xi}\\
				-i\xi &\lambda_{\pm}\frac{\xi\trans{\xi}}{|\xi|^{2}}
			\end{pmatrix}.
		\end{align*}
		Furthermore, if $0<r_\infty<\frac{1}{2}$, 
		then there exist a constant $C>0$ such that the estimates
		\begin{eqnarray}
			\|\Pi_{j}(\xi)\|\leq C\,(j=0,\pm) 
			\label{Pi}
		\end{eqnarray}
		hold for $|\xi|\leq r_{\infty}$.
	\end{lem}
	
	\vspace{2ex}
	Hereafter we fix $0<r_1<r_\infty<\frac{1}{2}$ 
	so that $(\ref{Pi})$ in Lemma \ref{eigenvalue} holds for $|\xi|\leq r_\infty$. 
	
	\vspace{2ex}
	\begin{lem}\label{eigenvalueestimate} 
		Let $\alpha$ be a multi-index. 
		Then  the following estimates hold true uniformly for $\xi$ with  $|\xi|\leq r_{\infty}$ and $t\in [0,T]$. 
		\begin{itemize}
			\item[{\rm (i)}] 
			$|\partial_{\xi}^{\alpha}\lambda_{0}|\leq C|\xi|^{-|\alpha|}$,  $|\partial_{\xi}^{\alpha}\lambda_{+}|\leq C|\xi|^{2-|\alpha|}$,   
			\ \ $|\partial_{\xi}^{\alpha}\lambda_{-}|\leq C|\xi|^{-|\alpha|}$ $(|\alpha|\geq 0)$. 
			\item[{\rm (ii)}] 
			$|(\partial_{\xi}^{\alpha}\Pi_{0})\hat{F}_1|\leq C|\xi|^{-|\alpha|}|\hat{F}^2_{1}|$, 
			$|(\partial_{\xi}^{\alpha}\Pi_{\pm})\hat{F}_1|
			\leq C|\xi|^{-|\alpha|}|\hat{F}_{1}|$ $(|\alpha|\geq 0)$, 
			where ${F}_1=\trans(F^{1}_{1},{F}^2_{1})$.
			\item[{\rm (iii)}] 
			$|\partial_{\xi}^{\alpha}(e^{\lambda_{\pm}t})|
			\leq C|\xi|^{-|\alpha|}$ 
			$(|\alpha|\geq 0)$. 
			\item[{\rm (iv)}] 
			$|(\partial_{\xi}^{\alpha}e^{-t\hat{A}_{\xi}})\hat{F}_1|
			\leq  C(|\xi|^{-|\alpha|}|\hat{F}_1^1|+|\xi|^{-|\alpha|}|\hat{F}^2_{1}|)$ 
			$(|\alpha|\geq 1)$, 
			where ${F}_1=\trans(F^{1}_{1},{F}^2_{1})$. 
			%\item[{\rm (v)}] 
			%$|\partial_{\xi}^{\alpha}(I-e^{\lambda_{0}t})^{-1}\hat{F}_1|
			%\leq C|\xi|^{-|\alpha|}|\hat{F}^2_{1}|$  
			%$(|\alpha|\geq 0)$,  
			%where ${F}_1=\trans(F^{1}_{1},{F}^2_{1})$. 
			%\item[{\rm (vi)}] 
			%$|\partial_{\xi}^{\alpha}(I-e^{\lambda_{\pm}t})^{-1}\hat{F}_1|
			%\leq C(|\xi|^{-2-|\alpha|}|\hat{F}^1_{1}|+|\xi|^{-1-|\alpha|}|\hat{F}^2_{1}|)$ 
			%$(|\alpha|\geq 0)$, where ${F}_1=\trans(F^{1}_{1},{F}^2_{1})$.
		\end{itemize}
	\end{lem}
	
	\vspace{2ex}
	Lemma $\ref{eigenvalueestimate}$ can be verified by direct computations based on Lemma $\ref{eigenvalue}.$

	\vspace{2ex}
	
	We set 
	$$
	(I-e^{\lambda_{k}T})^{-1}\Pi_{k}(\xi):=\begin{pmatrix}
		E_{k, 11}(\xi) & E_{k, 12}(\xi) \\
		E_{k, 21}(\xi) & E_{k, 22}(\xi)
	\end{pmatrix}
	.
	$$
	
	\vspace{2ex}
	
	\begin{lem}\label{eigenvalueestimate2} 
		Let $\alpha$ be a multi-index. 
		Then  the following estimates hold true uniformly for $\xi$ with  $|\xi|\leq r_{\infty}$ and $t\in [0,T]$. 
		\begin{itemize}
			\item[{\rm (i)}] 
			$|\partial_{\xi}^{\alpha}E_{+, 11}\hat{F}_1|
			\leq C|\xi|^{-2-|\alpha|}|\hat{F}_{1}|$ 
			$(|\alpha|\geq 0)$
			\item[{\rm (ii)}] 
			$|\partial_{\xi}^{\alpha}E_{k, ij}\hat{F}_1|
			\leq C|\xi|^{-1-|\alpha|}|\hat{F}_{1}|$ 
			$(|\alpha|\geq 0)$, where $k=0,\pm$ and $ij \neq 11$ for $k=+$.  
		\end{itemize}
	\end{lem}
	
	%\item[{\rm (vi)}] 
	%$|\partial_{\xi}^{\alpha}(I-e^{\lambda_{\pm}t})^{-1}\hat{F}_1|
	%\leq C(|\xi|^{-2-|\alpha|}|\hat{F}^1_{1}|+|\xi|^{-1-|\alpha|}|\hat{F}^2_{1}|)$ 
	%$(|\alpha|\geq 0)$, where ${F}_1=\trans(F^{1}_{1},{F}^2_{1})$.
	
	\vspace{2ex}
	
	\vspace{2ex}
	Lemma $\ref{eigenvalueestimate2}$ can also be verified by direct computations based on Lemma $\ref{eigenvalue}.$ We note that 
	$$
	(I-e^{\lambda_{k}T})^{-1}\Pi_{k}\sim \dfrac{1}{\lambda_{k}T}\Pi_{k} 
	$$
	as $|\xi| \rightarrow 0$ by the Taylor expansion of $e^{\lambda_{k}T}$. 
	
	\vspace{2ex}
	
	Set
	\begin{eqnarray*}
		E_{1,k}(x)
		:={\cal F}^{-1}(\hat{\chi}_{0}(I-e^{\lambda_{k}T})^{-1}\Pi_{j}) \ \ (k=0,\pm),
		\ \ \mbox{$(x\in \mathbb{R}^{d})$}
	\end{eqnarray*}
	where $\chi_0$ is the cut-off function defined by $(\ref{subcutofflowpart})$. In addition, we write as    
	\begin{eqnarray*}
		E_{1,k}(x):= \begin{pmatrix}
			E_{1, k, 11} & E_{1, k, 12} \\
			E_{1, k, 21} & E_{1, k, 22}
		\end{pmatrix}
		, 
	\end{eqnarray*}
	where $E_{1, k, ij}= {\cal F}^{-1}(\hat{\chi}_{0}E_{k, ij})$. 
	
	\begin{lem}\label{inverse} 
		Let $\alpha$ be a multi-index satisfying $|\alpha|\geq 0$.
		Then  the following estimates hold true uniformly for $x\in \mathbb{R}^{d}$. 
		\begin{itemize}
			\item[{\rm (i)}]
			$|\partial_{x}^{\alpha}E_{1,+,11}(x)|\leq C(1+|x|)^{-(d-2+|\alpha|)}.$
			\item[{\rm (ii)}] 
			$|\partial_{x}^{\alpha}E_{1,k,ij}(x)|\leq C(1+|x|)^{-(d-1+|\alpha|)}$, where $k=0,\pm$ and  $ij \neq 11$ for $k=+$. 
		\end{itemize}
	\end{lem}
	
	\vspace{2ex}
	\noindent\textbf{Proof.} It follows from Lemma $\ref{eigenvalueestimate}$ that
	\begin{eqnarray*}
		\sum_{k}|\partial_{x}^{\alpha}E_{1,k}(x)|\leq C\int_{|\xi|\leq  r_{\infty}}|\xi|^{-2}d\xi
		\ \ \mbox{$(x\in \mathbb{R}^{d})$}.
	\end{eqnarray*}
	Since $\int_{|\xi|\leq r_{\infty}}|\xi|^{-2}d\xi < \infty$ for $d\geq 3$, 
	we see that
	\begin{eqnarray}
		\sum_{k}|\partial_{x}^{\alpha}E_{1,k}(x)|\leq C
		\ \ \mbox{$(x\in \mathbb{R}^{d})$},\label{inverseconstant}
	\end{eqnarray}
	where $C>0$ is a constant depending on $\alpha$, $T$ and $d$.
	By Lemma $\ref{eigenvalueestimate}$, we have 
	\begin{eqnarray*}
		|\partial_{\xi}^{\beta}((i\xi)^{\alpha}\hat{\chi}_{0}E_{+,11}(\xi)| &\leq &C|\xi|^{-2+|\alpha|-|\beta|}
		\ \ \mbox{for} \ |\beta| \geq 0,\\
		|\partial_{\xi}^{\beta}((i\xi)^{\alpha}\hat{\chi}_{0}E_{k,ij}| &\leq &C|\xi|^{-1+|\alpha|-|\beta|}
		\ \ \mbox{for} \ |\beta| \geq 0,
	\end{eqnarray*}
	where $k=0,\pm$ and  $ij \neq 11$ for $k=+$. 
	It then follows from Lemma  $\ref{sekibunkaku}$ that 
	\begin{eqnarray}
		|\partial_{x}^{\alpha}E_{1,+,11}(x)|\leq C|x|^{-(d-2+|\alpha|)}
		\ \ \mbox{and} \ \ |\partial_{x}^{\alpha}E_{1,k,ij}(x)|\leq C|x|^{-(d-1+|\alpha|)}.\label{inversemugenennpou}
	\end{eqnarray}
	From $(\ref{inverseconstant})$ and $(\ref{inversemugenennpou})$, we obtain the desired estimates. 
	This completes the proof.$\hfill\square$
	
	\vspace{2ex}
	
	%Let us prove Proposition \ref{S1}.

	\vspace{2ex}
	Now we are in a position to prove Proposition \ref{inverseprop}.
	
	\vspace{2ex}
	
	\noindent
	\textbf{Proof of Proposition} {\bf \ref{inverseprop}.}
	We define a function $u$ by
	\begin{eqnarray*}
		u=\mathcal{F}^{- 1}(I-e^{-T\hat{A}_{\xi}})^{-1}\hat{F}_1.
	\end{eqnarray*}
	
	{\rm (i)}  
	By using Lemma \ref{eigenvalueestimate}, one can easily obtain $(\ref{inverseL2estimate1})$. 
	As for $(\ref{inveeseLinftyestimate1})$, note that  
	\begin{eqnarray*}
		u=\mathcal{F}^{- 1}
		((I-e^{-T\hat{A}_{\xi}})^{-1}\hat{F}_1)
		=\sum_{k}E_{1,k}\ast F_1,
	\end{eqnarray*}
	where $E_{1,k}$ is the ones defined in Lemma \ref{inverse}. We can rewrite by 
	\begin{eqnarray*}
		\sum_{k}E_{1,k}\ast F_1 &=
		\begin{pmatrix}
			\sum_{k} E_{k, 11}\ast {F}^1_1 +\sum_{k}E_{k, 12}\ast {F}^2_1 \\
			\sum_{k}E_{k, 21}\ast {F}^1_1 + \sum_{k}E_{k, 22}\ast {F}^2_1
		\end{pmatrix}
		&=\mathcal F^{-1}\Big[\hat{\chi}_{0}
		\begin{pmatrix}
			\sum_{k} E_{k, 11}\hat{F}^1_1 +\sum_{k}E_{k, 12}\hat{F}^2_1 \\
			\sum_{k}E_{k, 21}\hat{F}^1_1 + \sum_{k}E_{k, 22}\hat{F}^2_1
		\end{pmatrix}
		\Big]
	\end{eqnarray*}
	for $F_1=\trans(F_1^1, F_1^2)$. 
	Then by Lemma \ref{inverse}, we see that $\sum_{k}E_{1,k,ij}$ satisfies 
	$$
	|\partial_{x}^{\alpha}E_{1,+,11}(x)|\leq C(1+|x|)^{-(d-2+|\alpha|)} \ \ (|\alpha|\geq 0)
	$$
	and 
	$$
	|\partial_{x}^{\alpha}\sum_{k}E_{1,k,ij}(x)|\leq C(1+|x|)^{-(d-1+|\alpha|)} \ \ (|\alpha|\geq 0), 
	$$
	where $k=0,\pm$ and  $ij \neq 11$ for $k=+$.
	Therefore, applying Lemma \ref{convolution} {\rm (i)}, we obtain $(\ref{inveeseLinftyestimate1})$. 
	
	\vspace{1ex}
	The assertion ${\rm (ii)}$ similarly follows from Lemma \ref{convolution} {\rm (ii)}, Lemma \ref{eigenvalueestimate} and Lemma \ref{inverse}.
	
	\vspace{1ex}
	${\rm (iii)}$ By using Lemma \ref{eigenvalueestimate}, one can easily obtain $(\ref{inverseL2estimate3})$. 
	As for $(\ref{inveeseLinftyestimate3})$,  
	if there exists a function $F^{(1)}_1 \in L^2_{(1)}\cap L^{\infty}_d$ satisfying $F_1 =\partial_{x}^{\alpha} F^{(1)}_1$ 
	for some $\alpha$ satisfying $|\alpha|\geq 1$, then
	\begin{eqnarray*}
		u=\left(\sum_{k}\partial_{x}^{\alpha}  E_{1,k}\right)\ast F^{(1)}_1.
	\end{eqnarray*}
	%Here, for a matrix function $F=(F_{jk})$, $\nabla \cdot F$ denotes the vector function $(\sum_{k=1}^n \partial_{x_k}F_{jk})$.  
	Lemma  \ref{inverse} yields
	\begin{eqnarray*}
		|\sum_{k}\partial_{x}^{\alpha+\beta}  E_{1,k}(x)| \leq C(1+|x|)^{-(d-1+|\beta|)}
	\end{eqnarray*}
	for $x\in \mathbb{R}^{d}$, $|\alpha| \geq 1$ and $|\beta|\geq 0$. 
	It  then follows from Lemma \ref{convolution} {\rm (iii)} that
	$$
	\|u\|_{{\scr X}_{(1),L^{\infty}}\times {\scr Y}_{(1)},L^{\infty}}\leq C\|F^{(1)}_1\|_{L^{\infty}_d}.
	$$
	
	%Consequently, we obtain that
	%$$
	%\|u\|_{{\scr X}_{(1)}\times {\scr Y}_{(1)}}\leq C \{\|F^{(1)}_1\|_{L^{\infty}_n}+\|F^{(1)}_1\|_{L^2_1}\}.
	%$$ 
	\noindent This completes the proof.$\hfill\square$
	
	\vspace{2ex}

	%By Lemma \ref{eigenvalue}, Lemma \ref{eigenvalueestimate}, and Proposition \ref{inverseprop}, we obtain the following Proposition related to the 
	%estimate of the solution of $(\ref{lowparttimeeq})$.
	
	%\vspace{2ex}
	
	%\begin{Prop}\label{lowestimatenakami}
	%Let $F_{1}\in \{0\}\times L^{2}_{(1),\infty}$. 
	%We set
	%\begin{eqnarray*}
	%\Gamma [F_{1}]=E(\tau)\ast F_{1}
	%\end{eqnarray*}
	%where 
	%$$
	%E=\mathcal{F}^{- 1}\{e^{-t\hat{A}_{\xi}}(I-e^{-T\hat{A}_{\xi}})^{-1}e^{-(T-\tau )\hat{A}_{\xi}}+e^{-(t-\tau )\hat{A}_{\xi}}\}.
	%$$
	%Then, there holds the estimates
	% \begin{eqnarray*}
		%\|Gamma [F_{1}]\|_{{\scr X}_{(1)}\times {\scr Y}_{(1)}}\leq C Q(F_{1}).
		%\end{eqnarray*}
		%\end{prop}

		\vspace{2ex} 
		In view of Proposition \ref{inverseprop} {\rm (i)}, $I-S_{1}(T)$ has a bounded inverse $(I-S_{1}(T))^{-1}$: $L^2_{(1),1}\cap L^{\infty}\cap L^{1}
		\rightarrow {\scr X}_{(1)}\times {\scr Y}_{(1)}$ and 
		it holds that
		\begin{align*}
			&\|(I-S_{1}(T))^{-1}F_{1}\|_{{\scr X}_{(1),L^{\infty}}\times {\scr Y}_{(1),L^{\infty}}}\leq C \{\|F_1\|_{L^{\infty}_d}+\|F_1\|_{L^1}\},\\
			&\|(I-S_{1}(T))^{-1}F_{1}\|_{{\scr X}_{(1),L^2}\times {\scr Y}_{(1),L^2}}\leq C (\|F_1\|_{L^1}+\|F_1\|_{L^2_1}).
		\end{align*}
		
		If ${F}_1 =\partial_{x}^{\alpha} F^{(1)}_1 \in L^{\infty}_d \cap L^2_{(1),1}$ with $F^{(1)}_1 \in L^2_{(1)}\cap L^{\infty}_{d-1}$ 
		for some $\alpha$ satisfying $|\alpha|=1$, then $(I-S_{1}(T))^{-1}F_1 \in {\scr X}_{(1)}\times {\scr Y}_{(1)}$ and 
		\begin{eqnarray*}
			&&\|(I-S_{1}(T))^{-1}F_1\|_{{\scr X}_{(1),L^{\infty}}\times {\scr Y}_{(1),L^{\infty}}}\leq C \{\|F_1\|_{L^{\infty}_d}+\|F^{(1)}_1\|_{L^{\infty}_{d-1}}\},\\
			&&\|(I-S_{1}(T))^{-1}F_1\|_{{\scr X}_{(1),L^{2}}\times {\scr Y}_{(1),L^{2}}}\leq C (\|F_1\|_{L^2_1}+\|F^{(1)}_1\|_{L^2}).
		\end{eqnarray*}
		
		Furthermore, if ${F}_1 =\partial_{x}^{\alpha} F^{(1)}_1 \in L^2_{(1)}$ with $F^{(1)}_1 \in L^2_{(1),1}\cap L^{\infty}_{d}$ 
		for some $\alpha$ satisfying $|\alpha|\geq 1$, then $(I-S_{1}(T))^{-1}F_1 \in {\scr X}_{(1)}\times {\scr Y}_{(1)}$ and 
		\begin{eqnarray*}
			&&\|(I-S_{1}(T))^{-1}F_1\|_{{\scr X}_{(1),L^{\infty}}\times {\scr Y}_{(1),L^{\infty}}}\leq C \|F^{(1)}_1\|_{L^{\infty}_{d}},\\
			&&\|(I-S_{1}(T))^{-1}F_1\|_{{\scr X}_{(1),L^{2}}\times {\scr Y}_{(1),L^{2}}}\leq C \|F^{(1)}_1\|_{L^2_1}.
		\end{eqnarray*}
		
		\vspace{2ex}

		We next estimate $S_1(t)\scr{S}_1(T)(I-S_1 (T))^{-1}F_{1}$ and $\scr{S}_1(t)F_1$.
		Let $E_{1}(t,\sigma)$ and $E_{2}(t,\tau)$ be defined by
		\begin{eqnarray*}
			&&E_{1}(t,\sigma)=\mathcal{F}^{- 1}\{\hat{\chi}_{0}e^{-t\hat{A}_{\xi}}(I-e^{-T\hat{A}_{\xi}})^{-1}e^{-(T-\sigma )\hat{A}_{\xi}}\},\\
			&&E_{2}(t,\tau)=\mathcal{F}^{- 1}\{\hat{\chi}_{0}e^{-(t-\tau )\hat{A}_{\xi}}\}
		\end{eqnarray*}
		for $\sigma \in [0,T], $ $0\leq \tau \leq t \leq T$, where $\hat{\chi}_{0}$ is the cut-off function defined by (4.9).  
		Then $\scr{S}_1(t)F_1$ and $S_1(t)\scr{S}_1(T)(I-S_1 (T))^{-1}F_{1}$ are given by 
		\begin{eqnarray}
			&&S_1(t)\scr{S}_1(T)(I-S_1 (T))^{-1}F_{1}=\int_{0}^{T}E_1(t,\sigma) \ast F_{1}(\sigma)d\sigma \label{kaisayosotosekibunkaku2}\\
			&&\scr{S}_1(t)F_1=\int_{0}^{t}S_1(t-\tau)F_{1}(\tau)d\tau=\int_{0}^{t}E_2(t,\tau) \ast F_{1}(\tau)d\tau \label{kaisayosotosekibunkaku1}.
		\end{eqnarray}
		
		We have the following estimates for $E_1(t,\sigma)\ast F_1$ and $E_2(t,\tau)\ast F_1$.
		
		\vspace{2ex}
		\begin{lem}\label{skibunkakumarugotoestimate}
			If ${F}_1$ satisfies the conditions given in ${\rm (i)}$-${\rm (iii)}$, then, $E_1(t,\sigma) \ast F_{1}\in {\scr X}_{(1)}\times {\scr Y}_{(1)}$, $E_2(t,\tau) \ast F_{1}\in {\scr X}_{(1)}\times {\scr Y}_{(1)}$ $(t, \sigma, \tau \in [0,T],j=1,2)$ 
			and $E_1(t,\sigma)\ast F_1,E_2(t,\tau)\ast F_1$ satisfy the estimates in each case of ${\rm (i)}$-${\rm (iii)}$. 
			
			\vspace{1ex}
			${\rm (i)}$ ${F}_1 \in L^2_{(1),1}\cap L^{\infty}\cap L^{1}$; 
			$$
			\|E_1(t,\sigma) \ast F_{1} \|_{{\scr X}_{(1),L^{\infty}}\times {\scr Y}_{(1),L^{\infty}}}+
			\|E_2(t,\tau) \ast F_{1} \|_{{\scr X}_{(1),L^{\infty}}\times {\scr Y}_{(1),L^{\infty}}}\leq C \{\|F_1\|_{L^{\infty}_d}+\|F_1\|_{L^1}\}
			$$
			and 
			$$
			\|E_1(t,\sigma) \ast F_{1} \|_{{\scr X}_{(1),L^{2}}\times {\scr Y}_{(1),L^{2}}}+
			\|E_2(t,\tau) \ast F_{1}\|_{{\scr X}_{(1),L^{2}}\times {\scr Y}_{(1),L^{2}}}\leq C (\|F_1\|_{L^1}+\|F_1\|_{L^2_1})
			$$
			uniformly for $\sigma \in [0,T]$ and $0\leq \tau \leq t \leq T$. 
			
			\vspace{1ex}
			${\rm (ii)}$ ${F}_1 =\partial_{x}^{\alpha} F^{(1)}_1 \in L^{\infty}_d \cap L^2_{(1),1}$ with $F^{(1)}_1 \in L^2_{(1)}\cap L^{\infty}_{d-1}$ 
			for some $\alpha$ satisfying $|\alpha|=1$;
			$$
			\|E_1(t,\sigma) \ast F_{1} \|_{{\scr X}_{(1),L^{\infty}}\times {\scr Y}_{(1),L^{\infty}}}+
			\|E_2(t,\tau) \ast F_{1} \|_{{\scr X}_{(1),L^{\infty}}\times {\scr Y}_{(1),L^{\infty}}}
			\leq C \{\|F_1\|_{L^{\infty}_d}+\|F^{(1)}_1\|_{L^{\infty}_{d-1}}\}
			$$
			and 
			$$
			\|E_1(t,\sigma) \ast F_{1} \|_{{\scr X}_{(1),L^{2}}\times {\scr Y}_{(1),L^{2}}}+
			\|E_2(t,\tau) \ast F_{1} \|_{{\scr X}_{(1),L^{2}}\times {\scr Y}_{(1),L^{2}}}
			\leq C (\|F_1\|_{L^2_1}+\|F^{(1)}_1\|_{L^2})
			$$
			uniformly for $\sigma \in [0,T]$ and $0\leq \tau \leq t \leq T$. 
			
			\vspace{1ex}
			${\rm (iii)}$ ${F}_1 =\partial_{x}^{\alpha} F^{(1)}_1 \in L^2_{(1)}$ with $F^{(1)}_1 \in L^2_{(1),1}\cap L^{\infty}_{d}$ 
			for some $\alpha$ satisfying $|\alpha|\geq 1$; 
			$$
			\|E_1(t,\sigma) \ast F_{1} \|_{{\scr X}_{(1),L^{\infty}}\times {\scr Y}_{(1),L^{\infty}}}+
			\|E_2(t,\tau) \ast F_{1} \|_{{\scr X}_{(1),L^{\infty}}\times {\scr Y}_{(1),L^{\infty}}}\leq C \|F^{(1)}_1\|_{L^{\infty}_d}
			$$ 
			and 
			$$
			\|E_1(t,\sigma) \ast F_{1} \|_{{\scr X}_{(1),L^{2}}\times {\scr Y}_{(1),L^{2}}}+
			\|E_2(t,\tau) \ast F_{1} \|_{{\scr X}_{(1),L^{2}}\times {\scr Y}_{(1),L^{2}}}\leq C \|F^{(1)}_1\|_{L^2_1}
			$$ 
			uniformly for $\sigma \in [0,T]$ and $0\leq \tau \leq t \leq T$.
		\end{lem}

		\vspace{2ex}
		
		\noindent
		\textbf{Proof of Lemma} {\bf \ref{skibunkakumarugotoestimate}.} 
		By Lemmas \ref{eigenvalue} and \ref{eigenvalueestimate}, we see that
		\begin{align*}
			&| \partial_{\xi}^{\beta}(\hat{\chi}_{0}(i\xi)^{\alpha}e^{-t\hat{A}_{\xi}}(I-e^{-T\hat{A}_{\xi}})^{-1}e^{-(T-\sigma )\hat{A}_{\xi}})| \leq C|\xi|^{-2+|\alpha|-|\beta|},\\
			&|\partial_{\xi}^{\beta}(\hat{\chi}_{0}(i\xi)^{\alpha}e^{-(t-\tau )\hat{A}_{\xi}})|\leq C|\xi|^{|\alpha|-|\beta|}
		\end{align*}
		for $\sigma \in [0,T]$, $0 \leq \tau \leq t \leq T$ and $|\beta| \geq 0$. It then follows from Lemma \ref{sekibunkaku} that
		\begin{eqnarray}
			|\partial_{x}^{\alpha}E_1(x)| \leq C(1+|x|)^{-(d-2+|\alpha|)}, \ \ 
			|\partial_{x}^{\alpha}E_2(x)| \leq C(1+|x|)^{-(d+|\alpha|)}\label{sekibunkakumarugotogennsuiestimate}
		\end{eqnarray}
		for  $|\alpha| \geq 0$. 
		Therefore, in a similarly manner to the proof of Proposition \ref{inverseprop}, we obtain the desired estimate by using Lemma $\ref{convolution}$ and Lemma $\ref{inverse}$.  This completes the proof. 
		$\hfill\square$

		\vspace{2ex}

		We see from Proposition \ref{S1} (i), (ii) and Lemma \ref{skibunkakumarugotoestimate} that
		the following estimates hold for $S_1(t)\scr{S}_1(T)(I-S_1 (T))^{-1}$ and $\scr{S}_{1}(t)$.

		\vspace{2ex}
		\begin{prop}\label{S_1andscrS_1estimate}
			Let $\Gamma_1$ and $\Gamma_2$ be defined by
			\begin{eqnarray}
				\Gamma_1[\tilde{F_1}](t)=S_1(t)\scr{S}_1(T)(I-S_1 (T))^{-1}
				\mathbb{F}^1
				, \ 
				\Gamma_2 [\tilde{F_1}](t)=\scr{S}_1(t)
				\mathbb{F}^1, 
				\label{kaisayousononlinearhyougenn}
			\end{eqnarray}
			where $\mathbb{F}^1=\begin{pmatrix}
				0\\
				\tilde{F_1}
			\end{pmatrix} 
			$ or $\mathbb{F}^1=\begin{pmatrix}
				\tilde{F_1}\\
				0
			\end{pmatrix}.
			$
			If $\tilde{F}_1$ satisfies the conditions given in ${\rm (i)}$-${\rm (iii)}$, then,  $\Gamma_j[\tilde{F_1}] \in $ $C^1([0,T];{\scr X}_{(1)} \times {\scr Y}_{(1)})$ $(j=1,2)$ and $\Gamma_j[\tilde{F_1}]$ satisfy the estimates in each case of ${\rm (i)}$-${\rm (iii)}$ for $j=1,2$. 
			
			\vspace{1ex}
			${\rm (i)}$ $\tilde{F}_1 \in C([0,T];L^2_{(1),1}\cap L^{\infty}\cap L^{1}\cap {\scr X}_{(1)}\times {\scr Y}_{(1)})$; 
			
			$$
			\|\Gamma_1[\tilde{F_1}]\|_{C([0,T];{\scr X}_{(1)}\times {\scr Y}_{(1)})} \leq 
			C\|\tilde{F}_1\|_{C([0,T];L^{\infty}_d \cap L^1\cap L^2_1)},
			$$
			$$
			\|\del_t \Gamma_1[\tilde{F_1}]\|_{C([0,T];{\scr X}_{(1)})\times {\scr Y}_{(1)})} \leq 
			C\|\tilde{F}_1\|_{C([0,T];L^{\infty}_d \cap L^1\cap L^2_1)}
			$$
			and 
			\begin{align*}
				&\| \Gamma_2[\tilde{F_1}]\|_{C([0,T];{\scr X}_{(1)}\times{\scr Y}_{(1)})} \leq 
				C\|\tilde{F}_1\|_{C([0,T];L^{\infty}_d \cap L^1\cap L^2_1)},
				\\
				&\|\del_t\Gamma_2[\tilde{F_1}]\|_{C([0,T];{\scr X}_{(1)})\times  {\scr Y}_{(1)})} \leq 
				C(\|\tilde{F}_1\|_{C([0,T];L^{\infty}_d \cap L^1\cap L^2_1)}+\|\tilde{F}_1\|_{ C([0,T];{\scr X}_{(1)}\times  {\scr Y}_{(1)})}).
			\end{align*}
			
			\vspace{1ex}
			${\rm (ii)}$ $\tilde{F}_1 =\partial_{x}^{\alpha} F^{(1)}_1 \in C([0,T];L^{\infty}_d \cap L^2_{(1),1}\cap {\scr X}_{(1)}\times {\scr Y}_{(1)})$ with $F^{(1)}_1 \in C([0,T];L^2_{(1)}\cap L^{\infty}_{d-1})$ 
			for some $\alpha$ satisfying $|\alpha|=1$;
			$$
			\|\Gamma_1[\tilde{F_1}]\|_{C([0,T];{\scr X}_{(1)}\times {\scr Y}_{(1)})} \leq 
			C(\|\tilde{F}_1\|_{C([0,T];L^{\infty}_d \cap L^2_1 )}+\|{F}^{(1)}_1\|_{C([0,T];L^{\infty}_{d-1}\cap L^2 )}),
			$$
			$$
			\|\del_t \Gamma_1[\tilde{F_1}]\|_{C([0,T];{\scr X}_{(1)})\times {\scr Y}_{(1)})} \leq 
			C(\|\tilde{F}_1\|_{C([0,T];L^{\infty}_d \cap L^2_1 )}+\|{F}^{(1)}_1\|_{C([0,T];L^{\infty}_{d-1}\cap L^2 )})
			$$
			and 
			\begin{eqnarray*}
				\| \Gamma_2[\tilde{F_1}]\|_{C([0,T];{\scr X}_{(1)}\times {\scr Y}_{(1)})} &\leq &
				C(\|\tilde{F}_1\|_{C([0,T];L^{\infty}_d\cap L^2_1  )}+\|{F}^{(1)}_1\|_{C([0,T];L^{\infty}_{d-1}\cap L^2 )}),
				\\
				\|\del_t\Gamma_2[\tilde{F_1}]\|_{C([0,T];{\scr X}_{(1)})\times  {\scr Y}_{(1)})} &\leq &
				C(\|\tilde{F}_1\|_{C([0,T];L^{\infty}_d \cap L^2_1 )}+\|{F}^{(1)}_1\|_{C([0,T];L^{\infty}_{d-1}\cap L^2 )}\\
				&\quad &+\|\tilde{F}_1\|_{ C([0,T];{\scr X}_{(1)}\times  {\scr Y}_{(1)})}).
			\end{eqnarray*}
			
			\vspace{1ex}
			${\rm (iii)}$ $\tilde{F}_1 =\partial_{x}^{\alpha} F^{(1)}_1 \in C([0,T];L^2_{(1)}\cap {\scr X}_{(1)}\times {\scr Y}_{(1)})$ with $F^{(1)}_1 \in C([0,T];L^2_{(1),1}\cap L^{\infty}_{d})$ 
			for some $\alpha$ satisfying $|\alpha|\geq 1$; 
			$$
			\|\Gamma_1[\tilde{F_1}]\|_{C([0,T];{\scr X}_{(1)}\times {\scr Y}_{(1)})} \leq 
			C\|F^{(1)}_1\|_{C([0,T];L^{\infty}_d  \cap L^2_1 )},
			$$
			$$
			\|\del_t \Gamma_1[\tilde{F_1}]\|_{C([0,T];{\scr X}_{(1)})\times {\scr Y}_{(1)})} \leq 
			C\|F^{(1)}_1\|_{C([0,T];L^{\infty}_d  \cap L^2_1 )}
			$$
			and 
			\begin{align*}
				&\| \Gamma_2[\tilde{F_1}]\|_{C([0,T];{\scr X}_{(1)}\times  {\scr Y}_{(1)})} \leq 
				C\|F^{(1)}_1\|_{C([0,T];L^{\infty}_d \cap L^2_1 )}
				,\\
				&\|\del_t \Gamma_2[\tilde{F_1}]\|_{C([0,T];{\scr X}_{(1)})\times  {\scr Y}_{(1)})} \leq 
				C(\|F^{(1)}_1\|_{C([0,T];L^{\infty}_d \cap L^2_1 )}
				+\|\tilde{F}_1\|_{C([0,T];{\scr X}_{(1)}\times  {\scr Y}_{(1)})}).
			\end{align*}

		\end{prop}
		
		\vspace{2ex}
		
		As for $\|\tilde{F}_1\|_{ C([0,T];{\scr X}_{(1),L^p} \times {\scr Y}_{(1),L^p})}$ $(p=2,\infty)$, we have the following proposition.
		\begin{prop}\label{F1tandokuestimates}
			If $\tilde{F}_1$ satisfies the conditions given in ${\rm (i)}$-${\rm (iii)}$, then, 
			$\tilde{F}_1\in  C([0,T];{\scr X}_{(1)}\times  {\scr Y}_{(1)})$ and $\tilde{F}_1$ satisfies the estimates in each case of ${\rm (i)}$-${\rm (iii)}$. 
			
			\vspace{1ex}
			${\rm (i)}$ $\tilde{F}_1 \in C([0,T];L^2_{(1),1}\cap L^{\infty}\cap L^{1})$; 
			\begin{eqnarray*}
				&& \|\tilde{F}_1\|_{C([0,T];{\scr X}_{(1),\infty} \times {\scr Y}_{(1),L^{\infty}})} \leq C\|\tilde{F}_1\|_{L^{2}(0,T;L^{\infty}_d \cap L^1)}, \\
				&& \|\tilde{F}_1\|_{C([0,T];{\scr X}_{(1), L^2}\times {\scr Y}_{(1),L^2})} \leq C\|\tilde{F}_1\|_{L^{2}(0,T; L^1\cap L^2_1)}.
			\end{eqnarray*}
			
			\vspace{1ex}
			${\rm (ii)}$ $\tilde{F}_1 =\partial_{x}^{\alpha} F^{(1)}_1 \in C([0,T];L^{\infty}_d \cap L^2_{(1),1})$ with $F^{(1)}_1 \in C([0,T];L^2_{(1)}\cap L^{\infty}_{d-1})$ for some $\alpha$ satisfying $|\alpha|= 1$ ;
			\begin{eqnarray*}
				&&\|\tilde{F}_1\|_{C([0,T];{\scr X}_{(1),\infty} \times {\scr Y}_{(1),L^{\infty}})} \leq C(\|\tilde{F}_1\|_{L^{2}(0,T;L^{\infty}_d )}+\|{F}^{(1)}_1\|_{L^{2}(0,T;L^{\infty}_{d-1} )}),\\ 
				&&\|\tilde{F}_1\|_{C([0,T];{\scr X}_{(1), L^2}\times {\scr Y}_{(1),L^2})} \leq C(\|{F}^{(1)}_1\|_{L^{2}(0,T; L^2)}+\|\tilde{F}_1\|_{L^{2}(0,T; L^2_1)}).
			\end{eqnarray*}
			
			\vspace{1ex}
			${\rm (iii)}$ $\tilde{F}_1 =\partial_{x}^{\alpha} F^{(1)}_1 \in C([0,T];L^2_{(1)})$ with $F^{(1)}_1 \in C([0,T];L^2_{(1),1}\cap L^{\infty}_{d})$ 
			for some $\alpha$ satisfying $|\alpha|\geq 1$; 
			$$
			\|\tilde{F}_1\|_{C([0,T];{\scr X}_{(1),\infty} \times {\scr Y}_{(1),L^{\infty}})}\leq C\|F^{(1)}_{1}\|_{L^{2}(0,T;L^{\infty}_d )}, 
			$$
			$$
			\|\tilde{F}_1\|_{C([0,T];{\scr X}_{(1), L^2}\times {\scr Y}_{(1),L^2})} \leq C\|F^{(1)}_{1}\|_{L^{2}(0,T;L^2_1 )}.
			$$
		\end{prop}

		\vspace{2ex}

		\noindent
		\textbf{Proof of Proposition} {\bf \ref{F1tandokuestimates}.} 
		We see that $\tilde{F}_1=\chi_0\ast F_1$, where $\chi_0=\mathcal{F}^{-1}\hat{\chi}_{0}$, $\hat{\chi}_0$ is the cut-off function defined by $(\ref{subcutofflowpart})$  satisfying $(\ref{E_0estimate})$. 
		Therefore, in a similar manner to the proof of Proposition \ref{inverseprop}, we obtain the desired estimates.  
		This completes the proof. 
		$\hfill\square$

		\vspace{2ex}

				\vspace{2ex}

				We are now in a position to give estimates for a solution of \eqref{lowparttimeeq} satisfying $u_1(0)=u_1(T)$. 
				
				For $\mathbb{F}^1=\begin{pmatrix}
					0\\
					\tilde{F_1}
				\end{pmatrix} 
				$ or $\mathbb{F}^1=\begin{pmatrix}
					\tilde{F_1}\\
					0
				\end{pmatrix}, 
				$ we set 
				$$
				\Gamma[\tilde{F}_1] =S_1(t)\scr{S}_1(T)(I-S_1(T))^{-1}\mathbb{F}^1 +\scr{S}_1(t)\mathbb{F}^1.
				$$
				Then $\Gamma[\tilde{F}_1]$ is written as 
				\begin{eqnarray}
					\Gamma[\tilde{F}_1](t)=\Gamma_1[\tilde{F}_1]+\Gamma_2[\tilde{F}_1]\label{uiconclusion},
				\end{eqnarray}
				where $\Gamma_1$ and $\Gamma_2$ are the ones defined by $(\ref{kaisayousononlinearhyougenn})$.

				\begin{prop}\label{propuiestimate}
					If $\tilde{F}_1$ satisfies the conditions given in ${\rm (i)}$-${\rm (v)}$, then, $\Gamma[\tilde{F}_1]$ is a solution of \eqref{lowparttimeeq} with $\mathbb{F}^1$ in $\scr{Z}_{(1)}(0,T)$ 
					satisfying $\Gamma[\tilde{F}_1](0)=\Gamma[\tilde{F}_1](T)$ and $\Gamma[\tilde{F}_1]$ satisfies the estimate in each case of ${\rm (i)}$-${\rm (v)}$.
					
					\vspace{1ex}
					${\rm (i)}$ $\tilde{F}_1 \in C([0,T];L^2_{(1),1}\cap L^{\infty}\cap L^{1})$; 
					\begin{eqnarray}
						\|\Gamma[\tilde{F}_1]\|_{\scr{Z}_{(1)}(0,T)} \leq 
						C\|\tilde{F}_1\|_{C([0,T];L^{\infty}_d \cap L^1\cap L^2_1)}.\label{propuiestimate1}
					\end{eqnarray}
					
					\vspace{1ex}
					${\rm (ii)}$ $\tilde{F}_1 =\partial_{x}^{\alpha} F^{(1)}_1 \in C([0,T];L^{\infty}_d \cap L^2_{(1),1})$ with $F^{(1)}_1 \in C([0,T];L^2_{(1)}\cap L^{\infty}_{d-1})$ for some $\alpha$ satisfying $|\alpha|= 1$ ;
					\begin{eqnarray}
						\|\Gamma[\tilde{F}_1]\|_{\scr{Z}_{(1)}(0,T)}
						\leq
						C(\|\tilde{F}_1\|_{C([0,T];L^{\infty}_d \cap L^2_1 )}+\|{F}^{(1)}_1\|_{C([0,T];L^{\infty}_{d-1}\cap L^2 )}).\label{propuiestimate2}
					\end{eqnarray}
					
					\vspace{1ex}
					${\rm (iii)}$ $\tilde{F}_1 =\partial_{x}^{\alpha} F^{(1)}_1 \in C([0,T];L^2_{(1)})$ with $F^{(1)}_1 \in C([0,T];L^2_{(1),1}\cap L^{\infty}_{d})$ 
					for some $\alpha$ satisfying $|\alpha|\geq 1$; 
					\begin{eqnarray}
						\|\Gamma[\tilde{F}_1]\|_{\scr{Z}_{(1)}(0,T)} \leq 
						C\|{F}^{(1)}_1\|_{C([0,T];L^{\infty}_d \cap L^2_1 )}.\label{propuiestimate3}
					\end{eqnarray}
					
					\vspace{1ex}
					%${\rm (iv)}$ $\tilde{F}_1 =\partial_{x}^{\alpha} F^{(1)}_1 \in C([0,T];L^{1}_{d-1}\cap L^{2}_{(1)})$ for some $\alpha$ satisfying $|\alpha|\geq 0$; 
					%\begin{eqnarray}
					%\|\Gamma[\tilde{F}_1]\|_{\scr{Z}_{(1)}(0,T)} \leq 
					%	C\|{F}_1^{(1)}\|_{C([0,T];L^{1}_{d-1})}.\label{propuiestimate4}
					%\end{eqnarray}
					
					%	\vspace{1ex}
					%	${\rm (v)}$ $\tilde{F}_1 =\partial_{x}^{\alpha} F^{(1)}_1  \in C([0,T];L^{2}_{(1),d-1})$ for some $\alpha$ satisfying $|\alpha|\geq 1$; 
					%	\begin{eqnarray}
						%\|\Gamma[\tilde{F}_1]\|_{\scr{Z}_{(1)}(0,T)} \leq 
						%	C\|{F}_1^{(1)}\|_{C([0,T];L^{2}_{d-1})}.\label{propuiestimate5}
						%\end{eqnarray}
						
					\end{prop}

					\vspace{2ex}
					\noindent\textbf{Proof.} 
					We find from Proposition \ref{S1} (iii), Proposition \ref{inverseprop} and Proposition \ref{F1tandokuestimates} that 
					$\Gamma[\tilde{F}_1]$ is a solution of \eqref{lowparttimeeq} with $F_1=\trans (0,\tilde{F}_1)$ satisfying $\Gamma[\tilde{F}_1](0)=\Gamma[\tilde{F}_1](T)$. 
					The estimates of $\Gamma[\tilde{F}_1]$ in ${\rm (i)}$-${\rm (iii)}$ follow from Proposition \ref{S_1andscrS_1estimate} and Proposition \ref{F1tandokuestimates}.  
					%We obtain the estimates of $\Gamma[\tilde{F}_1]$ in ${\rm (iv)}$ and ${\rm (v)}$ by Proposition \ref{S_1andscrS_1estimate for nonlinear}.  
					This completes the proof.
					$\hfill\square$

					\section{Properties of $S_{{\infty},\tilde{u}}(t)$ and $\scr{S}_{\infty,\tilde{u}}(t)$}\label{S6}
					
					In this section we mainly investigate some properties of $S_{{\infty},\tilde{u}}(t)$ and $\scr{S}_{\infty,\tilde{u}}(t)$ in weighted Sobolev spaces. 
					
					First of all, we consider local existence of the following linear system. 
					
					\begin{equation}\label{Mainsystem-0}
						\left\{
						\begin{aligned}
							&\partial_ta_{\infty}+\text{div}v_\infty + \tilde{v}\cdot\nabla a_\infty=F_\infty^0,\\
							&\partial_tv_\infty+\nabla a_\infty+v_\infty+\tilde{v}\cdot\nabla v_\infty+ g_3(\tilde{\phi})\cdot\nabla a_\infty= F_\infty^1,\\
							&(a_\infty,v_\infty)|_{t=0}=(a_{0\infty},v_{0\infty}).
						\end{aligned}
						\right.
					\end{equation}
					
					\vspace{2ex}
					\begin{lem}\label{lemA2-0}
						Let $d\geq 3$ and $s$ be an integer satisfying $s\geq [\frac{d}{2}]+2$. Assume that $\nabla\tilde{\phi}\in C([0,T'];H^{s-1}),~\tilde v\in C([0,T'];H^s)$ with 
						$$
						\|\nabla\tilde{\phi}\|_{H^{s-1}} +\|\tilde{v}\|_{H^s} \leq 1,
						$$
						$F_\infty=\trans{(F_\infty^0,F_\infty^1)} \in L^2(0,T';H_{(\infty)}^s) \cap C(0,T';H_{(\infty)}^{s-1})$ and $\trans(a_{0\infty}, v_{0\infty})\in H_{(\infty)}^s$. Here $T'$ is a given positive number. Then \eqref{Mainsystem-0}  has a unique solution $\trans(a_\infty, v_{\infty}) \in C([0,T'];H_{(\infty)}^s)$ satisfying that 
						$$
						\|\trans(a_\infty, v_{\infty}) \|_{C([0,T'];H^s)} \leq C\|\trans(a_{0\infty}, v_{0\infty})\|_{H^s}+ \|F_\infty\|_{ L^2(0,T';H^s) \cap C(0,T';H^{s-1})},
						$$
						where $C>0$ depends on $T'$. 
					\end{lem}
					\vspace{2ex}
					
					Since $$
					\sup_{0\leq t\leq T'} (\|g_3(\tilde{\phi})(t)\|_{L^\infty} + \|\nabla g_3(\tilde{\phi})(t)\|_{H^{s-1}})
					\leq C\sup_{0\leq t\leq T'} \|\nabla \tilde{\phi}(t)\|_{H^{s-1}} \leq C, 
					$$
					and $F_\infty\in L^1([0,T'];H_{(\infty)}^s) \cap C([0,T'];H_{(\infty)}^{s-1})$ for a given $T'$, 
					Lemma \ref{lemA2-0} follows directly by \cite[Theorem I]{Kato1}. 
					
					%(We need $F_\infty^0\in  C([0,T'];H^{s-1})$ in the following lemmas and propositions?)  
					
					\BLACK 
					%\RED Local solution of this system without $P_\infty$? \BLACK 
					
					\vspace{2ex}
					
					At the beginning, we introduce the solvability of \eqref{eq:(2)}.
					
					\vspace{2ex}
					\subsection{Solvability of the system \eqref{eq:(2)}}
					We begin with the solvability of the following system for the density,
					\begin{equation}\label{111601}
						\left\{
						\begin{aligned}
							&\partial_ta_\infty+P_\infty(\tilde v\cdot\nabla a_\infty)=F_\infty^0,\\
							&a_\infty|_{t=0}=a_{0\infty}.	
						\end{aligned}	
						\right.
					\end{equation}

					After taking a same method developed in Lemma 6.2 in \cite{Kagei-Tsuda}, we obtain the following solvability of \eqref{111601}
					\begin{lem}\label{lemA2}
						Let $d\geq 3$ and $s$ be an integer satisfying $s\geq [\frac{d}{2}]+2$. Set $k=s-1$ or $s$. Assume that $\tilde v \in C([0,T'];H^s)\cap L^2(0,T';H^{s})$,
						$F_\infty^0\in L^2(0,T';H_{(\infty)}^k)\cap C([0,T'];H_{(\infty)}^{k-1})$ and $a_{0\infty}\in H^k_{(\infty)}$. Here $T'$ is a given positive number. Then \eqref{111601}  has a unique solution $a_\infty \in C([0,T'];H_{(\infty)}^k)$ and $a_\infty$ satisfies 
						\begin{align}
							\|a_\infty(t)\|_{H^k}^2\leq C\Big\{\|a_{0\infty}\|_{H^k}^2&+\int_0^t(\|\tilde v  (\tau)\|_{H^s}+\|\tilde v(\tau)\|_{H^s}^2)\|a_\infty(\tau)\|_{H^k}^2d\tau\nonumber\\
							&+\int_0^t\|F_\infty^0(\tau)\|_{H^k}\|a(\tau)\|_{H^k}d\tau\Big\},
						\end{align}
						and 
						\begin{align*}
							\|a_\infty(t)\|_{H^k}^2\leq Ce^{C\int_0^t(1+\|\tilde v(\tau)\|_{H^s}+\|\tilde v(\tau)\|_{H^s}^2)d\tau}\Big\{\|a_{0\infty}\|_{H^k}^2+\int_0^t\|F_\infty^0(\tau)\|_{H^k}^2d\tau\Big\},
						\end{align*}
						for $t\in [0,T']$.
					\end{lem}
					
					Next, we prove the solvability of the following system of the velocity,
					\begin{equation}\label{111701}
						\left\{
						\begin{aligned}
							&\partial_t v_\infty+v_\infty+P_{\infty}(\tilde{v}\cdot\nabla v_\infty)=F_\infty^1,\\
							&v_\infty|_{t=0}=v_{0\infty}.	
						\end{aligned}
						\right.
					\end{equation}
					\begin{lem}\label{lemA1}
						Let $d\geq 3$ and $s$ be an integer satisfying $s\geq [\frac{d}{2}]+2$. %Let $k=s-1$ or $s$. 
						Assume that $\tilde v \in C([0,T'];H^s)\cap L^2(0,T';H^{s})$, $F_\infty^1\in L^2(0,T';H_{(\infty)}^{s})\cap C([0,T'];H_{(\infty)}^{s-1})$ and $v_{0\infty} \in H_{(\infty)}^s$. Here $T'$ is  a given positive number. Then \eqref{111701} admits a unique solution $v_\infty\in C([0,T'];H_{(\infty)}^s)\cap L^2(0,T';H_{(\infty)}^{s})\cap H^1(0,T';H_{(\infty)}^{s-1})$ satisfying
						\begin{align}
							\|v_\infty(t)\|_{H^s}^2\leq C\Big\{\|v_{0\infty}\|_{H^s}^2&+\int_0^t\|\tilde v(\tau)\|_{H^s}^2\|v_\infty(\tau)\|_{H^s}^2d\tau \nonumber\\ &+\int_0^t\|F_\infty^1(\tau)\|_{H^s}^2d\tau\Big\},
						\end{align}
						and
						\begin{align}
							\|v_\infty(t)\|_{H^s}^2\leq Ce^{C\int_0^t\|\tilde v(\tau)\|_{H^s}d\tau}\Big\{\|v_{0\infty}\|_{H^s}^2+C\int_0^t\|F_\infty^1(\tau)\|_{H^s}^2d\tau\Big\}
						\end{align}
						for $t\in [0,T']$.
					\end{lem}
					\noindent\textbf{Proof.}
					Define $v_\infty^{(0)}=0$. Assume that $v_\infty^{(n)}(n\geq 1)$ satisfies
					\begin{equation*}
						\left\{
						\begin{aligned}
							&\partial_tv_\infty^{(n)}+v_\infty^{(n)}+\tilde{v}\cdot\nabla v_\infty^{(n)}
							=F_\infty^1+P_1(\tilde{v}\cdot\nabla v_\infty^{(n-1)}),\\
							&v_{\infty}^{(n)}|_{t=0}=v_{0\infty}.
						\end{aligned}
						\right.
					\end{equation*}
					Then we could derive the equation of $(v_\infty^{(n+1)}-v_\infty^{(n)})$ as follows,
					\begin{align*}
						&\partial_t(v_\infty^{(n+1)}-v_\infty^{(n)})+v_\infty^{(n+1)}-v_\infty^{(n)}+\tilde{v}\cdot\nabla (v_\infty^{(n+1)}-v_\infty^{(n)})\\
						&=P_1(\tilde{v}\cdot \nabla (v_\infty^{(n)}-v_\infty^{(n-1)})),
					\end{align*}
					and the initial data satisfy
					\begin{align*}
						(v_\infty^{(n+1)}-v_\infty^{(n)})|_{t=0}=0.
					\end{align*}
					Let  $s\geq [\frac{d}{2}]+2$. Taking $L^2$ energy estimates on the above equations and using H\"{o}lder inequality, we have
					\begin{align*}
						&\frac{1}{2}\frac{d}{dt}\|v_{\infty}^{(n+1)}-v_{\infty}^{(n)}\|_{H^s}^2+\|v_{\infty}^{(n+1)}-v_{\infty}^{(n)}\|_{H^s}^2\\
						&\leq C\|\tilde{v}\|_{H^s}\|v_{\infty}^{(n+1)}-v_{\infty}^{(n)}\|_{H^s}^2
						+C\|\tilde{v}\|_{H^s}\|v_{\infty}^{(n)}-v_{\infty}^{(n-1)}\|_{H^s}\|v_{\infty}^{(n+1)}-v_{\infty}^{(n)}\|_{H^s}\\
						&\leq C\|\tilde{v}\|_{H^s}\|v_{\infty}^{(n+1)}-v_{\infty}^{(n)}\|_{H^s}^2
						+C\|\tilde{v}\|_{H^s}^2\|v_{\infty}^{(n)}-v_{\infty}^{(n-1)}\|_{H^s}^2+\varepsilon\|v_{\infty}^{(n+1)}-v_{\infty}^{(n)}\|_{H^s}^2.
					\end{align*}
					Due to the smallness of $\varepsilon$, it holds that 
					\begin{align*}
						&\frac{d}{dt}\|v_{\infty}^{(n+1)}-v_{\infty}^{(n)}\|_{H^s}^2+\|v_{\infty}^{(n+1)}-v_{\infty}^{(n)}\|_{H^s}^2\\
						&\leq C\|\tilde{v}\|_{H^s}\|v_{\infty}^{(n+1)}-v_{\infty}^{(n)}\|_{H^s}^2
						+C\|\tilde{v}\|_{H^s}^2\|v_{\infty}^{(n)}-v_{\infty}^{(n-1)}\|_{H^s}^2.
					\end{align*}
					In terms of Gr\"{o}nwall's inequality, one has
					\begin{align*}
						\|v_\infty^{(n+1)}(t)-v_\infty^{(n)}(t)\|_{H^s}^2\leq M_0\frac{(M_1t)^{n+1}}{(n+1)!}\quad (n\geq 0),
					\end{align*}
					where 
					\begin{align}
						M_0=e^{C\int_0^{T'}\|\tilde v(t)\|_{H^s}dt}\Big\{\|v_{0\infty}\|_{H^s}^2+\int_0^{T'}\|F_\infty^1(t)\|_{H^s}^2dt\Big\},
					\end{align}
					and 
					\begin{align}
						M_1=C\|\tilde v\|_{C([0,T'];H^s)}^2e^{C\int_0^{T'}\|\tilde v(t)\|_{H^s}dt}.	
					\end{align}
					Therefore, one can see that $v_\infty^{(n)}$ converges in $C([0,T'];H^s)$ to a function $v_\infty\in C([0,T'];H^s)$ that satisfies 
					\begin{equation*}
						\left\{	
						\begin{aligned}
							&\partial_tv_\infty+v_\infty+\tilde{v}\cdot\nabla v_\infty=F_\infty^1+P_1(\tilde{v}\cdot\nabla v_\infty),\\
							&v_\infty|_{t=0}=v_{0\infty}.
						\end{aligned}
						\right.
					\end{equation*}
					Hence $v_\infty$ is a solution of \eqref{111701}.  Meanwhile, it holds that 
					\begin{align}
						\|v_\infty\|_{H^s}^2\leq Ce^{C\int_0^t\|\tilde v(\tau)\|_{H^s}d\tau}(\|v_{0\infty}\|_{H^s}^2+C\int_0^t\|F_\infty(\tau)\|_{H^s}^2d\tau).
					\end{align}
					
					Since the initial data $\supp \hat{v}_{0\infty}(t)\subset \{|\xi|\geq r_1\}$, according to the above estimates, we can prove $\supp \hat{v}_\infty(t)\subset \{|\xi|\geq r_1\}$ for $t\in [0,T']$. Therefore we complete the proof.
					$\hfill\square$
					
					With the Lemmas \ref{lemA1} and \ref{lemA2} in hands, we are able to prove the solvability of \eqref{Mainsystem}.
					\begin{prop}\label{solvabilityhighpart}
						Let $d\geq 3$ and $s$ be an integer satisfying $s\geq [\frac{d}{2}]+2.$  
						Assume that
						\begin{align}
							&\nabla\tilde{\phi}\in C([0,T'];H^{s-1})\cap L^2([0,T'];H^{s-1}),~\partial_t\tilde{\phi}\in C([0,T'];H^{s-1})\nonumber\\
							&\tilde{v}\in C([0,T'];H^{s})\cap L^{2}(0,T';H^s),\nonumber\\
							& 
							u_{0\infty}\in H^{s}_{(\infty)},\quad
							F_\infty=\trans{(F_\infty^0,F_\infty^1)}\in L^{2}(0,T';H^{s}_{(\infty)})\cap C([0,T'];H_{(\infty)}^{s-1}). 
						\end{align}
						Here $T'$ is a given positive number. 
						Then there exists a unique solution 
						$u_{\infty}=\trans{(a_{\infty},v_\infty)}$ of \eqref{Mainsystem} satisfying 
						\begin{align}
							a_{\infty}\in C([0,T'];H^{s}_{(\infty)}),v_\infty\in C([0,T'];H^{s}_{(\infty)})\cap L^{2}(0,T';H^{s}_{(\infty)})
							\cap H^1(0,T';H^{s-1}_{(\infty)}).
						\end{align}
					\end{prop}
					
					\noindent\textbf{Proof.}
					We define $u_\infty^{(n)}=\trans{(a_\infty^{(n)},v_\infty^{(n)})}, n=0,1,2,...$  For $n=0$, $v_\infty^{(0)}=0$ and $a_\infty^{(0)}$ is the solution of 
					\begin{equation}
						\left\{
						\begin{aligned}
							&\partial_ta_\infty^{(0)}+P_\infty(\tilde{v}\cdot\nabla a_\infty^{(0)})=F_\infty^0,\\
							&a_\infty^{(0)}|_{t=0}=a_{0\infty}.	
						\end{aligned}
						\right.
					\end{equation}
					For $n\ge 1$, $\trans{(a_\infty^{(n)}, v_\infty^{(n)})}$ is the solution of 
					\begin{equation}\label{91903}
						\left\{
						\begin{aligned}
							&\partial_ta_\infty^{(n)}+\text{div}v_\infty^{(n)}+\tilde {v}\cdot\nabla a_\infty^{(n)}=F_\infty^0+P_1(\tilde{v}\cdot\nabla a_\infty^{(n-1)}),\\
							&a_\infty^{(n)}|_{t=0}=a_{0\infty},
						\end{aligned}
						\right.
					\end{equation}
					and
					\begin{equation}
						\left\{
						\begin{aligned}
							&\partial_tv_\infty^{(n)}+\nabla a_\infty^{(n)}+v_\infty^{(n)}+\tilde{v}\cdot\nabla v_\infty^{(n)}+g_3(\tilde\phi)\nabla a_\infty^{(n)}\\
							&\quad =F_\infty^{1}+P_1(\tilde{v}\cdot\nabla v_\infty^{(n-1)})+P_1(g_3(\tilde\phi)\nabla v_\infty^{(n-1)}),\\
							&v_\infty^{(0)}|_{t=0}=v_{0\infty}.
						\end{aligned}
						\right.
					\end{equation}
					Then we can get the equation of $\trans{(a_\infty^{(n+1)}-a_\infty^{(n)},v_\infty^{(n+1)}-v_\infty^{(n)})}$ as follows,
					\begin{equation}
						\left\{
						\begin{aligned}
							&\partial_t(a_\infty^{(n+1)}-a_\infty^{(n)})+\text{div}(v_\infty^{(n+1)}-v_\infty^{(n)})+\tilde{v}\cdot\nabla(a_\infty^{(n+1)}-a_\infty^{(n)})=P_1(\tilde{v}\cdot\nabla(a_\infty^{(n)}-a_\infty^{(n-1)})),\nonumber\\
							&\partial_t(v_\infty^{(n+1)}-v_\infty^{(n)})+\nabla(a_\infty^{(n+1)}-a_\infty^{(n)})+
							v_\infty^{(n+1)}-v_\infty^{(n)}+\tilde{v}\cdot\nabla(v_\infty^{(n+1)}-v_\infty^{(n)})\nonumber\\
							&+g_3(\tilde\phi)\nabla(a_\infty^{(n+1)}-a_\infty^{(n)})=P_1(\tilde{v}\cdot\nabla(v_\infty^{(n)}-v_\infty^{(n-1)}))+P_1(g_3(\tilde\phi)\nabla(v_\infty^{(n)}-v_\infty^{(n-1)})),
						\end{aligned}
						\right.
					\end{equation}
					with the initial data
					\begin{align*}
						(a_\infty^{(n+1)}-a_\infty^{(n)},v_\infty^{(n+1)}-v_\infty^{(n)})|_{t=0}=(0,0).
					\end{align*}
					Then by $L^2$ energy method, one has 
					\begin{align*}
						\|(a_\infty^{(n+1)}-a_\infty^{(n)},v_\infty^{(n+1)}-v_\infty^{(n)})\|_{H^s}^2\leq M_3\frac{(M_2t)^{n+1}}{(n+1)!},   
					\end{align*}
					where 
					\begin{align}
						M_2=e^{C\int_0^{T'}(\|\nabla\tilde \phi(t)\|_{H^{s-1}}+\|\tilde v(t)\|_{H^s})dt}&\Big\{\|a_{0\infty}\|_{H^s}^2+\|v_{0\infty}\|_{H^s}^2\nonumber\\
						&+\int_0^{T'}(\|F_\infty^0(t)\|_{H^s}^2+\|F_\infty^1(t)\|_{H^s}^2)dt\Big\},
					\end{align}
					and 
					\begin{align}
						M_3=C(\|\nabla\tilde \phi\|_{C([0,T'];H^{s-1})}^2+\|\tilde v\|_{C([0,T'];H^s)}^2)e^{C\int_0^{T'}(\|\nabla\tilde \phi(t)\|_{H^{s-1}}+\|\tilde v(t)\|_{H^s})dt}.	
					\end{align}
					To this end, one can show that $u_\infty^{(n)}=\trans{(a_\infty^{(n)}, v_\infty^{(n)})}$ converges to a pair of function $u_\infty=\trans{(a_\infty,v_\infty)}$ in $C([0,T'];H_{(\infty)}^s)$. It is not difficult to see that
					$u_\infty=\trans{(a_\infty,v_\infty)}$ is a unique solution of \eqref{Mainsystem}. This completes the proof. 
					$\hfill\square$
					\vspace{2ex}
					\subsection{Estimates of operators $S_{{\infty},\tilde{u}}(t)$  and $\scr{S}_{\infty,\tilde{u}}(t)$ }
					We first consider the following equation
					\begin{equation}\label{Mainsystem}
						\left\{
						\begin{aligned}
							&\partial_ta_{\infty}+\text{div}v_\infty+\tilde{v}\cdot\nabla a_\infty=F_\infty^0,\\
							&\partial_tv_\infty+\nabla a_\infty+v_\infty+\tilde{v}\cdot\nabla v_\infty+ g_3(\tilde{\phi})\cdot\nabla a_\infty=F_\infty^1,\\
							&(a_\infty,v_\infty)|_{t=0}=(a_{0\infty},v_{0\infty}).
						\end{aligned}
						\right.
					\end{equation}
					
					For nonnegative integers $k$ and $\ell$, we define $|f|_{H_j^k}$ by 
					\begin{align*}
						|f|_{H_j^k}=\sum_{|\alpha|=0}^k\Big\||x|^j\partial_x^\alpha f \Big\|_{L^2}.
					\end{align*}
					For simplicity, we introduce the energy $E_{j}^k[u_{\infty}](t)$ and  dissipation $D_{j}^k[u_{\infty}](t)$ as follows 
					\begin{align}
						E_j^k[u_{\infty}](t)&=|a_{\infty}|_{H_j^k}^2+|v_{\infty}|_{H_j^k}^2+\frac{1}{2} \sum_{|\alpha|=0}^k\int_{\mathbb{R}^3}g_3(\tilde\phi)\big||x|^j\partial_x^\alpha (\eta_R a_\infty)\big|^2\mathrm{d}x,\label{121001}\\
						D_{j}^k[u_{\infty}](t)&=|\nabla a_\infty|_{H_j^{k-1}}^2+|v_\infty|_{H_j^{k}}^2.\label{121002}
					\end{align}
					In terms of the Poincar\'{e} type inequality for the high-frequency part of the solution,  {\rm i.e., } Lemma  \ref{lemPinftyweight} with that $\|g_3 (\tilde{\phi})\|_{L^\infty} \leq C$,  we have
					\begin{equation}\label{EUestimate}
						E_j^k[u_{\infty}](t)\leq CD_{j}^k[u_{\infty}](t),
					\end{equation}
					for some constant $C>0$.
					
					First, we establish the energy estimates of $E_0^s[u_{\infty}](t)$ and $D_0^s[u_{\infty}](t)$.
					\begin{prop}\label{E0} 
						Let $d\geq 3$ and $s$ be a nonnegative integer satisfying $s\geq [\frac{d}{2}]+2$. Assume that  
						\begin{align}
							& \nabla \tilde{\phi}\in C([0,T'];H^{s-1})\cap L^2([0,T'];H^{s-1}), 
							~\partial_t\tilde{\phi}\in C([0,T'];H^{s-1})\nonumber\\
							&\tilde{v}\in C([0,T];H^{s})\cap L^{2}(0,T;H^{s}),\nonumber\\
							& 
							u_{0\infty}=\trans(a_{0\infty},v_{0\infty})\in H_{(\infty)}^{s},\nonumber\\
							& 
							F_\infty=\trans(F_\infty^0,F_{\infty}^{1})
							\in L^{2}(0,T';H^{s})\cap C([0,T'];H^{s-1}). 
							\nonumber
						\end{align}
						Then the initial value problem \eqref{Mainsystem} exists a unique solution $u_\infty=\trans(a_{\infty},v_{\infty})$ in  $C([0,T'];H^s)\cap L^2(0,T;H^s)$. %(Comment; In \eqref{Mainsystem}, it does not have $P_1$ term; directly has $\tilde{v}\cdot \nabla a_\infty$, hence the solution does not belong to the high frequency part. )\BLACK 
						
						Furthermore, there exist a positive constant $\kappa_0$ such that the estimates 
						\begin{align}
							&\frac{\mathrm{d}}{\mathrm{d}t}E_0^s[u_\infty](t)+\kappa_0 D_0^s[u_\infty](t)\nonumber\\
							&\leq  \varepsilon \|a_\infty\|_{L^2}^2+  C(\|\partial_t\tilde\phi\|_{H^{s-1}}+\|\nabla\tilde{\phi}\|_{H^{s-1}}+\|\nabla\tilde \phi\|_{H^{s-1}}^2+\|\tilde{v}\|_{H^{s}}+\|\tilde v\|_{H^{s}}^2)\|u_\infty\|_{H^{s}}^2+C\|F_{\infty}\|^2_{H^s}
							\label{91302}
						\end{align}
						holds on $t\in (0,T')$ with $T'$  a given positive number, where $C>0$ is a constant independent of time.
					\end{prop}
					
					\noindent
					\textbf{Proof.} 
					The existence is due to \cite[Theorem I]{Kato1} as in the Proposition \ref{solvabilityhighpart}.
					We mainly focus the proof of \eqref{91302}. 
					
					Multiplying \eqref{Mainsystem} by $u_\infty$, and integrating  over $\mathbb{R}^d$ by parts then adding them together, we obtain 
					\begin{align}
						&\frac{1}{2}\frac{\mathrm{d}}{\mathrm{d}t}(\|a_\infty\|_{L^2}^2+\|v_\infty\|_{L^2}^2)+\|v_\infty\|_{L^2}^2\nonumber\\
						&=\int_{\mathbb{R}^3}(-\tilde v\cdot\nabla a_\infty a_\infty-\tilde{v}\cdot\nabla v_\infty\cdot v_\infty-g_3(\tilde\phi)\nabla a_\infty\cdot v_\infty +F_\infty^0 a_\infty+ F_\infty^1\cdot v_\infty)\mathrm{d}x.
					\end{align} 
					By H\"{o}lder inequality, integration by parts and Sobolev inequality, we obtain 
					\begin{align*}
						&	\frac{1}{2}\frac{\mathrm{d}}{\mathrm{d}t}(\|a_\infty\|_{L^2}^2+\|v_\infty\|_{L^2}^2)+\|v_\infty\|_{L^2}^2\nonumber\\
						&\leq C\Big\{\|\text{div}\tilde v\|_{L^\infty}\|a_\infty\|_{L^2}^2+\|\text{div}\tilde v\|_{L^\infty}\|v_\infty\|_{L^2}^2+\|g_3(\tilde\phi)\|_{L^\infty}\|\nabla a_\infty\|_{L^2}\|v_\infty\|_{L^2}\nonumber\\
						&+\|F_\infty^0\|_{L^2} \|a_\infty\|_{L^2}+\|F_\infty^1\|_{L^2}\|v_\infty\|_{L^2}\Big\}\nonumber\\
						&\leq C\Big\{\|\tilde v\|_{H^s}\|a_\infty\|_{L^2}^2+\|\tilde v\|_{H^s}\|v_\infty\|_{L^2}^2+\|\nabla\tilde \phi\|_{H^{s-1}}\|\nabla a_\infty\|_{L^2}\|v_\infty\|_{L^2}\Big\}\nonumber\\
						&+\varepsilon\| a_\infty\|_{L^2}^2+\varepsilon\|v_\infty\|_{L^2}^2+C\|F_\infty\|_{L^2}^2,
					\end{align*}
					where $\varepsilon>0$ is a sufficiently small constant. %\BLUE and we used Lemma  \ref{lemPinftyweight} to get $\|\nabla a_\infty\|_{L^2}^2$. \BLACK  
					
					Due to the smallness of $\varepsilon$, one has 
					\begin{align}
						&	\frac{\mathrm{d}}{\mathrm{d}t}(\|a_\infty\|_{L^2}^2+\|v_\infty\|_{L^2}^2)+\|v_\infty\|_{L^2}^2\nonumber\\
						&\leq C\Big\{\|\tilde v\|_{H^s}\|a_\infty\|_{L^2}^2+\|\tilde v\|_{H^s}\|v_\infty\|_{L^2}^2+\|\nabla\tilde \phi\|_{H^{s-1}}\|\nabla a_\infty\|_{L^2}\|v_\infty\|_{L^2}\Big\}\nonumber\\
						&+\varepsilon\| a_\infty\|_{L^2}^2+C\|F_\infty\|_{L^2}^2
						,\label{92402}
					\end{align}
					
					For $|\alpha|\geq 1$, applying $\partial_x^\alpha $ to \eqref{Mainsystem} and multiplying by $\partial_x^\alpha u_\infty$, and integrating over $\mathbb{R}^d$, then adding them together, we obtain 
					\begin{align}
						&	\frac{1}{2}\frac{\mathrm{d}}{\mathrm{d}t}(\|\partial_x^\alpha a_\infty\|_{L^2}^2+\|\partial_x^\alpha v_\infty\|_{L^2}^2)+\|\partial_x^\alpha v_\infty\|_{L^2}^2\nonumber\\
						&=-(\partial_x^\alpha(\tilde{v}\cdot \nabla a_\infty),\partial_x^\alpha a_\infty)-(\partial_x^\alpha(\tilde{v}\cdot\nabla v_\infty),\partial_x^\alpha v_\infty)\nonumber\\
						&-(\partial_x^\alpha(g_3(\tilde \phi)\cdot\nabla a_\infty),\partial_x^\alpha v_\infty)+(\partial_x^\alpha F_\infty^0,\partial_x^\alpha a_\infty)
						+(\partial_x^\alpha F_\infty^1,\partial_x^\alpha v_\infty).\label{601}
					\end{align}
					Noting that 
					\begin{align}
						&-(\partial_x^\alpha(\tilde{v}\cdot \nabla a_\infty),\partial_x^\alpha a_\infty)\nonumber\\
						&=-(\partial_x^\alpha(\tilde{v}\cdot \nabla a_\infty)-\tilde{v}\cdot\partial_x^\alpha\nabla a_\infty,\partial_x^\alpha a_\infty)-(\tilde{v}\cdot\partial_x^\alpha\nabla a_\infty,\partial_x^\alpha a_\infty) \nonumber\\
						&\leq C(\|\nabla\tilde{v}\|_{L^\infty}\|\partial_x^\alpha a_\infty\|_{L^2}+\|\nabla a_\infty\|_{L^\infty}\|\partial_x^\alpha \tilde{v}\|_{L^2})\|\partial_x^\alpha a_\infty\|_{L^2}+C\|\text{div}\tilde{v}\|_{L^\infty}\|\partial_x^\alpha a_\infty\|_{L^2}^2\nonumber\\
						&\leq C\|\tilde{v}\|_{H^s}\|\partial_x^\alpha a_\infty\|_{L^2}^2+C\|a_\infty\|_{H^s}\|\partial_x^\alpha\tilde{v}\|_{L^2}\|\partial_x^\alpha a_\infty\|_{L^2},\label{91301}
					\end{align}	
					
					\begin{align*}
						&-(\partial_x^\alpha(\tilde{v}\cdot\nabla v_\infty),\partial_x^\alpha v_\infty)\\
						&=-(\partial_x^\alpha(\tilde{v}\cdot\nabla v_\infty)-\tilde{v}\cdot\partial_x^\alpha\nabla v_\infty,\partial_x^\alpha v_\infty)-(\tilde{v}\cdot\partial_x^\alpha\nabla v_\infty,\partial_x^\alpha v_\infty)\\
						&\leq C(\|\nabla \tilde v\|_{L^\infty}\|\partial_x^\alpha v_\infty\|_{L^2}+\|\partial_x^\alpha\tilde v\|_{L^2}\|\nabla v_\infty\|_{L^\infty})\|\partial_x^\alpha v_\infty\|_{L^2}+C\|\text{div}\tilde v\|_{L^\infty}\|\partial_x^\alpha v_\infty\|_{L^2}^2,\\
						&\leq C\|\tilde v\|_{H^s}\|\partial_x^\alpha v_\infty\|_{L^2}^2+C\|\partial_x^\alpha\tilde{v}\|_{L^2}\|v_\infty\|_{H^s}\|\partial_x^\alpha v_\infty\|_{L^2},
					\end{align*}
					and
					\begin{align*}
						&-(\partial_x^\alpha(g_3(\tilde\phi)\cdot\nabla a_\infty),\partial_x^\alpha v_\infty)\\
						&=-(\partial_x^\alpha(g_3(\tilde\phi)\cdot\nabla a_\infty)-g_3(\tilde{\phi})\cdot\partial_x^\alpha\nabla a_\infty,\partial_x^\alpha v_\infty)-(g_3(\tilde{\phi})\cdot\partial_x^\alpha\nabla a_\infty,\partial_x^\alpha v_\infty)\\
						&\leq C(\|\nabla g_3(\tilde \phi)\|_{L^\infty}\|\partial_x^\alpha a_\infty\|_{L^2}+\|\partial_x^\alpha g_3(\tilde \phi)\|_{L^2}\|\nabla a_\infty\|_{L^\infty})\|\partial_x^\alpha v_\infty\|_{L^2}-(g_3(\tilde{\phi})\cdot\partial_x^\alpha\nabla a_\infty,\partial_x^\alpha v_\infty)\\
						&\leq C(\|\nabla \tilde \phi\|_{H^{s-1}}\|\partial_x^\alpha a_\infty\|_{L^2}+ \|\nabla  \tilde \phi\|_{H^{s-1}}  \|\nabla a_\infty\|_{L^\infty})\|\partial_x^\alpha v_\infty\|_{L^2}-(g_3(\tilde{\phi})\cdot\partial_x^\alpha\nabla a_\infty,\partial_x^\alpha v_\infty)\\
						&\leq C(\|\nabla\tilde \phi\|_{H^{s-1}}\|\partial_x^\alpha a_\infty\|_{L^2}+\|\partial_x^\alpha \tilde \phi\|_{L^2}\| a_\infty\|_{H^s})\|\partial_x^\alpha v_\infty\|_{L^2}-(g_3(\tilde{\phi})\cdot\partial_x^\alpha\nabla a_\infty,\partial_x^\alpha v_\infty).
					\end{align*}
					The last term on the right-hand side should be handled carefully. It is shown that 
					\begin{align}
						&-(g_3(\tilde{\phi})\cdot\partial_x^\alpha\nabla a_\infty,\partial_x^\alpha v_\infty)\nonumber\\
						&=(g_3(\tilde\phi)\partial_x^\alpha a_\infty,\partial_x^\alpha\text{div}v_\infty)+(\nabla g_3(\tilde\phi)\partial_x^\alpha a_\infty,\partial_x^\alpha v_\infty)\nonumber\\
						&=-(g_3(\tilde\phi)\partial_x^\alpha a_\infty,\partial_x^\alpha \partial_ta_\infty)-(g_3(\tilde\phi)\partial_x^\alpha a_\infty,\partial_x^\alpha(\tilde{v}\cdot\nabla a_\infty))+(\nabla g_3(\tilde\phi)\partial_x^\alpha a_\infty,\partial_x^\alpha v_\infty).\label{92401}
					\end{align}
					The first term on the right-hand side of \eqref{92401} is given by 
					\begin{align*}
						-(g_3(\tilde\phi)\partial_x^\alpha a_\infty,\partial_x^\alpha \partial_ta_\infty)
						&=-\frac{1}{2}\frac{d}{dt}\int_{\mathbb{R}^3}g_3(\tilde\phi)|\partial_x^\alpha a_\infty|^2dx+\frac{1}{2}\int_{\mathbb{R}^3}\partial_tg_3(\tilde\phi)|\partial_x^\alpha a_\infty|^2dx\\
						&\leq -\frac{1}{2}\frac{d}{dt}\int_{\mathbb{R}^3}g_3(\tilde\phi)|\partial_x^\alpha a_\infty|^2dx+C\|\partial_t\tilde\phi\|_{H^{s-1}}\|\partial_x^\alpha a_\infty\|_{L^2}^2.
					\end{align*}
					The second term on the right-hand side of \eqref{92401} is given by 
					\begin{align*}
						&-(g_3(\tilde\phi)\partial_x^\alpha a_\infty,\partial_x^\alpha(\tilde{v}\cdot\nabla a_\infty))\\
						&=-(g_3(\tilde\phi)\partial_x^\alpha a_\infty,\partial_x^\alpha(\tilde{v}\cdot\nabla a_\infty)-\tilde{v}\cdot\partial_x^\alpha\nabla a_\infty)-(g_3(\tilde\phi)\partial_x^\alpha a_\infty,\tilde{v}\cdot\partial_x^\alpha\nabla a_\infty)\\
						&\leq C\|g_3(\tilde\phi)\|_{L^\infty}\|\partial_x^\alpha a_\infty\|_{L^2}(\|\nabla \tilde{v}\|_{L^\infty}\|\partial_x^\alpha a_\infty\|_{L^2}+\|\partial_x^\alpha\tilde{v}\|_{L^2}\|\nabla a_\infty\|_{L^\infty})+C\|\text{div}(g_3(\tilde\phi)\tilde{v})\|_{L^\infty}\|\partial_x^\alpha a_\infty\|_{L^2}^2\\
						&\leq C\|\nabla\tilde{\phi}\|_{H^{s-1}}\|\partial_x^\alpha a_\infty\|_{L^2}(\|\tilde{v}\|_{H^s}\|\partial_x^\alpha a_\infty\|_{L^2}+\|\partial_x^\alpha\tilde{v}\|_{L^2}\|a_\infty\|_{H^s})+C\|\nabla\tilde{\phi}\|_{H^{s-1}}\|\tilde{v}\|_{H^s}\|\partial_x^\alpha a_\infty\|_{L^2}^2.
					\end{align*}
					The third term on the right-hand side of \eqref{92401} is given by 
					\begin{align*}
						(\nabla g_3(\tilde\phi)\partial_x^\alpha a_\infty,\partial_x^\alpha v_\infty)&\leq
						\|\nabla g_3(\tilde\phi)\|_{L^\infty}\|\partial_x^\alpha a_\infty\|_{L^2}\|\partial_x^\alpha v_\infty\|_{L^2}\\
						&\leq C\|\nabla\tilde\phi\|_{L^\infty}\|\partial_x^\alpha a_\infty\|_{L^2}\|\partial_x^\alpha v_\infty\|_{L^2}\\
						&\leq C \|\nabla\tilde\phi\|_{H^{s-1}}\|\partial_x^\alpha a_\infty\|_{L^2}\|\partial_x^\alpha v_\infty\|_{L^2}.
					\end{align*}
					It then follows from above estimates to show that 
					\begin{align*}
						&-(\partial_x^\alpha(g_3(\tilde\phi)\cdot\nabla a_\infty),\partial_x^\alpha v_\infty)\\
						&\leq C(\|\tilde a\|_{H^s}\|\partial_x^\alpha a_\infty\|_{L^2}+\|\partial_x^\alpha \tilde a\|_{L^2}\| a_\infty\|_{H^s})\|\partial_x^\alpha v_\infty\|_{L^2}\\
						&-\frac{1}{2}\frac{d}{dt}\int_{\mathbb{R}^3}g_3(\tilde\phi)|\partial_x^\alpha a_\infty|^2dx+C\|\partial_t\tilde\phi\|_{H^{s-1}}\|\partial_x^\alpha a_\infty\|_{L^2}^2\\
						&+C\|\nabla\tilde{\phi}\|_{H^{s-1}}\|\partial_x^\alpha a_\infty\|_{L^2}(\|\tilde{v}\|_{H^s}\|\partial_x^\alpha a_\infty\|_{L^2}+\|\partial_x^\alpha\tilde{v}\|_{L^2}\|a_\infty\|_{H^s})\\
						&+C \|\nabla\tilde\phi\|_{H^{s-1}}\|\partial_x^\alpha a_\infty\|_{L^2}\|\partial_x^\alpha v_\infty\|_{L^2}.
					\end{align*}
					Noting that 
					\begin{align*}
						&(\partial_x^\alpha F_\infty^0,\partial_x^\alpha a_\infty)+(\partial_x^\alpha F_\infty^1,\partial_x^\alpha v_\infty)\\
						&\leq \|\partial_x^\alpha F_\infty^0\|_{L^2}\|\partial_x^\alpha a_\infty\|_{L^2}+\|\partial_x^\alpha F_\infty^1\|_{L^2}\|\partial_x^\alpha v_\infty\|_{L^2}\\
						&\leq \varepsilon (\|\partial_x^\alpha a_\infty\|_{L^2}^2+\|\partial_x^\alpha v_\infty\|_{L^2}^2)+C\|\partial_x^\alpha F_\infty\|_{L^2}^2.
					\end{align*}
					Substituting the above estimates into \eqref{601} and using the smallness of $\varepsilon$, we are able to prove that
					\begin{align*}
						&	\frac{1}{2}\frac{\mathrm{d}}{\mathrm{d}t}(\|\partial_x^\alpha a_\infty\|_{L^2}^2+\|\partial_x^\alpha v_\infty\|_{L^2}^2)+\frac{1}{2}\|\partial_x^\alpha v_\infty\|_{L^2}^2+\frac{1}{2}\frac{d}{dt}\int_{\mathbb{R}^3}g_3(\tilde\phi)|\partial_x^\alpha a_\infty|^2dx\nonumber\\
						&\leq C\|\tilde{v}\|_{H^s}\|\partial_x^\alpha a_\infty\|_{L^2}^2+C\|a_\infty\|_{H^s}\|\partial_x^\alpha\tilde{v}\|_{L^2}\|\partial_x^\alpha a_\infty\|_{L^2}\\
						&+C(\|\tilde v\|_{H^s}\|\partial_x^\alpha v_\infty\|_{L^2}+\|\partial_x^\alpha\tilde v\|_{L^2}\|v_\infty\|_{H^s})\|\partial_x^\alpha v_\infty\|_{L^2}+C\|\tilde v\|_{H^s}\|\partial_x^\alpha v_\infty\|_{L^2}^2\\
						&+C(\|\nabla\tilde \phi\|_{H^{s-1}}\|\partial_x^\alpha a_\infty\|_{L^2}+\|\partial_x^\alpha \tilde a\|_{L^2}\| a_\infty\|_{H^s})\|\partial_x^\alpha v_\infty\|_{L^2}\\
						&+C\|\partial_t\tilde\phi\|_{H^{s-1}}\|\partial_x^\alpha a_\infty\|_{L^2}^2
						+C\|\nabla\tilde{\phi}\|_{H^{s-1}}\|\partial_x^\alpha a_\infty\|_{L^2}(\|\tilde{v}\|_{H^s}\|\partial_x^\alpha a_\infty\|_{L^2}+\|\partial_x^\alpha\tilde{v}\|_{L^2}\|a_\infty\|_{H^s})\\
						&+C \|\nabla\tilde\phi\|_{H^{s-1}}\|\partial_x^\alpha a_\infty\|_{L^2}\|\partial_x^\alpha v_\infty\|_{L^2}
						+\varepsilon\|\partial_x^\alpha a_\infty\|_{L^2}^2+C\|\partial_x^\alpha F_\infty\|_{L^2}^2.
					\end{align*}
					Summing the above inequality with respect to $|\alpha|$ from $1$ to $s$ and using \eqref{92402}, we get 
					\begin{align*}
						&	\frac{1}{2}\frac{\mathrm{d}}{\mathrm{d}t}(\| a_\infty\|_{H^s}^2+\| v_\infty\|_{H^s}^2)+\frac{1}{2}\| v_\infty\|_{H^s}^2+\frac{1}{2}\sum_{|\alpha|=0}^s\frac{d}{dt}\int_{\mathbb{R}^3}g_3(\tilde\phi)|\partial_x^\alpha a_\infty|^2dx\nonumber\\
						&\leq C(\|\nabla\tilde\phi\|_{H^{s-1}}+\|\nabla\tilde\phi\|_{H^{s-1}}^2+\|\tilde v\|_{H^s}+\|\tilde v\|_{H^s}^2)\|u_\infty\|_{H^s}^2\nonumber\\
						&\quad+C\|\partial_t\tilde\phi\|_{H^{s-1}}\|a_\infty\|_{H^s}^2
						+\varepsilon\|a_\infty\|_{H^{s-1}}^2+C\|F_\infty\|_{H^s}^2.
					\end{align*}
					In terms of the definition in \eqref{121001}, we prove that 
					\begin{align}
						&\frac{\mathrm{d}}{\mathrm{d}t}E_0^s[u_\infty](t)+\| v_\infty\|_{H^s}^2\nonumber\\
						&\leq C(\|\nabla\tilde\phi\|_{H^s}+\|\nabla\tilde\phi\|_{H^{s-1}}^2+\|\tilde v\|_{H^s}+\|\tilde v\|_{H^s}^2)\|u_\infty\|_{H^s}^2\nonumber\\
						&+C\|\partial_t\tilde\phi\|_{H^{s-1}}\|a_\infty\|_{H^s}^2+ \varepsilon\|a_\infty\|_{H^{s-1}}^2  +C\|F_\infty\|_{H^s}^2.
						\label{O3}
					\end{align}
					
					In the next section, we consider the dissipation estimates for the density $a_\infty$.
					
					For $|\alpha|\geq 1$, applying $\partial_x^{\alpha-1}$ to $\eqref{Mainsystem}_2$, then taking inner product with $\partial_x^\alpha a_\infty$ on $\mathbb{R}^d$, we have 
					\begin{align}
						&\frac{\mathrm{d}}{\mathrm{d}t}(\partial_x^{\alpha-1}v_\infty,\partial_x^\alpha a_\infty)+\|\partial_x^\alpha a_\infty\|_{L^2}^2\nonumber\\
						&=(\partial_x^{\alpha-1}v_\infty,\partial_x^\alpha\partial_ta_\infty)
						-(\partial_x^{\alpha-1}v_\infty,\partial_x^\alpha a_\infty)
						+(\partial_x^{\alpha}F_\infty^0,\partial_x^{\alpha-1} v_\infty)+(\partial_x^{\alpha-1}F_\infty^1,\partial_x^\alpha a_\infty)\nonumber\\
						&-(\partial_x^{\alpha-1}(\tilde{v}\cdot\nabla v_\infty),\partial_x^\alpha a_\infty)-(\partial_x^{\alpha-1}(g_3(\tilde\phi)\cdot\nabla a_\infty),\partial_x^\alpha a_\infty).
					\end{align}
					By $\eqref{Mainsystem}_1$ and integration by parts, one has
					\begin{align}
						&\frac{\mathrm{d}}{\mathrm{d}t}(\partial_x^{\alpha-1}v_\infty,\partial_x^\alpha a_\infty)+\|\partial_x^\alpha a_\infty\|_{L^2}^2\nonumber\\
						&=\|\partial_x^\alpha v_\infty\|_{L^2}^2-(\partial_x^{\alpha-1}v_\infty,\partial_x^\alpha(\tilde{v}\cdot\nabla a_\infty))\nonumber\\
						&\quad-(\partial_x^{\alpha-1}v_\infty,\partial_x^\alpha a_\infty)+(\partial_x^{\alpha}F_\infty^0,\partial_x^{\alpha-1} v_\infty)+(\partial_x^{\alpha-1}F_\infty^1,\partial_x^\alpha a_\infty)\nonumber\\
						&\quad-(\partial_x^{\alpha-1}(\tilde{v}\cdot\nabla v_\infty),\partial_x^\alpha a_\infty)-(\partial_x^{\alpha-1}(g_3(\tilde\phi)\cdot\nabla a_\infty),\partial_x^\alpha a_\infty).\label{120101}
					\end{align}
					Noting that 
					\begin{align*}
						(\partial_x^{\alpha-1} v_\infty,\partial_x^\alpha(\tilde v\cdot\nabla a_\infty))
						& =(\partial_x^{\alpha-1}v_\infty,\partial_x^\alpha(\tilde v\cdot\nabla a_\infty)-\tilde v\cdot\partial_x^\alpha\nabla a_\infty)+(\partial_x^{\alpha-1}v_\infty,\tilde v\cdot\partial_x^\alpha\nabla a_\infty)\\
						&\leq C\|\partial_x^{\alpha-1} v_\infty\|_{L^2}(\|\nabla \tilde v\|_{L^\infty}\|\partial_x^\alpha a_\infty\|_{L^2}+\|\partial_x^\alpha\tilde v \|_{L^2}\|\nabla a_\infty\|_{L^\infty})\\
						&+C(\|\partial_x^\alpha v_\infty\|_{L^2}\|\tilde v\|_{L^\infty}\|\partial_x^\alpha a_\infty\|_{L^2}+\|\partial_x^{\alpha-1}v_\infty\|_{L^2}\|\nabla \tilde v\|_{L^\infty}\|\partial_x^\alpha a_\infty\|_{L^2})\nonumber\\
						&\leq \varepsilon \|\partial_x^\alpha a_\infty\|_{L^2}^2+C\|\tilde v\|_{H^s}^2\|\partial_x^\alpha v_\infty\|_{L^2}^2+C\|a_\infty\|_{H^s}\|\partial_x^\alpha \tilde v\|_{L^2}\|\partial_x^\alpha v_\infty\|_{L^2},
					\end{align*}
					\begin{align*}
						&-(\partial_x^{\alpha-1}v_\infty,\partial_x^\alpha a_\infty)+(\partial_x^{\alpha}F_\infty^0,\partial_x^{\alpha-1} v_\infty)+(\partial_x^{\alpha-1}F_\infty^1,\partial_x^\alpha a_\infty)\\
						&\leq 2\varepsilon\|\partial_x^\alpha a_\infty\|_{L^2}^2+C\|\partial_x^{\alpha}F_\infty^0\|_{L^2}^2+C\|\partial_x^{\alpha-1}v_\infty\|_{L^2}^2+C\|\partial_x^{\alpha-1}F_\infty^1\|_{L^2}^2\\
						&\leq 2\varepsilon\|\partial_x^\alpha a_\infty\|_{L^2}^2+C(\|\partial_x^{\alpha}v_\infty\|_{L^2}^2+\|\partial_x^{\alpha}F_\infty\|_{L^2}^2),
					\end{align*}
					and
					\begin{align*}
						&-(\partial_x^{\alpha-1}(\tilde{v}\cdot\nabla v_\infty),\partial_x^\alpha a_\infty)\\
						&=(\partial_x^{\alpha}(\tilde{v}\cdot\nabla v_\infty),\partial_x^{\alpha-1} a_\infty)\\
						&=(\partial_x^{\alpha}(\tilde{v}\cdot\nabla v_\infty)-\tilde{v}\cdot\partial_x^\alpha\nabla v_\infty,\partial_x^{\alpha-1} a_\infty)+(\tilde{v}\cdot\partial_x^\alpha\nabla v_\infty,\partial_x^{\alpha-1}a_\infty)\\
						&\leq C(\|\nabla \tilde v\|_{L^\infty}\|\partial_x^\alpha v_\infty\|_{L^2}+\|\partial_x^\alpha \tilde v\|_{L^2}\|\nabla v_\infty\|_{L^\infty})\|\partial_x^{\alpha-1}a_\infty\|_{L^2}\\
						&\quad+C(\|\nabla \tilde v\|_{L^\infty}\|\partial_x^\alpha v_\infty\|_{L^2}\|\partial_x^{\alpha-1}a_\infty\|_{L^2}+\|\tilde v\|_{L^\infty}\|\partial_x^\alpha v_\infty\|_{L^2}\|\partial_x^\alpha a_\infty\|_{L^2})\\
						&\leq C(\|\tilde v\|_{H^s}\|\partial_x^\alpha v_\infty\|_{L^2}+\|\partial_x^\alpha \tilde v\|_{L^2}\|v_\infty\|_{H^s})\|\partial_x^{\alpha}a_\infty\|_{L^2}\\
						&\quad+C(\|\tilde v\|_{H^s}\|\partial_x^\alpha v_\infty\|_{L^2}\|\partial_x^{\alpha}a_\infty\|_{L^2}+\|\tilde v\|_{H^s}\|\partial_x^\alpha v_\infty\|_{L^2}\|\partial_x^\alpha a_\infty\|_{L^2})\\
						&\leq 2\varepsilon\|\partial_x^\alpha a_\infty\|_{L^2}^2+C(\|\tilde v\|_{H^s}^2\|\partial_x^\alpha v_\infty\|_{L^2}^2+\|v_\infty\|_{H^s}^2\|\partial_x^\alpha \tilde v\|_{L^2}^2).
					\end{align*}
					The remaining term on the right-hand side of \eqref{120101} is given by 
					\begin{align*}
						&-(\partial_x^{\alpha-1}(g_3(\tilde\phi)\cdot\nabla a_\infty),\partial_x^\alpha a_\infty)\\
						&=(\partial_x^{\alpha}(g_3(\tilde\phi)\cdot\nabla a_\infty),\partial_x^{\alpha-1} a_\infty)\\
						&=(\partial_x^{\alpha}(g_3(\tilde\phi)\cdot\nabla a_\infty)-g_3(\tilde\phi)\partial_x^\alpha \nabla a_\infty,\partial_x^{\alpha-1} a_\infty)+(g_3(\tilde\phi)\partial_x^\alpha \nabla a_\infty,\partial_x^{\alpha-1} a_\infty)\\
						&\leq C(\|\nabla g_3(\tilde\phi)\|_{L^\infty}\|\partial_x^\alpha a_\infty\|_{L^2}+\|\partial_x^\alpha g_3(\tilde\phi)\|_{L^2}\|\nabla a_\infty\|_{L^\infty})\|\partial_x^{\alpha-1}a_\infty\|_{L^2}\\
						&\quad+C\|\nabla g_3(\tilde\phi)\|_{L^\infty}\|\partial_x^\alpha a_\infty\|_{L^2}\|\partial_x^{\alpha-1}a_\infty\|_{L^2}+C\|g_3(\tilde\phi)\|_{L^\infty}\|\partial_x^\alpha a_\infty\|_{L^2}^2\\
						&\leq C(\|\nabla\tilde\phi\|_{H^{s-1}}\|\partial_x^\alpha a_\infty\|_{L^2}+\|\partial_x^\alpha\tilde \phi\|_{L^2}\| a_\infty\|_{H^s})\|\partial_x^{\alpha}a_\infty\|_{L^2}\\
						&\quad+C\|\nabla\tilde\phi\|_{H^{s-1}}\|\partial_x^\alpha a_\infty\|_{L^2}^2+C\|\nabla\tilde\phi\|_{H^{s-1}}\|\partial_x^\alpha a_\infty\|_{L^2}^2\\
						&\leq \varepsilon \|\partial_x^\alpha a_\infty\|_{L^2}^2+C\|\nabla\tilde\phi\|_{H^{s-1}}\|\partial_x^\alpha a_\infty\|_{L^2}^2+C\|a_\infty\|_{H^s}^2\|\partial_x^\alpha\tilde\phi\|_{L^2}^2.
					\end{align*}
					Combining the above estimates together, we have
					\begin{align}
						&\frac{\mathrm{d}}{\mathrm{d}t}(\partial_x^{\alpha-1}v_\infty,\partial_x^\alpha a_\infty)+\|\partial_x^\alpha a_\infty\|_{L^2}^2\nonumber\\
						&\leq 6\varepsilon \|\partial_x^\alpha a_\infty\|_{L^2}^2+ C(\|\tilde{v}\|_{H^s}^2\|\partial_x^\alpha v_\infty\|_{L^2}^2+\|a_\infty\|_{H^s}\|\partial_x^\alpha\tilde{v}\|_{L^2}\|\partial_x^\alpha v_\infty\|_{L^2}\nonumber\\
						&+\|\partial_x^\alpha F_\infty\|_{L^2}^2+\|v_\infty\|_{H^s}^2\|\partial_x^\alpha \tilde v\|_{L^2}^2+\|\nabla\tilde\phi\|_{H^{s-1}}\|\partial_x^\alpha a_\infty\|_{L^2}^2\nonumber\\
						&+\|a_\infty\|_{H^s}^2\|\partial_x^\alpha \tilde\phi\|_{L^2}^2)+C\|\partial_x^\alpha v_\infty\|_{L^2}^2.
					\end{align}
					Taking the summation of $|\alpha|$ from $1$ to $s$  and using the smallness of $\varepsilon$, one has
					\begin{align}
						&\frac{\mathrm{d}}{\mathrm{d}t}\sum_{1\leq |\alpha|\leq s}(\partial_x^{\alpha-1}v_\infty,\partial_x^\alpha a_\infty)+\frac{1}{2}\|\nabla a_\infty\|_{H^{s-1}}^2\nonumber\\
						&\leq  C(\|\tilde{v}\|_{H^s}^2\|\nabla v_\infty\|_{H^{s-1}}^2+\|a_\infty\|_{H^s}\|\nabla\tilde{v}\|_{H^{s-1}}\|\nabla v_\infty\|_{H^{s-1}}\nonumber\\
						&+\|\nabla F_\infty\|_{H^{s-1}}^2+\|v_\infty\|_{H^s}^2\|\nabla \tilde v\|_{H^{s-1}}^2+\|\nabla\tilde\phi\|_{H^{s-1}}\|\nabla a_\infty\|_{H^{s-1}}^2\nonumber\\
						&+\|a_\infty\|_{H^s}^2\|\nabla \tilde\phi\|_{H^{s-1}}^2)+C\|\nabla v_\infty\|_{H^{s-1}}^2. \label{O2}
					\end{align}
					Taking  $\beta_1\times$\eqref{O3}+\eqref{O2} with $\beta_1$ a sufficiently small positive constant, we obtain 
					\begin{align}
						& \frac{\mathrm{d}}{\mathrm{d}t}\Big\{E_0^s[u_\infty](t)+\beta_1\sum_{1\leq |\alpha|\leq s}(\partial_x^{\alpha-1}v_\infty,\partial_x^\alpha a_\infty)\Big\}
						+\frac{\beta_1}{2}\|\nabla a_\infty\|_{H^{s-1}}^2+\frac{1}{2}\| v_\infty\|_{H^s}^2\nonumber\\
						&\leq \varepsilon \|a_\infty\|_{L^2}^2+ C(\|\nabla\tilde\phi\|_{H^{s-1}}^2+\|\tilde v\|_{H^s}^2+\|\tilde v\|_{H^s})\|u_\infty\|_{H^s}^2+C\|\partial_t\tilde\phi\|_{H^{s-1}}\|a_\infty\|_{H^s}^2+C\varepsilon\|\nabla a_\infty\|_{H^{s-1}}^2\nonumber\\
						&+C\|F_\infty\|_{H^s}^2
						+C \beta_1\Big(\|\tilde{v}\|_{H^s}^2\|\nabla v_\infty\|_{H^{s-1}}^2+\|a_\infty\|_{H^s}\|\nabla\tilde{v}\|_{H^{s-1}}\|\nabla v_\infty\|_{H^{s-1}}\nonumber\\
						&+\|\nabla F_\infty\|_{H^{s-1}}^2+\|v_\infty\|_{H^s}^2\|\nabla \tilde v\|_{H^{s-1}}^2+\|\nabla\tilde\phi\|_{H^{s-1}}\|\nabla a_\infty\|_{H^{s-1}}^2\nonumber\\
						&+\|a_\infty\|_{H^s}^2\|\nabla \tilde\phi\|_{H^{s-1}}^2+\|\nabla v_\infty\|_{H^{s-1}}^2\Big).
						\label{O4}
					\end{align}
					Due to the smallness of $\beta_1$ and $\varepsilon$, there exists a positive constant $\bar{\kappa}_0>0$ such that 
					\begin{align}
						& \frac{\mathrm{d}}{\mathrm{d}t}E_0^s[u_\infty](t)
						+\bar{\kappa}_0D_0^s[u_\infty](t)\nonumber\\
						&\leq  \varepsilon \|a_\infty\|_{L^2}^2+  C(\|\nabla\tilde\phi\|_{H^{s-1}}+\|\nabla\tilde\phi\|_{H^{s-1}}^2+\|\tilde{v}\|_{H^s}+\|\tilde{v}\|_{H^s}^2)\|u_\infty\|_{H^s}^2\nonumber\\
						&+C\|\partial_t\tilde\phi\|_{H^{s-1}}\|a_\infty\|_{H^s}^2+C\|F_\infty^1\|_{H^s}^2\nonumber\\
						&\leq  C(\|\partial_t\tilde\phi\|_{H^{s-1}}+\|\nabla\tilde\phi\|_{H^{s-1}}+\|\nabla\tilde\phi\|_{H^{s-1}}^2+\|\tilde{v}\|_{H^s}+\|\tilde{v}\|_{H^s}^2)\|u_\infty\|_{H^s}^2+C\|F_\infty\|_{H^s}^2.
					\end{align}
					%Similarly, we also obtain the estimates for $k=s-1$, which gives that
					%\begin{align*}
					%	& \frac{\mathrm{d}}{\mathrm{d}t}E_0^{s-1}[u_\infty](t)
					%	+\tilde{\kappa}_0D_0^{s-1}[u_\infty](t)\nonumber\\
					%	&\leq  C(\|\partial_t\tilde\phi\|_{H^{s-1}}+\|\nabla\tilde\phi\|_{H^{s-1}}+\|\nabla\tilde\phi\|_{H^{s-1}}^2+\|\tilde{v}\|_{H^s}+\|\tilde{v}\|_{H^s}^2)\|u_\infty\|_{H^{s-1}}^2
					%	+C\|F_\infty\|_{H^{s-1}}^2. 
					%\end{align*}
					Let $\kappa_0=\max\{\bar{\kappa}_0,\tilde{\kappa}_0\}$. In terms of the definition of $E_0^s[u_\infty](t)$ and $D_0^s[u_\infty](t)$, we get \eqref{91302}. This completes the proof of this proposition.
					$\hfill\square$
					
					\vspace{2ex}
					
					In the next content, we aim to establish the estimates for $E_\ell^s[u_\infty](t)$ and $D_\ell^s[u_\infty](t)$ with $\ell\geq 1$. Let us introduce a cut-off function. We fix a nonincreasing function $\tilde{\eta}\in C([0,\infty))$ satisfying $0\leq \tilde{\eta}\leq 1$ and 
					\begin{align*}
						\tilde{\eta}=\left\{
						\begin{aligned}
							&1  &&(|r|\leq 1),\\
							&0  &&(|r|>2).
						\end{aligned}
						\right.
					\end{align*}
					Define $\eta_R\in C_0^\infty(\mathbb{R}^d)$ by $\eta_R(x)=\tilde{\eta}(\frac{|x|}{R})$. Observe that, for any multi-index $\alpha$ with $|\alpha|\geq 1$, it holds that 
					\begin{align*}
						\supp (\partial_x^\alpha \eta_{R}) \subset \{R\leq |x| \leq 2R \},\quad \Big||x|\partial_x^\alpha(\eta_R(x)) \Big|\leq C_\alpha R^{1-|\alpha|}, ~(x\in \mathbb{R}^d).
					\end{align*}
					We first prove the following proposition
					\begin{prop}
						Let $d\geq 3$ and $s$ be a nonnegative integer satisfying $s\geq [\frac{d}{2}]+2$ and let $\ell$ be an integer satisfying $\ell\geq 1$. Assume that 
						\begin{align*}
							u_{0\infty}=\trans{(a_{0\infty},v_{0\infty})}\in H^s,~F_\infty=\trans{(F_\infty^0,F_\infty^1)}\in L^2(0,T';H^s)\cap C([0,T'];H^{s-1}). 
						\end{align*}
						Here $T'$ is a given positive number. Assume also that $u_{\infty}=\trans{(a_{\infty},v_{\infty})}$ is the solution of system \eqref{Mainsystem}. If $\nabla\tilde{\phi}\in C([0,T];H^{s-1})\cap  L^{2}(0,T;H^{s-1})$, $\tilde{v}\in C([0,T'];H^s)\cap L^2(0,T';H^{s})$, $\partial_t\tilde{\phi}\in C([0,T];H^{s-1})$, and that $u_{\infty}$ satisfies 
						\begin{align*}
							a_{\infty} \in C([0,T'];H^s), v_{\infty}\in C([0,T'];H^s)\cap L^2(0,T';H^s),
						\end{align*}
						Then there exists a positive constant $\kappa_1>0$ such that the estimate 
						\begin{align}
							&\frac{\mathrm{d}}{\mathrm{d}t}E_{\ell}^s[\eta_R u_{\infty}](t)+\kappa_1D_{\ell}^s[\eta_R u_{\infty}](t)\nonumber\\
							&\leq \varepsilon\||x|^\ell\eta_Ra_\infty\|_{L^2}^2 + C(\|\partial_t\tilde\phi\|_{H^{s-1}}+\|\nabla\tilde{\phi}\|_{H^{s-1}}+\|\tilde{v}\|_{H^s}+\|\nabla\tilde{\phi}\|_{H^{s-1}}^2+\|\tilde{v}\|_{H^s}^2)|\eta_R u_{\infty}|_{H_{\ell}^s}^2\nonumber\\
							&+C(1+\|\nabla\tilde{\phi}\|_{H^{s-1}(N_R)}+\|\tilde{v}\|_{H^s(N_R)}+\|\nabla\tilde{\phi}\|_{H^{s-1}(N_R)}^2+\|\tilde{v}\|_{H^s(N_R)}^2)\|u_{\infty}\|_{H^s(N_R)}^2,\nonumber\\
							&+C|\eta_Ru_\infty|_{H_{\ell-1}^s}^2+C|\eta_R F_{\infty}|_{H_{\ell}^s}^2,
							\label{etaRU1}
						\end{align}
						holds on $(0,T')$. Here $C$ is a positive constant independent of $T'$, and $R\geq 1$; and $N_R$ denotes the set $N_R=\{ x\in \mathbb{R}^d; R\leq |x|\leq 2R\}$.
					\end{prop}
					\noindent\textbf{Proof.} 
					By multiplying $\eta_R$ to \eqref{Mainsystem}, we have
					\begin{equation}
						\left\{
						\begin{aligned}
							&\partial_t(\eta_Ra_\infty)+\tilde{v}\cdot \nabla(\eta_R a_\infty)+\text{div}(\eta_Rv_{\infty})=\eta_RF_\infty^0+K_1,\\
							&\partial_t(\eta_Rv_\infty)+\nabla(\eta_R a_{\infty})+\eta_{R} v_{\infty}+\tilde{v}\cdot\nabla(\eta_Rv_\infty)+g_3(\tilde\phi)\nabla(\eta_Ra_\infty)=\eta_R F_\infty^1+K_2,\label{Mainsystem1}
						\end{aligned}
						\right.
					\end{equation}
					where the nonlinear terms $K_1$ and $K_2$ are given by 
					\begin{align}
						&K_1=\tilde{v}\cdot\nabla\eta_R a_{\infty}+v_{\infty}\cdot\nabla\eta_R,\nonumber\\
						&K_2=a_{\infty}\nabla \eta_R+\tilde{v}\cdot\nabla\eta_Rv_\infty+g_3(\tilde\phi)\nabla\eta_Ra_\infty.
					\end{align}
					
					Multiplying \eqref{Mainsystem1} by $|x|^{2\ell}\eta_Ru_\infty$, and integrating  over $\mathbb{R}^d$ by parts then adding them together, we obtain 
					\begin{align*}
						&	\frac{1}{2}\frac{\mathrm{d}}{\mathrm{d}t}(\big\||x|^{\ell}\eta_Ra_\infty
						\big\|_{L^2}^2+\big\||x|^{\ell}\eta_Rv_\infty
						\big\|_{L^2}^2)+\big\||x|^{\ell}\eta_Rv_\infty
						\big\|_{L^2}^2\nonumber\\
						&=\int_{\mathbb{R}^3}\eta_Rv_\infty \nabla(|x|^{2\ell})\eta_Ra_\infty\mathrm{d}x
						-\int_{\mathbb{R}^3}\tilde{v}\cdot\nabla(\eta_Ra_\infty)|x|^{2\ell}\eta_Ra_\infty\mathrm{d}x
						\nonumber\\
						&+\int_{\mathbb{R}^3}(\tilde{v}\cdot\nabla\eta_Ra_\infty+\nabla\eta_Rv_\infty)\cdot |x|^{2\ell}\eta_Ra_\infty\mathrm{d}x+\int_{\mathbb{R}^3}(\nabla\eta_R a_\infty+\tilde{v}\cdot\nabla\eta_Rv_\infty+g_3(\tilde\phi)\nabla\eta_Ra_\infty)|x|^{2\ell}\eta_Rv_\infty\mathrm{d}x\nonumber\\
						&-\int_{\mathbb{R}^3}\tilde{v}\cdot\nabla(\eta_Rv_\infty)|x|^{2\ell}\eta_Rv_\infty \mathrm{d}x
						-\int_{\mathbb{R}^3}g_3(\tilde\phi)\nabla(\eta_Ra_\infty)|x|^{2\ell}\eta_Rv_\infty\mathrm{d}x\\ &+\int_{\mathbb{R}^3}(\eta_RF_\infty^0|x|^{2\ell}\eta_Ra_\infty+\eta_RF_\infty^1|x|^{2\ell}\eta_Rv_\infty)\mathrm{d}x\nonumber\\
						&:= I_1+I_2+I_3+I_4+I_5+I_6+I_7.
					\end{align*} 
					It is shown that 
					\begin{align*}
						I_1&\leq \||x|^{\ell-1}\eta_Ra_\infty\|_{L^2}\||x|^\ell\eta_Rv_\infty\|_{L^2}\\
						&\leq \varepsilon\||x|^\ell\eta_Rv_\infty\|_{L^2}^2+C\||x|^{\ell-1}\eta_Ra_\infty\|_{L^2}^2,
					\end{align*}
					\begin{align*}
						I_2&=-\frac{1}{2}\int_{\mathbb{R}^3}\tilde{v} \cdot|x|^{2\ell}\nabla(|\eta_Ra_\infty|^2)\mathrm{d}x\\
						&=\frac{1}{2}\int_{\mathbb{R}^3}\text{div}\tilde{v} |x|^{2\ell}(\eta_Ra_\infty)^2\mathrm{d}x+\frac{1}{2}\int_{\mathbb{R}^3}\tilde{v}\cdot\nabla(|x|^{2\ell})|\eta_Ra_\infty|^2\mathrm{d}x\\
						&\leq C\|\text{div}\tilde v\|_{L^\infty}\||x|^\ell\eta_Ra_\infty\|_{L^2}^2+C\|\tilde v\|_{L^\infty}\||x|^\ell\eta_Ra_\infty\|_{L^2}\||x|^{\ell-1}\eta_Ra_\infty\|_{L^2}\\
						&\leq C\|\tilde v\|_{H^s}\||x|^\ell\eta_Ra_\infty\|_{L^2}^2+C\|\tilde v\|_{H^s}^2\||x|^{\ell}\eta_Ra_\infty\|_{L^2}^2+C\||x|^{\ell-1}\eta_Ra_\infty\|_{L^2}^2,
					\end{align*}
					\begin{align*}
						I_3&\leq \|\tilde{v}\|_{L^\infty(N_R)}\|\nabla\eta_Ra_\infty\|_{L^2(N_R)}\||x|^\ell\eta_Ra_\infty\|_{L^2(N_R)}
						\\
						&+\|\nabla\eta_R\|_{L^\infty(N_R)})\|v_\infty\|_{L^2(N_R)}\||x|^\ell\eta_Ra_\infty\|_{L^2(N_R)}\\
						&\leq C\|\tilde{v}\|_{H^s(N_R)}\|a_\infty\|_{L^2(N_R)}^2+C\|a_\infty\|_{L^2(N_R)}\|v_\infty\|_{L^2(N_R)}\\
						&\leq C(1+\|\tilde{v}\|_{H^s(N_R)})\|u_\infty\|_{L^2(N_R)}^2,
					\end{align*}
					\begin{align*}
						I_4&\leq (\|\nabla\eta_R\|_{L^\infty(N_R)}\|a_\infty\|_{L^2(N_R)}+\|\tilde v\|_{L^\infty(N_R)}\|\nabla\eta_R\|_{L^\infty(N_R)}\|v_\infty\|_{L^2(N_R)}\\
						&+\|g_3(\tilde\phi)\|_{L^\infty(N_R)}\|\nabla\eta_R\|_{L^\infty(N_R)}\|a_\infty\|_{L^2(N_R)})\||x|^\ell \eta_Rv_\infty\|_{L^2(N_R)}\\
						&\leq C(1+\|\tilde v\|_{H^s(N_R)}+\|\nabla\tilde \phi\|_{H^{s-1}(N_R)})(\|a_\infty\|_{L^2(N_R)}^2+\|v_\infty\|_{L^2(N_R)}^2)\\
						&\leq C(1+\|\tilde v\|_{H^s(N_R)}+\|\nabla\tilde \phi\|_{H^{s-1}(N_R)})\|u_\infty\|_{L^2(N_R)}^2,
					\end{align*}
					\begin{align*}
						I_5&=-\frac{1}{2}\int_{\mathbb{R}^3}\tilde{v} \cdot|x|^{2\ell}\nabla(|(\eta_Rv_\infty)|^2)\mathrm{d}x\\
						&=\frac{1}{2}\int_{\mathbb{R}^3}\text{div}\tilde{v}|x|^{2\ell}(\eta_Rv_\infty)^2\mathrm{d}x+\frac{1}{2}\int_{\mathbb{R}^3}\tilde{v}\cdot\nabla(|x|^{2\ell})|\eta_Rv_\infty|^2\mathrm{d}x\\
						&\leq C\|\text{div}\tilde v\|_{L^\infty}\||x|^\ell\eta_Rv_\infty\|_{L^2}^2+C\|\tilde v\|_{L^\infty}\||x|^\ell\eta_Rv_\infty\|_{L^2}\||x|^{\ell-1}\eta_Rv_\infty\|_{L^2}^2\\
						&\leq C\|\tilde v\|_{H^s}\||x|^\ell\eta_Rv_\infty\|_{L^2}^2+C\|\tilde v\|_{H^s}^2\||x|^{\ell}\eta_Rv_\infty\|_{L^2}^2+C\||x|^{\ell-1}\eta_Rv_\infty\|_{L^2}^2,
					\end{align*}
					\begin{align*}
						I_6&\leq \|g_3(\tilde\phi)\|_{L^\infty}\||x|^\ell\nabla(\eta_R a_\infty)\|_{L^2}\||x|^\ell\eta_Rv_\infty\|_{L^2}\\
						&\leq C\|\nabla\tilde\phi\|_{H^{s-1}}	\||x|^\ell\nabla(\eta_Ra_\infty)\|_{L^2}\||x|^\ell\eta_Rv_\infty\|_{L^2}\\
						&\leq \varepsilon\||x|^\ell\eta_Rv_\infty\|_{L^2}^2+C\|\nabla\tilde\phi\|_{H^{s-1}}^2	\||x|^\ell\nabla(\eta_Ra_\infty)\|_{L^2}^2,
					\end{align*}
					\begin{align*}
						I_7&\leq \||x|^\ell\eta_RF_\infty^0\|_{L^2}\||x|^\ell\eta_Ra_\infty\|_{L^2}+\||x|^\ell\eta_RF_\infty^1\|_{L^2}\||x|^\ell\eta_Rv_\infty\|_{L^2}\\
						&\leq \varepsilon(\||x|^\ell\eta_Ra_\infty\|_{L^2}^2+\||x|^\ell\eta_Rv_\infty\|_{L^2}^2)+C\||x|^\ell\eta_RF_\infty^0\|_{L^2}^2+C\||x|^\ell\eta_RF_\infty^1\|_{L^2}^2\\
						&\leq \varepsilon(\||x|^\ell\eta_Ra_\infty\|_{L^2}^2+\||x|^\ell\eta_Rv_\infty\|_{L^2}^2)+C\||x|^\ell\eta_RF_\infty\|_{L^2}^2.
					\end{align*}
					Combining the estimates of $I_1\sim I_7$ together, one has 
					\begin{align}
						&	\frac{1}{2}\frac{\mathrm{d}}{\mathrm{d}t}(\big\||x|^{\ell}\eta_Ra_\infty
						\big\|_{L^2}^2+\big\||x|^{\ell}\eta_Rv_\infty
						\big\|_{L^2}^2)+\big\||x|^{\ell}\eta_Rv_\infty
						\big\|_{L^2}^2\nonumber\\
						&\leq 3\varepsilon\||x|^\ell\eta_R v_\infty\|_{L^2}^2 + \varepsilon\||x|^\ell\eta_Ra_\infty\|_{L^2}^2  +C(\|\tilde{v}\|_{H^s}+\|\tilde{v}\|_{H^s}^2)\||x|^\ell\eta_Ru_\infty\|_{L^2}^2+C\||x|^{\ell-1}\eta_Ru_\infty\|_{L^2}^2\nonumber\\
						&+C\|\nabla\tilde\phi\|_{H^{s-1}}^2\||x|^\ell\nabla\eta_Ra_\infty)\|_{L^2}^2+C(1+\|\nabla\tilde\phi\|_{H^{s-1}(N_R)}+\|\tilde v\|_{H^s(N_R)})\|u_\infty\|_{L^2(N_R)}^2\nonumber\\
						&+C\||x|^\ell\eta_RF_\infty\|_{L^2}^2.
					\end{align} 
					It then follows from above and the smallness of $\varepsilon$ to prove that
					\begin{align}\label{102801}
						&\frac{\mathrm{d}}{\mathrm{d}t}(\big\||x|^{\ell}\eta_Ra_\infty
						\big\|_{L^2}^2+\big\||x|^{\ell}\eta_Rv_\infty
						\big\|_{L^2}^2)+\big\||x|^{\ell}\eta_Rv_\infty
						\big\|_{L^2}^2\nonumber\\
						&\leq  \varepsilon\||x|^\ell \eta_Ra_\infty\|_{L^2}^2 + C(\|\tilde{v}\|_{H^s}+\|\tilde{v}\|_{H^s}^2)\||x|^\ell\eta_Ru_\infty\|_{L^2}^2+C\||x|^{\ell-1}\eta_Ru_\infty\|_{L^2}^2\nonumber\\
						&+C\|\nabla\tilde\phi\|_{H^{s-1}}^2\||x|^\ell\nabla(\eta_Ra_\infty)\|_{L^2}^2+C(1+\|\nabla\tilde\phi\|_{H^{s-1}(N_R)}+\|\tilde v\|_{H^s(N_R)})\|u_\infty\|_{L^2(N_R)}^2\nonumber\\
						&+C\||x|^\ell\eta_RF_\infty\|_{L^2}^2.
					\end{align} 
					For a multi-index $\alpha$ satisfying $1\leq |\alpha|\leq s$, we take the inner product of $\partial_x^\alpha \eqref{Mainsystem1}_1$ with $|x|^{2\ell}\partial_x^\alpha (\eta_Ra_\infty)$ to prove 
					\begin{align}
						&\frac{1}{2}\frac{\mathrm{d}}{\mathrm{d}t}\big\||x|^{\ell}\partial_x^\alpha(\eta_Ra_\infty)
						\big\|_{L^2}^2+(\partial_x^\alpha\text{div}(\eta_Rv_\infty),|x|^{2\ell}\partial_x^\alpha (\eta_Ra_\infty))\nonumber\\
						&=\frac{1}{2}(\text{div}\tilde{v},|x|^{2\ell}|\partial_x^\alpha(\eta_Ra_{\infty})|^2)-([\partial_x^\alpha,\tilde{v}]\cdot \nabla(\eta_Ra_{\infty}),|x|^{2\ell}\partial_x^\alpha(\eta_R a_{\infty}))\nonumber\\
						&+\frac{1}{2}(\tilde{v}\cdot\nabla(|x|^{2\ell}),|\partial_x^\alpha (\eta_R a_{\infty})|^2)+(\partial_x^\alpha K_1,|x|^{2\ell}\partial_x^\alpha(\eta_R a_{\infty}))\nonumber\\
						&+(\partial_x^\alpha(\eta_R F_{\infty}^0),|x|^{2\ell}\partial_x^\alpha(\eta_Ra_{\infty}))\nonumber\\
						&:= J_1^\alpha+J_2^\alpha+J_3^\alpha+J_4^\alpha+J_5^\alpha,\label{M1}		
					\end{align}
					where we have used the fact that 
					\begin{align*}
						&(\partial_x^\alpha(\tilde{v}\cdot \nabla(\eta_R a_{\infty})),|x|^{2\ell}\partial_x^\alpha(\eta_R a_{\infty}))\\
						&=(\tilde{v}\cdot\nabla \partial_x^\alpha(\eta_R a_{\infty}),|x|^{2\ell}\partial_x^\alpha(\eta_R a_{\infty}))+([\partial_x^\alpha,\tilde{v}]\cdot\nabla(\eta_Ra_{\infty}),|x|^{2\ell}\partial_x^\alpha(\eta_Ra_{\infty}))\\
						&=\frac{1}{2}(|x|^{2\ell}\tilde{v},\nabla|\partial_x^\alpha(\eta_R a_{\infty})|^2)+([\partial_x^\alpha,\tilde{v}]\cdot\nabla(\eta_Ra_{\infty}),|x|^{2\ell}\partial_x^\alpha(\eta_R a_{\infty}))\\
						&=-\frac{1}{2}(\text{div}\tilde{v},|x|^{2\ell}|\partial_x^\alpha(\eta_Ra_{\infty})|^2)-\frac{1}{2}(\tilde{v}\cdot\nabla(|x|^{2\ell}),|\partial_x^\alpha(\eta_Ra_{\infty})|^2)\\
						&~~+([\partial_x^\alpha,\tilde{v}]\cdot\nabla(\eta_Ra_{\infty}),|x|^{2\ell}\partial_x^\alpha(\eta_Ra_{\infty})).
					\end{align*}
					This calculation can be justified by using the standard Friedrichs commutator argument.

					Then we estimate $J_1^\alpha\sim J_4^\alpha$ as follows,
					\begin{align*}
						\sum_{|\alpha|=1}^sJ_1^\alpha\leq C\sum_{|\alpha|=1}^s\|\text{div}\tilde{v}\|_{L^\infty} \big\||x|^{\ell}\partial_x^\alpha(\eta_Ra_\infty)
						\big\|_{L^2}^2\leq C\|\tilde{v}\|_{H^s}|\nabla(\eta_Ra_\infty)|_{H_{\ell}^{s-1}}^2,
					\end{align*}
					\begin{align*}
						\sum_{|\alpha|=1}^sJ_2^\alpha&\leq C\|\tilde{v}\|_{H^s}|\nabla(\eta_R a_\infty)|_{H_\ell^{s-1}}^2,
					\end{align*}
					\begin{align*}
						\sum_{|\alpha|=1}^sJ_3^\alpha&\leq C\sum_{|\alpha|=1}^s\|\tilde{v}\|_{L^\infty}\||x|^\ell\partial_x^\alpha (\eta_Ra_\infty)\|_{L^2}\||x|^{\ell-1}\partial_x^\alpha(\eta_Ra_\infty)\|_{L^2}\\
						&\leq C\|\tilde{v}\|_{H^s}|\nabla (\eta_Ra_\infty)|_{H_\ell^{s-1}}|\nabla(\eta_Ra_\infty)|_{H_{\ell-1}^{s-1}}\\
						&\leq C |\nabla (\eta_Ra_\infty)|_{H_{\ell-1}^{s-1}}^2+
						C \|\tilde{v}\|_{H^s}^2|\nabla(\eta_Ra_\infty)|_{H_{\ell}^{s-1}}^2,
					\end{align*}
					\begin{align*}
						\sum_{|\alpha|=1}^sJ_4^\alpha&=\sum_{|\alpha|=1}^s(\partial_x^\alpha(\tilde v\cdot\nabla\eta_Ra_\infty+v_\infty\cdot\nabla\eta_R),|x|^{2\ell}\partial_x^\alpha(\eta_Rv_\infty))\\
						&\leq C(\|\tilde v\|_{H^s(N_R)}\|\nabla\eta_R\|_{H^s(N_R)}\|a_\infty\|_{H^s(N_R)}+\|v_\infty\|_{H^s(N_R)}\|\nabla\eta_R\|_{H^s(N_R)})\|\nabla(\eta_Rv_\infty)\|_{H^{s-1}(N_R)}\\
						&\leq C (\|\tilde v\|_{H^s(N_R)}\|a_\infty\|_{H^s(N_R)}+\|v_\infty\|_{H^s(N_R)})\|v_\infty\|_{H^s(N_R)},
					\end{align*}
					\begin{align*}
						\sum_{|\alpha|=1}^s	J_5^\alpha&\leq \sum_{|\alpha|=1}^s	\||x|^\ell\partial_x^\alpha(\eta_RF_\infty^0)\|_{L^2}\||x|^{\ell}\partial_x^\alpha(\eta_Ra_\infty)\|_{L^2}\\
						&\leq C \varepsilon |\nabla(\eta_Ra_\infty)|_{H_\ell^{s-1}}^2+C|\nabla(\eta_RF_\infty^0)|_{H_\ell^{s-1}}^2.
					\end{align*}
					Summing $|\alpha|$ from $1$ to $s$, we are able to show that 
					\begin{align}
						&\frac{1}{2}\frac{\mathrm{d}}{\mathrm{d}t}\big|\nabla(\eta_Ra_\infty)|_{H_\ell^{s-1}}^2+\sum_{|\alpha|=1}^s(\partial_x^\alpha\text{div}(\eta_Rv_\infty),|x|^{2\ell}\partial_x^\alpha (\eta_Ra_\infty))\nonumber\\
						&\leq C(\|\tilde v\|_{H^s}+\|\tilde v\|_{H^s}^2)|\nabla(\eta_Ra_\infty)|_{H_\ell^{s-1}}^2+C|\nabla(\eta_Ra_\infty)|_{H_{\ell-1}^{s-1}}^2\nonumber\\
						&+C\varepsilon |\nabla(\eta_Ra_\infty)|_{H_\ell^{s-1}}^2+C\|\tilde v\|_{H^s(N_R)}(\|a_\infty\|_{H^s(N_R)}^2+\|v_\infty\|_{H^s(N_R)}^2)+C\|v_\infty\|_{H^s(N_R)}^2.\label{101002}
					\end{align}
					We take the inner product of $\partial_x^\alpha \eqref{Mainsystem1}_2$ with $|x|^{2\ell}\partial_x^\alpha(\eta_Rv_{\infty})$ and integrate by parts to produce 
					\begin{align}
						&\frac{1}{2}\frac{\mathrm{d}}{\mathrm{d}t}\big\| 
						|x|^\ell\partial_x^\alpha(\eta_Rv_{\infty})\big\|_{L^2}^2-(\partial_x^\alpha(\eta_Ra_{\infty}),|x|^{2\ell}\partial_x^\alpha\text{div}(\eta_Rv_{\infty}))+\big\||x|^\ell\partial_x^\alpha(\eta_Rv_{\infty})\big\|_{L^2}^2\nonumber\\
						&=
						(\partial_x^\alpha(\eta_Ra_{\infty}),\nabla(|x|^{2\ell})\partial_x^\alpha(\eta_Rv_{\infty}))
						+(\partial_x^\alpha K_2,|x|^{2\ell}\partial_x^\alpha(\eta_Rv_{\infty}))\nonumber\\
						&-(\partial_x^\alpha(\tilde{v}\cdot\nabla(\eta_Rv_\infty)),|x|^{2\ell}\partial_x^\alpha(\eta_Rv_\infty))-(\partial_x^\alpha(g_3(\tilde\phi)\nabla(\eta_Ra_\infty)),|x|^{2\ell}\partial_x^\alpha(\eta_Rv_\infty))
						\nonumber\\
						&+(\partial_x^\alpha(\eta_R F_{\infty}^1),|x|^{2\ell}\partial_x^\alpha(\eta_Rv_{\infty}))\nonumber\\
						&:= J_6^\alpha+J_7^\alpha+J_8^\alpha+J_9^\alpha+J_{10}^\alpha.\label{M4}
					\end{align}
					It is directly proved that 
					\begin{align*}
						\sum_{|\alpha|=1}^sJ_6^\alpha
						&\leq C \sum_{|\alpha|=1}^s \||x|^{\ell-1}\partial_x^\alpha(\eta_Ra_\infty)\|_{L^2}\||x|^{\ell}\partial_x^\alpha(\eta_Rv_\infty)\|_{L^2}\\
						&\leq 	\varepsilon |\nabla(\eta_Rv_\infty)|_{H_\ell^{s-1}}^2+C|\nabla(\eta_Ra_\infty)|_{H_{\ell-1}^{s-1}}^2.
					\end{align*}
					For $J_7^\alpha$, we consider it below,
					\begin{align*}
						\sum_{|\alpha|=1}^sJ_7^\alpha&=\sum_{|\alpha|=1}^s(\partial_x^\alpha(\nabla\eta_R a_\infty+\tilde{v}\cdot\nabla\eta_Rv_\infty+g_3(\tilde\phi)\nabla\eta_R a_\infty),|x|^{2\ell}\partial_x^\alpha(\eta_Rv_\infty))\\
						&\leq C(\|\nabla\eta_R\|_{H^{s}(N_R)}\|a_\infty\|_{H^s(N_R)}+\|\tilde v\|_{H^s(N_R)}\|\nabla\eta_R\|_{H^s(N_R)}\|v_\infty\|_{H^s(N_R)}\\
						&+\|g_3(\tilde\phi)\|_{H^s(N_R)}\|\nabla\eta_R\|_{H^s(N_R)}\|a_\infty\|_{H^s(N_R)})\| \nabla(\eta_Rv_\infty)\|_{H^{s-1}(N_R)}\\
						&\leq C(\|a_\infty\|_{H^s(N_R)}+\|\tilde v\|_{H^s(N_R)}\|v_\infty\|_{H^s(N_R)}+\|\nabla\tilde\phi\|_{H^{s-1}(N_R)}\|a_\infty\|_{H^s(N_R)})\|v_\infty\|_{H^s(N_R)}\\
						&\leq C(1+\|\nabla\tilde\phi\|_{H^{s-1}(N_R)}+\|\nabla\tilde v\|_{H^{s-1}(N_R)})\|u_\infty\|_{H^s(N_R)}^2.
					\end{align*}
					For $J_8^\alpha$, it holds that 
					\begin{align*}
						\sum_{|\alpha|=1}^sJ_8^\alpha&=-(\partial_x^\alpha(\tilde{v}\cdot\nabla(\eta_Rv_\infty))-\tilde{v}\cdot\partial_x^\alpha\nabla(\eta_R v_\infty),|x|^{2\ell}\partial_x^\alpha(\eta_Rv_\infty))-(\tilde{v}\cdot\partial_x^\alpha\nabla(\eta_R v_\infty),|x|^{2\ell}\partial_x^\alpha(\eta_Rv_\infty))\\
						& = (\partial_x^\alpha(\tilde{v}\cdot\nabla(\eta_Rv_\infty))-\tilde{v}\cdot\partial_x^\alpha\nabla(\eta_R v_\infty),|x|^{2\ell}\partial_x^\alpha(\eta_Rv_\infty)) \\
						&\quad+(\div \tilde{v}, |x|^{2\ell}|\partial_x^\alpha(\eta_Rv_\infty)|^2) +(\tilde{v}, \nabla(|x|^{2\ell})|\partial_x^\alpha(\eta_Rv_\infty)|^2)  \\ 
						&\leq C\|\tilde v\|_{H^s}|\nabla(\eta_Rv_\infty)|_{H_\ell^{s-1}}^2+C\|\tilde v\|_{H^s}|\nabla(\eta_Rv_\infty)|_{H_\ell^{s-1}}|\nabla(\eta_Rv_\infty)|_{H_{\ell-1}^{s-1}}\\
						&\leq C(\|\tilde v\|_{H^s}+\|\tilde v\|_{H^s}^2)|\nabla(\eta_Rv_\infty)|_{H_\ell^{s-1}}^2+C|\nabla(\eta_Rv_\infty)|_{H_{\ell-1}^{s-1}}^2.
					\end{align*}
					For $J_9^\alpha$, we have
					\begin{align*}
						J_9^\alpha&=-(\partial_x^\alpha(g_3(\tilde\phi)\nabla(\eta_R a_\infty))-g_3(\tilde\phi)\partial_x^\alpha(\nabla(\eta_R a_\infty)),|x|^{2\ell}\partial_x^\alpha(\eta_R v_\infty))\\
						&-(g_3(\tilde\phi)\partial_x^\alpha(\nabla(\eta_R a_\infty)),|x|^{2\ell}\partial_x^\alpha(\eta_Rv_\infty))\\
						&:= J_{91}^\alpha+J_{92}^\alpha.
					\end{align*}
					The first term $J_{91}^\alpha$ is estimated as follows,
					\begin{align}\label{103001}
						\sum_{|\alpha|=1}^sJ_{91}^\alpha
						&\leq C\|\nabla\tilde\phi\|_{H^{s-1}}|\nabla(\eta_Ra_\infty)|_{H_\ell^{s-1}}|\nabla(\eta_Rv_\infty)|_{H_\ell^{s-1}}
						\nonumber\\
						&\leq \varepsilon |\nabla(\eta_Rv_\infty)|_{H_\ell^{s-1}}^2+C\|\nabla\tilde\phi\|_{H^{s-1}}^2|\nabla(\eta_Ra_\infty)|_{H_\ell^{s-1}}^2.
					\end{align}
					In terms of integration by parts, we have 
					\begin{align}
						J_{92}^\alpha&=(\nabla g_3(\tilde\phi)\partial_x^\alpha(\eta_R a_\infty),|x|^{2\ell}\partial_x^\alpha(\eta_Rv_\infty ))+(g_3(\tilde\phi)\partial_x^\alpha(\eta_R a_\infty),\nabla(|x|^{2\ell})\partial_x^\alpha(\eta_Rv_\infty ))\nonumber\\
						&+(g_3(\tilde\phi)\partial_x^\alpha(\eta_R a_\infty),|x|^{2\ell}\partial_x^\alpha\text{div}(\eta_Rv_\infty )).\label{100401}
					\end{align}
					It is shown that 
					\begin{align*}
						&(\nabla g_3(\tilde\phi)\partial_x^\alpha(\eta_R a_\infty),|x|^{2\ell}\partial_x^\alpha(\eta_Rv_\infty ))+(g_3(\tilde\phi)\partial_x^\alpha(\eta_R a_\infty),\nabla(|x|^{2\ell})\partial_x^\alpha(\eta_Rv_\infty ))\\
						&\leq \|\nabla g_3(\tilde\phi)\|_{L^\infty}\||x|^\ell\partial_x^\alpha (\eta_R a_\infty)\|_{L^2}
						\||x|^\ell\partial_x^\alpha (\eta_R v_\infty)\|_{L^2}\\
						&~~+C\|g_3(\tilde\phi)\|_{L^\infty}\||x|^\ell\partial_x^\alpha (\eta_R a_\infty)\|_{L^2}
						\||x|^{\ell-1}\partial_x^\alpha (\eta_R v_\infty)\|_{L^2}\\
						&\leq C\|\nabla\tilde\phi\|_{H^{s-1}}\||x|^\ell\partial_x^\alpha (\eta_R a_\infty)\|_{L^2}
						\||x|^\ell\partial_x^\alpha (\eta_R v_\infty)\|_{L^2}\\
						&~~+C\|\nabla\tilde\phi\|_{H^{s-1}}\||x|^\ell\partial_x^\alpha (\eta_R a_\infty)\|_{L^2}
						\||x|^{\ell-1}\partial_x^\alpha (\eta_R v_\infty)\|_{L^2}\\
						&\leq \varepsilon \||x|^\ell\partial_x^\alpha (\eta_R v_\infty)\|_{L^2}^2+C\|\nabla\tilde\phi\|_{H^{s-1}}^2\||x|^\ell\partial_x^\alpha (\eta_R a_\infty)\|_{L^2}^2+C\||x|^{\ell-1}\partial_x^\alpha (\eta_R v_\infty)\|_{L^2}^2.
					\end{align*}
					Thus we have
					\begin{align*}
						&\sum_{|\alpha|=1}^s((\nabla g_3(\tilde\phi)\partial_x^\alpha(\eta_R a_\infty),|x|^{2\ell}\partial_x^\alpha(\eta_Rv_\infty ))+(g_3(\tilde\phi)\partial_x^\alpha(\eta_R a_\infty),\nabla(|x|^{2\ell})\partial_x^\alpha(\eta_Rv_\infty )))\\
						&\leq \varepsilon |\nabla(\eta_Rv_\infty)|_{H_\ell^{s-1}}^2+
						C\|\nabla\tilde\phi\|_{H^{s-1}}^2|\nabla(\eta_Ra_\infty)|_{H_\ell^{s-1}}^2+C|\nabla(\eta_Ra_\infty)|_{H_{\ell-1}^{s-1}}^2.
					\end{align*}
					The remaining term on the right-hand side of \eqref{100401} is given by
					\begin{align}
						&	(g_3(\tilde\phi)\partial_x^\alpha(\eta_R a_\infty),|x|^{2\ell}\partial_x^\alpha\text{div}(\eta_Rv_\infty ))\nonumber\\
						&=-(g_3(\tilde\phi)\partial_x^\alpha(\eta_R a_\infty)|x|^{2\ell},\partial_x^\alpha\partial_t(\eta_R a_\infty))\nonumber\\
						&-(g_3(\tilde\phi)\partial_x^\alpha(\eta_R a_\infty)|x|^{2\ell},
						\partial_x^\alpha(\tilde{v}\cdot\nabla(\eta_Ra_\infty)))\nonumber\\
						&+(g_3(\tilde\phi)\partial_x^\alpha(\eta_R a_\infty),\partial_x^\alpha(\tilde{v}\cdot\nabla\eta_R a_\infty+v_\infty\cdot\nabla\eta_R)).\label{102901}
					\end{align}
					Noting that 
					\begin{align*}
						&-\sum_{|\alpha|=1}^s(g_3(\tilde\phi)\partial_x^\alpha(\eta_R a_\infty),|x|^{2\ell}\partial_x^\alpha\partial_t(\eta_R a_\infty))\\
						&=-\frac{1}{2}\sum_{|\alpha|=1}^s\frac{d}{dt}\int_{\mathbb{R}^3}g_3(\tilde\phi)\big||x|^\ell\partial_x^\alpha (\eta_R a_\infty)\big|^2\mathrm{d}x+\frac{1}{2}\sum_{|\alpha|=1}^s\frac{d}{dt}\int_{\mathbb{R}^3}\partial_tg_3(\tilde\phi)\big||x|^\ell\partial_x^\alpha (\eta_R a_\infty)\big|^2\mathrm{d}x\\
						&\leq -\frac{1}{2}\sum_{|\alpha|=1}^s\frac{d}{dt}\int_{\mathbb{R}^3}g_3(\tilde\phi)\big||x|^\ell\partial_x^\alpha (\eta_R a_\infty)\big|^2\mathrm{d}x+C\|\partial_t\tilde\phi\|_{H^{s-1}}|\nabla(\eta_Ra_\infty)|_{H_\ell^{s-1}}^2,
					\end{align*}
					and
					\begin{align*}
						&	-(g_3(\tilde\phi)\partial_x^\alpha(\eta_R a_\infty)|x|^{2\ell},\partial_x^\alpha(\tilde{v}\cdot\nabla(\eta_R a_\infty))\\
						&=-(g_3(\tilde\phi)\partial_x^\alpha(\eta_R a_\infty)|x|^{2\ell},\partial_x^\alpha(\tilde{v}\cdot\nabla(\eta_R a_\infty))-\tilde{v}\cdot\partial_x^\alpha\nabla(\eta_R a_\infty))\\
						&-(g_3(\tilde\phi)\partial_x^\alpha(\eta_R a_\infty)|x|^{2\ell},\tilde{v}\cdot\partial_x^\alpha\nabla(\eta_R a_\infty)).
					\end{align*}
					Then we have
					\begin{align*}
						&\sum_{|\alpha|=1}^s(-(g_3(\tilde\phi)\partial_x^\alpha(\eta_R a_\infty)|x|^{2\ell},\partial_x^\alpha(\tilde{v}\cdot\nabla(\eta_R a_\infty))-\tilde{v}\cdot\partial_x^\alpha\nabla(\eta_R a_\infty)))\\
						&\leq C\|g_3(\tilde\phi)\|_{L^\infty}|\nabla(\eta_Ra_\infty)|_{H_\ell^{s-1}}\|\tilde v\|_{H^s}|\nabla(\eta_Ra_\infty)|_{H_\ell^{s-1}}	\\
						&\leq C(\|\nabla\tilde\phi\|_{H^{s-1}}^2+\|\tilde v\|_{H^s}^2)|\nabla(\eta_Ra_\infty)|_{H_\ell^{s-1}}^2,
					\end{align*}
					and 
					\begin{align*}
						&\sum_{|\alpha|=1}^s((g_3(\tilde\phi)\partial_x^\alpha(\eta_R a_\infty)|x|^{2\ell},\tilde{v}\cdot\partial_x^\alpha\nabla(\eta_R a_\infty))\\
						& = -\sum_{|\alpha|=1}^s(\div ((g_3(\tilde\phi)\tilde{v})\partial_x^\alpha(\eta_R a_\infty), |x|^{2\ell}\partial_x^\alpha \eta_R a_\infty) -\sum_{|\alpha|=1}^s( (g_3(\tilde\phi)\tilde{v} \partial_x^\alpha(\eta_R a_\infty), \nabla(|x|^{2\ell}) \partial_x^\alpha \eta_R a_\infty)\\  
						&\leq C \|\nabla \tilde{v}\|_{H^{s-1}} \|\nabla g_3(\tilde\phi)\|_{L^\infty}|\nabla(\eta_R a_\infty)|_{H_\ell^{s-1}}^2+C\|\nabla \tilde{v}\|_{H^{s-1}} \|g_3(\tilde\phi)\|_{L^\infty}|\nabla(\eta_Ra_\infty)|_{H_\ell^{s-1}}|\nabla(\eta_Ra_\infty)|_{H_{\ell-1}^{s-1}}\\
						&\leq C\|\nabla \tilde{v}\|_{H^{s-1}} \|\nabla\tilde\phi\|_{H^{s-1}}|\nabla(\eta_Ra_\infty)|_{H_\ell^{s-1}}^2+C\|\nabla \tilde{v}\|_{H^{s-1}} \|\nabla\tilde\phi\|_{H^{s-1}}|\nabla(\eta_Ra_\infty)|_{H_\ell^{s-1}}|\nabla(\eta_Ra_\infty)|_{H_{\ell-1}^{s-1}}\\
						&\leq  C(\|\nabla \tilde{v}\|_{H^{s-1}}^2 +\|\nabla\tilde\phi\|_{H^{s-1}}^2)|\nabla(\eta_Ra_\infty)|_{H_\ell^{s-1}}^2+C|\nabla(\eta_Ra_\infty)|_{H_{\ell-1}^{s-1}}^2.
					\end{align*}
					It then concludes from above to show that 
					\begin{align*}
						&\sum_{|\alpha|=1}^s(	-(g_3(\tilde\phi)\partial_x^\alpha(\eta_R a_\infty)|x|^{2\ell},\partial_x^\alpha(\tilde{v}\cdot\nabla(\eta_R a_\infty))\\
						&\leq C(\|\nabla\tilde\phi\|_{H^{s-1}}+\|\nabla\tilde\phi\|_{H^{s-1}}^2+\|\tilde v\|_{H^s}^2)|\nabla(\eta_Ra_\infty)|_{H_\ell^{s-1}}^2+C|\nabla(\eta_Ra_\infty)|_{H_{\ell-1}^{s-1}}^2.
					\end{align*}
					The third term on the right-hand side of \eqref{102901} is estimates as
					\begin{align*}
						&\sum_{|\alpha|=1}^s(g_3(\tilde\phi)\partial_x^\alpha(\eta_R a_\infty),|x|^{2\ell}\partial_x^\alpha(\tilde{v}\cdot\nabla\eta_R a_\infty+v_\infty\cdot\nabla\eta_R))\\
						& \leq C \|g_3(\tilde\phi)\|_{L^\infty(N_R)}\|\nabla (\eta_Ra_\infty)\|_{H^{s-1}(N_R)}
						(\|\tilde{v}\|_{H^s(N_R)}\|a_\infty\|_{H^s(N_R)}+\|v_\infty\|_{H^s(N_R)})\\
						&\leq C  \|\nabla\tilde\phi\|_{H^{s-1}(N_R)}\|\nabla(\eta_Ra_\infty)\|_{H^{s-1}(N_R)}(\|\tilde{v}\|_{H^s(N_R)}\|a_\infty\|_{H^s(N_R)}+\|v_\infty\|_{H^s(N_R)})\\
						&\leq C \|\nabla\tilde\phi\|_{H^{s-1}(N_R)}\|a_\infty\|_{H^s(N_R)}(\|\tilde{v}\|_{H^s(N_R)}\|a_\infty\|_{H^s(N_R)}+\|v_\infty\|_{H^s(N_R)})\\
						&\leq C  (\|\nabla\tilde\phi\|_{H^{s-1}(N_R)}+\|\nabla\tilde\phi\|_{H^{s-1}(N_R)}^2+\|\tilde v\|_{H^s(N_R)}^2)(\|a_\infty\|_{H^s(N_R)}^2+\|v_\infty\|_{H^s(N_R)}^2).
					\end{align*}
					Therefore we conclude from above that 
					\begin{align*}
						\sum_{|\alpha|=1}^sJ_{92}^\alpha
						&\leq-\frac{1}{2}\sum_{|\alpha|=1}^s\frac{d}{dt}\int_{\mathbb{R}^3}g_3(\tilde\phi)\big||x|^\ell\partial_x^\alpha (\eta_R a_\infty)\big|^2\mathrm{d}x+C\|\partial_t\tilde\phi\|_{H^{s-1}}|\nabla(\eta_Ra_\infty)|_{H_\ell^{s-1}}^2\nonumber\\
						&+C(\|\nabla\tilde\phi\|_{H^{s-1}}+\|\nabla\tilde\phi\|_{H^{s-1}}^2+\|\tilde v\|_{H^s}^2)|\nabla(\eta_Ra_\infty)|_{H_\ell^{s-1}}^2+C|\nabla(\eta_Ra_\infty)|_{H_{\ell-1}^{s-1}}^2\nonumber\\
						&+C(\|\nabla\tilde\phi\|_{H^{s-1}(N_R)}+\|\nabla\tilde\phi\|_{H^{s-1}(N_R)}^2+\|\tilde v\|_{H^s(N_R)}^2)(\|a_\infty\|_{H^s(N_R)}^2+\|v_\infty\|_{H^s(N_R)}^2),
					\end{align*}
					which together with \eqref{103001} yields that 
					\begin{align*}
						\sum_{|\alpha|=1}^s J_9^\alpha
						&\leq-\frac{1}{2}\sum_{|\alpha|=1}^s\frac{d}{dt}\int_{\mathbb{R}^3}g_3(\tilde\phi)\big||x|^\ell\partial_x^\alpha (\eta_R a_\infty)\big|^2\mathrm{d}x+C\|\partial_t\tilde\phi\|_{H^{s-1}}|\nabla(\eta_Ra_\infty)|_{H_\ell^{s-1}}^2\nonumber\\
						&+C(\|\nabla\tilde\phi\|_{H^{s-1}}+\|\nabla\tilde\phi\|_{H^{s-1}}^2+\|\tilde v\|_{H^s}^2)|\nabla(\eta_Ra_\infty)|_{H_\ell^{s-1}}^2+C|\nabla(\eta_Ra_\infty)|_{H_{\ell-1}^{s-1}}^2\nonumber\\
						&+C(\|\nabla\tilde\phi\|_{H^{s-1}(N_R)}+\|\nabla\tilde\phi\|_{H^{s-1}(N_R)}^2+\|\tilde v\|_{H^s(N_R)}^2)(\|a_\infty\|_{H^s(N_R)}^2+\|v_\infty\|_{H^s(N_R)}^2).
					\end{align*}
					The last term $J_{10}^\alpha$ is estimated as follows,
					\begin{align*}
						\sum_{|\alpha|=1}^s	J_{10}^\alpha&\leq \sum_{|\alpha|=1}^s	\||x|^\ell\partial_x^\alpha(\eta_RF_\infty^1)\|_{L^2}\||x|^{\ell}\partial_x^\alpha(\eta_Rv_\infty)\|_{L^2}\\
						&\leq \varepsilon |\nabla(\eta_Rv_\infty)|_{H_\ell^{s-1}}^2+C|\nabla(\eta_RF_\infty^1)|_{H_\ell^{s-1}}^2.
					\end{align*}
					Summing $|\alpha|$ from $1$ to $s$ and noting the smallness of $\varepsilon$, we deduce that 
					\begin{align}
						&\frac{1}{2}\frac{\mathrm{d}}{\mathrm{d}t}\big| \nabla(
						\eta_Rv_{\infty})|_{H_\ell^{s-1}}^2-\sum _{|\alpha|=1}^s(\partial_x^\alpha(\eta_Ra_{\infty}),|x|^{2\ell}\partial_x^\alpha\text{div}(\eta_Rv_{\infty}))\nonumber\\
						&+\frac{1}{2}\big|\nabla(\eta_Rv_{\infty})\big|_{H_\ell^{s-1}}^2
						+\frac{1}{2}\sum_{|\alpha|=1}^s\frac{\mathrm{d}}{\mathrm{d}t}\int_{\mathbb{R}^3}g_3(\tilde\phi)\big||x|^\ell\partial_x^\alpha (\eta_R a_\infty)\big|^2\mathrm{d}x\nonumber\\
						&\leq C|\nabla(\eta_RF_\infty^1)|_{H_\ell^{s-1}}^2
						+C(\|\nabla\tilde\phi\|_{H^{s-1}}+\|\nabla\tilde\phi\|_{H^{s-1}}^2+\|\tilde v\|_{H^s}^2)|\nabla(\eta_Ra_\infty)|_{H_\ell^{s-1}}^2\nonumber\\
						&+C\|\partial_t\tilde\phi\|_{H^{s-1}}|\nabla(\eta_Ra_\infty)|_{H_\ell^{s-1}}^2+C|\nabla(\eta_Ra_\infty)|_{H_{\ell-1}^{s-1}}^2\nonumber\\
						&+C(1+\|\nabla\tilde\phi\|_{H^{s-1}(N_R)}+\|\tilde v\|_{H^s(N_R)}+\|\nabla\tilde\phi\|_{H^{s-1}(N_R)}^2+\|\tilde v\|_{H^s(N_R)}^2)\|u_\infty\|_{H^s(N_R)}^2.
						\label{101001}
					\end{align}
					Combining the estimates \eqref{101001} and \eqref{101002} together, we get 
					\begin{align*}
						&\frac{1}{2}\frac{\mathrm{d}}{\mathrm{d}t}\big(\big| 
						\nabla(\eta_Ra_{\infty})|_{H_\ell^{s-1}}^2+\big|\nabla(\eta_Rv_\infty)\big|_{H_\ell^{s-1}}^2\big)+\frac{1}{2}\big|\nabla(\eta_Rv_{\infty})|_{H_\ell^{s-1}}^2\\
						&+\frac{1}{2}\sum_{|\alpha|=1}^s\frac{\mathrm{d}}{\mathrm{d}t}\int_{\mathbb{R}^3}g_3(\tilde\phi)\big||x|^\ell\partial_x^\alpha (\eta_R a_\infty)\big|^2\mathrm{d}x\\
						&\leq C|\nabla(\eta_RF_\infty)|_{H_\ell^{s-1}}^2+C\|\partial_t\tilde\phi\|_{H^{s-1}}|\nabla(\eta_Ra_\infty)|_{H_\ell^{s-1}}^2+C\varepsilon  |\nabla(\eta_Ra_\infty)|_{H_\ell^{s-1}}^2 \nonumber\\
						&+C(\|\nabla\tilde\phi\|_{H^{s-1}}+\|\nabla\tilde\phi\|_{H^{s-1}}^2+\|\tilde v\|_{H^s}+\|\tilde v\|_{H^s}^2)|\nabla(\eta_Ra_\infty)|_{H_\ell^{s-1}}^2\\
						&+C|\nabla(\eta_Ra_\infty)|_{H_{\ell-1}^{s-1}}^2+C(1+\|\nabla\tilde\phi\|_{H^{s-1}(N_R)}+\|\tilde v\|_{H^s(N_R)}+\|\nabla\tilde\phi\|_{H^{s-1}(N_R)}^2+\|\tilde v\|_{H^s(N_R)}^2)\|u_\infty\|_{H^s(N_R)}^2,
					\end{align*}
					which together with \eqref{102801} yields that 
					\begin{align}
						&\frac{\mathrm{d}}{\mathrm{d}t}\big(\big| 
						\eta_Ra_{\infty}|_{H_\ell^{s}}^2+\big|\eta_Rv_\infty\big|_{H_\ell^s}^2\big)+\big|\eta_Rv_{\infty}|_{H_\ell^s}^2+\sum_{|\alpha|=0}^s\frac{\mathrm{d}}{\mathrm{d}t}\int_{\mathbb{R}^3}g_3(\tilde\phi)\big||x|^\ell\partial_x^\alpha (\eta_R a_\infty)\big|^2\mathrm{d}x\nonumber\\
						&\leq C|\eta_RF_\infty|_{H_\ell^{s}}^2+C|\eta_Ra_\infty|_{H_{\ell-1}^s}^2+C\varepsilon  |\nabla(\eta_Ra_\infty)|_{H_\ell^{s-1}}^2\nonumber\\
						&+C(\|\partial_t\tilde\phi\|_{H^{s-1}}+\|\nabla\tilde\phi\|_{H^{s-1}}+\|\nabla\tilde\phi\|_{H^{s-1}}^2+\|\tilde v\|_{H^s}+\|\tilde v\|_{H^s}^2)|\eta_Ra_\infty|_{H_\ell^{s}}^2
						\nonumber\\
						&+C(1+\|\nabla\tilde\phi\|_{H^{s-1}(N_R)}+\|\tilde v\|_{H^s(N_R)}+\|\nabla\tilde\phi\|_{H^{s-1}(N_R)}^2+\|\tilde v\|_{H^s(N_R)}^2)\|u_\infty\|_{H^s(N_R)}^2.\label{110101}
					\end{align}
					We next consider the estimates of $\big\||x|^{2\ell}\partial_x^\alpha a_{\infty}\big\|_{L^2}^2$ for $1\leq |\alpha|\leq s$. For multi-index $\alpha$ satisfying $1\leq |\alpha|\leq s$, we take the inner product of $\partial_x^{\alpha-1}\eqref{Mainsystem1}_3$
					with $|x|^{2\ell}\partial_x^\alpha (\eta_Ra_{\infty})$ to prove 
					\begin{align*}
						& (\partial_t\partial_x^{\alpha-1}(\eta_Rv_\infty),|x|^{2\ell}\partial_x^{\alpha}(\eta_Ra_{\infty}))+\|\partial_x^{\alpha}(\eta_Ra_{\infty})\|_{L_\ell^2}^2 \\
						&=-(\partial_x^{\alpha-1}(\eta_Rv_\infty),|x|^{2\ell}\partial_x^\alpha(\eta_R a_{\infty}))+(\partial_x^{\alpha}(\eta_R F_{\infty}^0),|x|^{2\ell}\partial_x^{\alpha-1}(\eta_Rv_{\infty}))\\
						&+(\partial_x^{\alpha-1}(\eta_R F_{\infty}^1),|x|^{2\ell}\partial_x^\alpha(\eta_Ra_{\infty}))+(\partial_x^{\alpha-1}K_2,|x|^{2\ell}\partial_x^\alpha(\eta_Ra_{\infty}))\\
						&-(\partial_x^{\alpha-1}(\tilde{v}\cdot\nabla(\eta_R v_\infty)),|x|^{2\ell}\partial_x^\alpha (\eta_R a_\infty))-(\partial_x^{\alpha-1}(g_3(\tilde\phi)\nabla(\eta_R a_\infty)),|x|^{2\ell}\partial_x^\alpha(\eta_Ra_\infty)).
					\end{align*}
					Note 
					\begin{align*}
						&(\partial_t\partial_x^{\alpha-1} (\eta_Rv_\infty),|x|^{2\ell}\partial_x^\alpha(\eta_Ra_{\infty}))\\
						&=\frac{\mathrm{d}}{\mathrm{d}t}(\partial_x^{\alpha-1}(\eta_Rv_\infty),|x|^{2\ell}\partial_x^\alpha(\eta_Ra_{\infty}))-(\partial_x^{\alpha-1}(\eta_Rv_\infty),|x|^{2\ell}\partial_x^\alpha\partial_t(\eta_R a_{\infty}))\\
						&=\frac{\mathrm{d}}{\mathrm{d}t}(\partial_x^{\alpha-1}(\eta_Rv_\infty),|x|^{2\ell}\partial_x^\alpha(\eta_Ra_{\infty}))+(\partial_x^{\alpha-1}\text{div}(\eta_Rv_\infty),|x|^{2\ell}\partial_x^{\alpha-1}\partial_t(\eta_R a_{\infty}))\\
						&~~+(\partial_x^{\alpha-1}(\eta_Rv_\infty),\nabla(|x|^{2\ell})\partial_x^{\alpha-1}\partial_t(\eta_R a_{\infty})),
					\end{align*}
					and
					\begin{align*}
						\partial_t(\eta_Ra_{\infty})=-\tilde{v}\cdot\nabla(\eta_Ra_{\infty})-\text{div}(\eta_Rv_\infty)+\eta_RF_\infty^0+K_1.
					\end{align*}
					Then we have
					\begin{align*}
						&\frac{\mathrm{d}}{\mathrm{d}t}(\partial_x^{\alpha-1}(\eta_Rv_\infty),|x|^{2\ell}\partial_x^\alpha(\eta_Ra_{\infty}))
						+\|\partial_x^\alpha(\eta_Ra_{\infty})\|_{L_\ell^2}^2 \\
						&=(\partial_x^{\alpha-1}\text{div}(\eta_Rv_\infty),|x|^{2\ell}\partial_x^{\alpha-1}(\tilde{v}\cdot\nabla(\eta_Ra_{\infty})+\text{div}(\eta_Rv_\infty)-K_1)\\
						&+(\partial_x^{\alpha-1}(\eta_Rv_\infty),\nabla(|x|^{2\ell})\partial_x^{\alpha-1}(\tilde{v}\cdot\nabla(\eta_Ra_{\infty})+\text{div}(\eta_Rv_\infty)-\eta_{R}F_{\infty}^1-K_1)\\
						&-(\partial_x^{\alpha-1}(\eta_Rv_\infty),|x|^{2\ell}\partial_x^\alpha(\eta_R a_{\infty}))
						+(\partial_x^{\alpha-1}(\eta_R F_{\infty}^1),|x|^{2\ell}\partial_x^\alpha(\eta_Ra_{\infty}))\\
						&+(\partial_x^{\alpha-1}K_2,|x|^{2\ell}\partial_x^\alpha(\eta_Ra_{\infty}))
						-(\partial_x^{\alpha-1}(\tilde{v}\cdot\nabla(\eta_R v_\infty)),|x|^{2\ell}\partial_x^\alpha (\eta_R a_\infty))\\
						&-(\partial_x^{\alpha-1}(g_3(\tilde\phi)\nabla(\eta_R a_\infty)),|x|^{2\ell}\partial_x^\alpha(\eta_Ra_\infty))+(\partial_x^{\alpha}(\eta_R F_{\infty}^0),|x|^{2\ell}\partial_x^{\alpha-1}(\eta_Rv_{\infty}))\\
						&:= L_1^\alpha+L_2^\alpha+L_3^\alpha+L_4^\alpha+L_5^\alpha+L_6^\alpha+L_7^\alpha+L_8^\alpha.
					\end{align*}
					It is easily seen that 
					\begin{align*}
						\sum_{|\alpha|=1}^sL_1^\alpha&\leq \sum_{|\alpha|=1}^s(\partial_x^{\alpha-1}\text{div}(\eta_Rv_\infty), |x|^{2\ell}\partial_x^{\alpha-1}(\tilde{v}\cdot\nabla(\eta_R a_\infty)))\\
						&+\sum_{|\alpha|=1}^s(\partial_x^{\alpha-1}\text{div}(\eta_Rv_\infty),|x|^{2\ell}\partial_x^{\alpha-1}\text{div}(\eta_Rv_\infty)) \\
						&-\sum_{|\alpha=1|^s}(\partial_x^{\alpha-1}\text{div}(\eta_Rv_\infty),|x|^{2\ell}\partial_x^{\alpha-1}K_1)\\
						&\leq C\|\tilde v\|_{H^s}|\nabla(\eta_R v_\infty)|_{H_\ell^{s-1}}^2+ |\nabla(\eta_Rv_\infty)|_{H_\ell^{s-1}}^2+C(1+\|\tilde{v}\|_{H^s(N_R)})\|u_\infty\|_{H^s(N_R)}^2.
					\end{align*}
					The second term $L_2^\alpha$ is given by 
					\begin{align*}
						\sum_{|\alpha|=1}^sL_2^\alpha
						&=\sum_{|\alpha|=1}^s(\partial_x^{\alpha-1}(\eta_Rv_\infty),\nabla(|x|^{2\ell})\partial_x^{\alpha-1}(\tilde{v}\cdot\nabla(\eta_Ra_\infty)+\text{div}(\eta_Rv_\infty)-\eta_RF_\infty^1-K_1))\\
						&\leq C|\eta_Rv_\infty|_{H_{\ell-1}^{s-1}}\|\tilde v\|_{H^s}|\nabla(\eta_Ra_\infty)|_{H_\ell^{s-1}}
						+C|\eta_Rv_\infty|_{H_{\ell-1}^{s-1}}|\nabla(\eta_Rv_\infty)|_{H_{\ell}^{s-1}}\\
						&+C|\eta_Rv_\infty|_{H_{\ell-1}^{s-1}}|\eta_RF_\infty^1|_{H_{\ell}^{s-1}}+C\|v_\infty\|_{H^s(N_R)}(\|\tilde v\|_{H^s(N_R)}\|a_\infty\|_{H^s(N_R)}+\|v_\infty\|_{H^s(N_R)})\\
						&\leq C\|\tilde v\|_{H^s}^2|\nabla(\eta_Ra_\infty)|_{H_\ell^{s-1}}^2+C|\eta_Rv_\infty|_{H_{\ell-1}^{s-1}}^2
						+C|\nabla(\eta_Rv_\infty)|_{H_\ell^{s-1}}^2+C|\eta_RF_\infty^1|_{H_\ell^{s-1}}^2\\
						&+C(1+\|\tilde v\|_{H^s(N_R)})\|u_\infty\|_{H^s(N_R)}^2.
					\end{align*}
					The third term $L_3^\alpha$ is given by 
					\begin{align*}
						\sum_{|\alpha|=1}^sL_3^\alpha&\leq 
						\sum_{|\alpha|=1}^s\||x|^{\ell-1}\partial_x^\alpha(\eta_Rv_\infty)\|_{L^2}\||x|^\ell\partial_x^\alpha(\eta_Ra_\infty)\|_{L^2}\\
						&\leq |\nabla(\eta_Rv_\infty)|_{H_{\ell-1}^{s-1}}|\nabla(\eta_Ra_\infty)|_{H_\ell^{s-1}}\\ 
						&\leq \varepsilon|\nabla(\eta_Ra_\infty)|_{H_\ell^{s-1}}^2+C|\nabla(\eta_Rv_\infty)|_{H_{\ell-1}^{s-1}}^2.
					\end{align*}
					The fourth term $L_4^\alpha$ is given by 
					\begin{align*}
						\sum_{|\alpha|=1}^sL_4^\alpha&=\sum_{|\alpha|=1}^s(\partial_x^{\alpha-1}(\eta_R F_\infty^1),|x|^{2\ell}\partial_x^\alpha(\eta_Ra_\infty))\\
						&\leq \varepsilon|\nabla(\eta_R a_\infty)|_{H_\ell^{s-1}}^2+C|\eta_RF_\infty^1|_{H_\ell^{s-1}}^2\\
						&\leq \varepsilon|\nabla(\eta_R a_\infty)|_{H_\ell^{s-1}}^2+C|\eta_RF_\infty^1|_{H_\ell^{s}}^2.
					\end{align*}
					The fifth term $L_5^\alpha$ is given by 
					\begin{align*}
						\sum_{|\alpha|=1}^sL_5^\alpha&=\sum_{|\alpha|=1}^s(\partial_x^{\alpha-1}K_2,|x|^{2\ell}\partial_x^\alpha(\eta_Ra_\infty))\\
						&=\sum_{|\alpha|=1}^s(\partial_x^{\alpha-1}(a_\infty\nabla\eta_R+\tilde{v}\cdot\nabla\eta_R v_\infty+g_3(\tilde\phi)\nabla\eta_R a_\infty),|x|^{2\ell}\partial_x^\alpha(\eta_Ra_\infty))\\
						&\leq C(\|a_\infty\|_{H^s(N_R)}+\|\tilde v\|_{H^s(N_R)}\|v_\infty\|_{H^s(N_R)}+\|\nabla\tilde{\phi}\|_{H^{s-1}(N_R)}\|a_\infty\|_{H^s(N_R)})\|a_\infty\|_{H^s(N_R)}\\
						&\leq C(1+\|\nabla\tilde\phi\|_{H^{s-1}(N_R)}+\|\tilde v\|_{H^s(N_R)})\|u_\infty\|_{H^s(N_R)}^2.
					\end{align*}
					The sixth term $L_6^\alpha$ is given by 
					\begin{align*}
						\sum_{|\alpha|=1}^sL_6^\alpha &=\sum_{|\alpha|=1}^s(\partial_x^{\alpha-1}(\tilde{v}\cdot\nabla(\eta_Rv_\infty)),|x|^{2\ell}\partial_x^\alpha(\eta_R a_\infty))\\
						&\leq C\|\tilde{v}\|_{H^s}|\nabla(\eta_Rv_\infty)|_{H_\ell^{s-1}}|\nabla(\eta_Ra_\infty)|_{H_\ell^{s-1}}\\
						&\leq \varepsilon|\nabla(\eta_Ra_\infty)|_{H_\ell^{s-1}}^2+C\|\tilde{v}\|_{H^s}^2|\nabla(\eta_Rv_\infty)|_{H_\ell^{s-1}}^2.
					\end{align*}
					The seventh term $L_7^\alpha$ is given by 
					\begin{align*}
						\sum_{|\alpha|=1}^sL_7^\alpha
						&= \sum_{|\alpha|=1}^s(\partial_x^{\alpha-1}(g_3(\tilde\phi)\nabla(\eta_R a_\infty)),|x|^{2\ell}\partial_x^\alpha(\eta_R a_\infty))\\
						&\leq C\|\nabla\tilde \phi\|_{H^{s-1}}|\nabla(\eta_Ra_\infty)|_{H_\ell^{s-1}}^2.
					\end{align*}
					The last term $L_8^\alpha$ is given by 
					\begin{align*}
						\sum_{|\alpha|=1}^sL_8^\alpha&=\sum_{|\alpha|=1}^s(\partial_x^{\alpha}(\eta_R F_\infty^0),|x|^{2\ell}\partial_x^{\alpha-1}(\eta_Rv_\infty))\\
						&\leq \varepsilon|\eta_R v_\infty|_{H_\ell^{s-1}}^2+C|\eta_RF_\infty^0|_{H_\ell^{s}}^2\\
						&\leq C\varepsilon|\nabla(\eta_R v_\infty)|_{H_\ell^{s-1}}^2+C|\eta_RF_\infty^0|_{H_\ell^{s}}^2.
					\end{align*}
					It then concludes from above that 
					\begin{align*}
						&\sum_{|\alpha|=1}^s\frac{\mathrm{d}}{\mathrm{d}t}(\partial_x^{\alpha-1}(\eta_Rv_\infty),|x|^{2\ell}\partial_x^\alpha(\eta_Ra_{\infty}))
						+|\nabla(\eta_Ra_{\infty})|_{H_\ell^{s-1}}^2 \\
						&\leq 3\varepsilon |\nabla(\eta_R a_\infty)|_{H_\ell^{s-1}}^2+C(\|\nabla\tilde\phi\|_{H^{s-1}}+\|\tilde v\|_{H^s}+\|\tilde v\|_{H^s}^2)|\nabla(\eta_Ra_\infty)|_{H_\ell^{s-1}}^2+C|\eta_Rv_\infty|_{H_{\ell-1}^{s-1}}^2\\
						&+C|\nabla(\eta_Rv_\infty)|_{H_\ell^{s-1}}^2+C|\eta_RF_\infty|_{H_\ell^{s}}^2+C(1+\|\nabla\tilde\phi\|_{H^{s-1}(N_R)}+\|\tilde v\|_{H^s(N_R)})\|u_\infty\|_{H^s(N_R)}^2.
					\end{align*}
					Using the smallness of $\varepsilon$, we get 
					\begin{align}
						&\sum_{|\alpha|=1}^s\frac{\mathrm{d}}{\mathrm{d}t}(\partial_x^{\alpha-1}(\eta_Rv_\infty),|x|^{2\ell}\partial_x^\alpha(\eta_Ra_{\infty}))
						+\frac{1}{2}|\nabla(\eta_Ra_{\infty})|_{H_\ell^{s-1}}^2 \nonumber\\
						&\leq C(\|\nabla\tilde\phi\|_{H^{s-1}}+\|\tilde v\|_{H^s}+\|\tilde v\|_{H^s}^2)|\nabla(\eta_Ra_\infty)|_{H_\ell^{s-1}}^2+C|\eta_Rv_\infty|_{H_{\ell-1}^{s}}^2\nonumber\\
						&+C|\nabla(\eta_Rv_\infty)|_{H_\ell^{s-1}}^2+C|\eta_RF_\infty|_{H_\ell^{s}}^2+C(1+\|\nabla\tilde\phi\|_{H^{s-1}(N_R)}+\|\tilde v\|_{H^s(N_R)})\|u_\infty\|_{H^s(N_R)}^2.
						\label{101003}
					\end{align}
					Taking $\beta_2\times \eqref{101003}+\eqref{110101}$ with $\beta_2$ a small positive constant, we have
					\begin{align*}
						&\frac{\mathrm{d}}{\mathrm{d}t}\Big\{\big| 
						\eta_Ra_{\infty}|_{H_\ell^s}^2+\big|\eta_Rv_\infty\big|_{H_\ell^s}^2+\beta_2 (\partial_x^{\alpha-1}(\eta_Rv_\infty),|x|^{2\ell}\partial_x^\alpha(\eta_Ra_{\infty}))\Big\}+\big|\eta_Rv_{\infty}|_{H_\ell^s}^2\nonumber\\
						&+\sum_{|\alpha|=0}^s\frac{\mathrm{d}}{\mathrm{d}t}\int_{\mathbb{R}^3}g_3(\tilde\phi)\big||x|^\ell\partial_x^\alpha (\eta_R a_\infty)\big|^2\mathrm{d}x
						+\frac{1}{2}\beta_3|\nabla(\eta_Ra_{\infty})|_{H_\ell^{s-1}}^2\nonumber\\
						&\leq C|\eta_RF_\infty|_{H_\ell^{s}}^2
						+C(\|\nabla\tilde\phi\|_{H^{s-1}}+\|\nabla\tilde\phi\|_{H^{s-1}}^2+\|\tilde v\|_{H^s}+\|\tilde v\|_{H^s}^2)|\eta_Ra_\infty|_{H_\ell^{s}}^2+C\|\partial_t\tilde\phi\|_{H^{s-1}}|\nabla(\eta_Ra_\infty)|_{H_\ell^{s-1}}^2\nonumber\\
						&+C|\eta_Ra_\infty|_{H_{\ell-1}^s}^2+C(1+\|\nabla\tilde\phi\|_{H^{s-1}(N_R)}+\|\tilde v\|_{H^s(N_R)}+\|\nabla\tilde\phi\|_{H^{s-1}(N_R)}^2+\|\tilde v\|_{H^s(N_R)}^2)\|u_\infty\|_{H^s(N_R)}^2\nonumber\\
						&+C\|\partial_t\tilde\phi\|_{H^{s-1}}|\nabla(\eta_Ra_\infty)|_{H_\ell^{s-1}}^2+\beta_2\Big\{C(\|\nabla\tilde\phi\|_{H^{s-1}}+\|\tilde v\|_{H^s}+\|\tilde v\|_{H^s}^2)|\nabla(\eta_Ra_\infty)|_{H_\ell^{s-1}}^2+C|\eta_Rv_\infty|_{H_{\ell-1}^{s-1}}^2\nonumber\\
						&+C|\nabla(\eta_Rv_\infty)|_{H_\ell^{s-1}}^2+C|\eta_RF_\infty|_{H_\ell^{s-1}}^2+C(1+\|\nabla\tilde\phi\|_{H^{s-1}(N_R)}+\|\tilde v\|_{H^s(N_R)})\|u_\infty\|_{H^s(N_R)}^2\Big\}\\
						&  + \varepsilon |\eta_Ra_\infty|_{L^2_\ell}^2.
					\end{align*}
					Since $\beta_2$ and $\varepsilon$ are small positive constants, we deduce that there exists a positive constant $\bar{\kappa}_1$ such that 
					\begin{align}\label{91901}
						&\frac{\mathrm{d}}{\mathrm{d}t}E_{\ell}^s[\eta_R u_{\infty}](t)+\bar{\kappa}_1D_{\ell}^s[\eta_R u_{\infty}](t)\nonumber\\
						&\leq \varepsilon |\eta_Ra_\infty|_{L^2_\ell}^2+ C(\|\partial_t\tilde\phi\|_{H^{s-1}}+\|\nabla\tilde{\phi}\|_{H^{s-1}}+\|\tilde{v}\|_{H^s}+\|\nabla\tilde{\phi}\|_{H^{s-1}}^2+\|\tilde{v}\|_{H^s}^2)|\eta_R u_{\infty}|_{H_{\ell}^s}^2\nonumber\\
						&+C(1+\|\nabla\tilde{\phi}\|_{H^{s-1}(N_R)}+\|\tilde{v}\|_{H^s(N_R)}+\|\nabla\tilde{\phi}\|_{H^{s-1}(N_R)}^2+\|\tilde{v}\|_{H^s(N_R)}^2)\|u_{\infty}\|_{H^s(N_R)}^2\nonumber\\
						&+C|\eta_Ru_\infty|_{H_{\ell-1}^s}^2+C|\eta_R F_{\infty}|_{H_{\ell}^s}^2.
					\end{align}
					Similarly, there exists a positive constant $\tilde{\kappa}_1$ such that 
					\begin{align}\label{91902}
						&\frac{\mathrm{d}}{\mathrm{d}t}E_{\ell}^{s-1}[\eta_R u_{\infty}](t)+\tilde{\kappa}_1D_{\ell}^{s-1}[\eta_R u_{\infty}](t)\nonumber\\
						&\leq \varepsilon |\eta_Ra_\infty|_{L^2_\ell}^2 +C(\|\partial_t\tilde\phi\|_{H^{s-1}}+\|\nabla\tilde{\phi}\|_{H^{s-1}}+\|\tilde{v}\|_{H^s}+\|\nabla\tilde{\phi}\|_{H^{s-1}}^2+\|\tilde{v}\|_{H^s}^2)|\eta_R u_{\infty}|_{H_{\ell}^{s-1}}^2\nonumber\\
						&+C(1+\|\nabla\tilde{\phi}\|_{H^{s-1}(N_R)}+\|\tilde{v}\|_{H^s(N_R)}+\|\nabla\tilde{\phi}\|_{H^{s-1}(N_R)}^2+\|\tilde{v}\|_{H^s(N_R)}^2)\|u_{\infty}\|_{H^{s-1}(N_R)}^2\nonumber\\
						&+C|\eta_Ru_\infty|_{H_{\ell-1}^{s-1}}^2+C|\eta_R F_{\infty}|_{H_{\ell}^{s-1}}^2.
					\end{align}
					Let $\kappa_1=\max\{\bar{\kappa}_1,\tilde{\kappa}_1\}$. We combine \eqref{91901} and \eqref{91902} together to complete the proof of the proposition.
					$\hfill\square$
					\begin{rem}\label{remR'}
						It should be noted that the above proposition also holds with $\eta_R$ and $N_R$ replaced by $\eta_R-\eta_R'$ and $N_{R,R'}$ for $R'>R\geq 1$, where $N_{R,R'}$ denoted the set $N_{R,R'}=\{x\in \mathbb{R}^d; R\leq |x|\leq 2R' \}$.  
					\end{rem}
					\begin{prop}\label{proR1}
						Let $d\geq 3$ and $s$ be a nonnegative integer satisfying $s\geq [\frac{d}{2}]+2$ and let $\ell$ be an integer satisfying $\ell\geq 1$. Assume that 
						\begin{align*}
							u_{0\infty}=\trans{(a_{0\infty},v_{0\infty})}\in H^s,~F_\infty=\trans{(F_\infty^0,F_\infty^1)}\in L^2(0,T';H^s)\cap C([0,T'];H^{s-1}). 
						\end{align*}
						Here $T'$ is a given positive number. Assume also that $u_{\infty}=\trans{(a_{\infty},v_\infty)}$ is the solution of system \eqref{Mainsystem}. If $\nabla\tilde{\phi}\in  C([0,T'];H^{s-1})\cap  L^2(0,T';H^{s-1}), \partial_t\tilde{\phi}\in C([0,T'];H^{s-1})$, $\tilde{v}\in C([0,T'];H^s)\cap L^2(0,T';H^{s})$ and that $u_{\infty}$ satisfies 
						\begin{align*}
							a_{\infty} \in C([0,T'];H^s),~v_\infty\in C([0,T'];H^s)\cap L^2(0,T';H^s),
						\end{align*}
						Then there exists a positive constant $\kappa_2$ such that the estimate 
						\begin{align}
							&\frac{\mathrm{d}}{\mathrm{d}t}E_{\ell}^s[u_{\infty}](t)+\kappa_2 D_{\ell}^s[u_{\infty}](t)\nonumber\\
							&\leq  \varepsilon \|a_\infty\|_{L^2_\ell}^2 +C(\|\partial_t\tilde\phi\|_{H^{s-1}}+\|\nabla\tilde{\phi}\|_{H^{s-1}}+\|\nabla\tilde{\phi}\|_{H^{s-1}}^2+\|\tilde{v}\|_{H^s}+\|\tilde{v}\|_{H^s}^2)|u_{\infty}|_{H_\ell^{s}}^2\nonumber\\
							&+C|u_\infty|_{H_{\ell-1}^s}^2+C|F_{\infty}|_{H_\ell^s}^2\label{113001}
						\end{align}
						holds on $(0,T')$, where  $C>0$ is a constant independent of $T'$.
					\end{prop}
					\noindent\textbf{Proof.}
					To complete the proof of this proposition, we apply the induction method to prove
					\begin{align*}
						\eta_Ru_{\infty}\rightarrow u_{\infty}~\text{in}~ C([0,T'];H_\ell^s)\cap L^2(0,T';H_\ell^s)
					\end{align*}
					as $R\rightarrow{\infty}$. 
					It is easy to observe that
					\begin{align}
						\eta_Ru_{\infty}\rightarrow u_{\infty}~\text{in}~ C([0,T'];H^s)\cap L^2(0,T';H^s),\label{etaR}
					\end{align}
					and 
					\begin{align*}
						\supp(\eta_R-\eta_{R'})\subset N_{R,R'}=\{x\in\mathbb{R}^d; R\leq |x|\leq 2R' \}~\text{for}~ R'>R.
					\end{align*}
					Thus it holds 
					\begin{align*}
						|\eta_Ru_{\infty}-\eta_{R'}u_{\infty}|_{H_\ell^s}\leq C|u_{\infty}|_{H_\ell^s(N_{R,R'})}.
					\end{align*}
					Let 
					\begin{align*}
						\varphi_{\ell,R,R'}(t)&=|\eta_Ru_{\infty}(t)-\eta_{R'}u_{\infty}(t)|_{H_\ell^s},\\
						b(t)&=1+\|\partial_t\tilde{\phi}\|_{H^{s-1}} +\|\nabla\tilde{\phi}(t)\|_{H^{s-1}}+\|\nabla\tilde{\phi}(t)\|_{H^{s-1}}^2+\|\tilde{v}(t)\|_{H^{s}}+\|\tilde{v}(t)\|_{H^s}^2 \in L^1(0,T'),\\
						a_{\ell,R,R'}&= |\eta_Ru_{0\infty}-\eta_{R'}u_{0\infty}|_{H_\ell^s}^2
						+\int_0^t|\eta_RF_{\infty}-\eta_{R'}F_{\infty}|_{H_\ell^s}^2\mathrm{d}\tau\\
						&+\int_0^t(1+\|\nabla\tilde{\phi}(\tau)\|_{H^{s-1}(N_R)}+\|\nabla\tilde{\phi}(\tau)\|_{H^{s-1}(N_R)}^2\\
						&+\|\tilde{v}(\tau)\|_{H^s(N_R)}+\|\tilde{v}(\tau)\|_{H^s(N_R)}^2)\|u_{\infty}(\tau)\|_{H^s(N_{R,R'})}\mathrm{d}\tau.
					\end{align*}
					Then for $\ell=1$, we have
					\begin{align}\label{phiR}
						\varphi_{1,R,R'}(t)+\int_0^tD_1^s[\eta_Ru_{\infty}-\eta_{R'}u_{\infty}]\mathrm{d}\tau \leq C\Big\{ 
						a_{1,R,R'}(T')+\int_0^t b(\tau)\varphi_{1,R,R'}(\tau)\mathrm{d}\tau\Big\}.
					\end{align}
					By Gr\"{o}nwall inequality, we have
					\begin{align*}
						\varphi_{1,R,R'}(t)\leq Ca_{1,R,R'}(T')e^{C\int_0^{T'}b(\tau)\mathrm{d}\tau}
					\end{align*}
					for $t\in [0,T']$. Since $a_{1,R,R'}(t)\rightarrow{0}$ \text{as} $R,R'\rightarrow{\infty}$, we see that 
					$\underset{0\leq t\leq T'}{\sup}\varphi_{1,R,R'}(t)\rightarrow{0}$~\text{as}~ $R, R' \rightarrow {\infty}.$ This, together with \eqref{phiR}, yields that
					\begin{align*}
						\int_0^{T'} D_1^s[\eta_Ru_{\infty}-\eta_{R'}u_{\infty}]\mathrm{d}\tau \rightarrow 0~\text{as}~R,R'\rightarrow{\infty}.
					\end{align*}
					In view of \eqref{etaR}, we conclude that $\{\eta_Ru_{\infty}\}$ is a Cauchy sequence in $C([0,T'];H_1^s)\cap L^2([0,T'];H_1^s\times H_1^s)$ and 
					\begin{align*}
						\eta_R u_{\infty}\rightarrow{u_{\infty}}~\text{in}~ 
						C([0,T'];H_1^s)\cap L^2(0,T';H_1^s),
					\end{align*}
					$\text{as}~R\rightarrow {\infty}$. Let $R \rightarrow{\infty}$ in with $\ell=1$, we have the desired estimate \eqref{113001} in Proposition  \ref{proR1} with $\ell=1$.
					
					We next suppose that Proposition $\ref{proR1}$ holds for $\ell=m$. We will prove that it also holds for $\ell=m+1$. By \eqref{etaRU1} and Remark \ref{remR'}, we have 
					\begin{align}
						&\varphi_{m+1,R,R'}(t)+\int_0^tD_{m+1}^s[\eta_Ru_{\infty}-\eta_{R'}u_{\infty}]\mathrm{d}\tau\nonumber\\
						&\leq C\Big\{ 
						a_{m+1,R,R'}(T')+\int_0^tb(\tau)\varphi_{m+1,R,R'}(\tau)\mathrm{d}\tau
						\Big\}.
					\end{align}    
					In terms of Gr$\ddot{\text{o}}$nwall inequality, it holds 
					\begin{align}\label{varphim}
						\varphi_{m+1,R,R'}(t)\leq 
						C a_{m+1,R,R'}(T')e^{C\int_0^{T'}b(\tau)\mathrm{d}\tau}.
					\end{align}
					By the induction assumption, we see that 
					\begin{align*}
						a_{m+1,R,R'}(T') \rightarrow{0}~\text{as}~R,R' \rightarrow{\infty},
					\end{align*}
					and hence, by \eqref{varphim} we have
					\begin{align*}
						\sup_{0\leq t\leq T} \varphi_{m+1,R,R'}(t) \rightarrow{0}, ~\text{as}~R,R' \rightarrow{\infty}.
					\end{align*}
					It then follows that 
					$\eta_Ru_{\infty}$ is a Cauchy sequence in $C([0,T'];H_{m+1}^s)\cap L^2(0,T';H_{m+1}^s)$ and 
					\begin{align*}
						\eta_Ru_{\infty}\rightarrow{u_\infty}~\text{in}~C([0,T'];H_{m+1}^s)\cap L^2(0,T';H_{m+1}^s),
					\end{align*}
					as $R\rightarrow{\infty}$. Denote the right-hand side of \eqref{etaRU1} by $G_\ell^s(t)$. Then we have 
					\begin{align*}
						\frac{\mathrm{d}}{\mathrm{d}t}E_\ell^s[u_{\infty}](t)=G_\ell^s(t)
					\end{align*}
					on $(0,T')$ for some $G_\ell^s(t)\in L^1(0,T')$. Thus $E_\ell^s[U_\infty](t)$ is absolutely continuous in $t\in[0,T']$. Let $R\rightarrow{\infty}$, we have the desired estimate in Proposition \ref{proR1} with $\ell=m+1$.
					This completes the proof.
					$\hfill\square$

					\vspace{2ex}
					
					In view of Proposition $\ref{solvabilityhighpart}$, 
					$S_{{\infty},\tilde{u}}(t)$ $(t\geq 0)$ 
					and $\scr{S}_{\infty,\tilde{u}}(t)$ $(t\in [0,T])$ are defined as follows.
					
					\vspace{2ex}
					We fix an integer $s$ satisfying $s\geq [\frac{d}{2}]+2$ 
					and a function $\tilde{u}=\trans{(\tilde{a},\tilde{v})}$ satisfying 
					\begin{eqnarray}\label{periodic-assumption}
						\nabla\tilde{\phi}\in C_{per}(\mathbb{R};H^{s-1})\cap L^{2}_{per}(\mathbb{R};H^{s-1}),\quad \tilde{v}\in C_{per}(\mathbb{R};H^{s})\cap L^{2}_{per}(\mathbb{R};H^{s}).
					\end{eqnarray}
					
					\vspace{2ex} 
					The operator $S_{{\infty},\tilde{u}}(t):H^{s}_{(\infty)}\, \longrightarrow \, H^{s}_{(\infty)}$ $(t\geq 0)$ 
					is defined by 
					$$
					u_{\infty}(t)=S_{{\infty},\tilde{u}}(t)u_{0\infty} \ \ \mbox{\rm for} 
					\ \ 
					u_{0\infty}=\trans(a_{0\infty},v_{0\infty})\in H^{s}_{(\infty)},
					$$
					where $u_{\infty}(t)$ is the solution of $(\ref{eq:(2)})$ 
					with $F_{\infty}=0$; 
					and the operator 
					$\scr{S}_{\infty,\tilde{u}}(t):L^{2}(0,T;H^{s}_{(\infty)}
					\times H^{s}_{(\infty)})\,\longrightarrow\, H^{s}_{(\infty)}$ 
					$(t\in [0,T])$ is defined by  
					\begin{align*}
						u_{\infty}(t)=\scr{S}_{\infty,\tilde{u}}(t)[F_{\infty}] \ \ \mbox{\rm for} \ \ 
						F_{\infty}=\trans{(F_\infty^0,F_{\infty}^1)}
						\in L^{2}(0,T;H^{s}_{(\infty)})\cap C([0,T'];H_{(\infty)}^{s-1}),
					\end{align*}
					where $u_{\infty}(t)$ is the solution of $(\ref{eq:(2)})$ with $u_{0\infty}=0$. 
					
					\vspace{2ex}
					
					\begin{prop}\label{p8} 
						Let $d\geq 3$ and $s$ be an integer satisfying $s\geq [\frac{d}{2}]+2.$
						Assume that
						\begin{align}
							&u_{0\infty}=\trans{(a_{0\infty},v_{0\infty})}\in H^s_{(\infty),\ell},\nonumber\\
							& F_{\infty}=\trans{(F_\infty^0,F_\infty^1)}\in L^{2}(0,T';H^{s}_{(\infty),\ell})\cap C([0,T'];H_{(\infty),\ell}^{s-1}). 
						\end{align}
						for all $T'>0$ and $\tilde{u}=\trans{(\tilde a,\tilde v)}$ satisfies \eqref{periodic-assumption}. Assume also that $u_\infty=\trans{(a_\infty,v_\infty)}$ is the solution of \eqref{eq:(2)} satisfying 
						\begin{align}
							a_{\infty}\in C([0,T'];H^{s}_{(\infty)}),v_\infty\in C([0,T'];H^{s}_{(\infty)})\cap L^{2}(0,T';H^{s}_{(\infty)})
						\end{align}
						for all $T'>0$.
						
						Then there exist a positive constant $\delta$ and an energy function $\mathcal{E}^s[u_\infty]$ such that if 
						\begin{align*}
							\|\partial_t\tilde{\phi}\|_{C([0,T];H^{s-1})}+\|\nabla\tilde{\phi}\|_{C([0,T];H^{s-1})}+\|\tilde{v}\|_{C([0,T];H^{s})}\leq \delta,
						\end{align*}
						there holds the estimate 
						\begin{align}
							&\frac{d}{dt}\mathcal{E}^s[u_\infty](t)+\kappa(\|\nabla a_\infty(t)\|_{H_{\ell}^{s-1}}+\|v_\infty(t)\|_{H_\ell^{s}})\nonumber\\
							&\leq C\Big\{(\|\partial_t\tilde{\phi}\|_{H^{s-1}}+\|\nabla\tilde \phi\|_{H^{s-1}}+\|\nabla\tilde \phi\|_{H^{s-1}}^2+\|\tilde v\|_{H^s}+\|\tilde v\|_{H^s}^2)\|u_\infty\|_{H_\ell^s}^2+\|F_\infty\|_{H_\ell^s}^2\Big\},\label{112701}
						\end{align}
						on $(0,T')$ for all $T'>0$. Here $\kappa$ is a positive constant; C is a positive constant depending of $T$ bu not no $T'$; $\mathcal{E}^s[u_\infty]$ is equivalent to $\|u_\infty\|_{H_{\ell}^s}^2$, i.e.,
						\begin{align*}
							C^{-1}\|u_\infty\|_{H_\ell^s}\leq \mathcal{E}^s[u_\infty]\leq C\|u_\infty\|_{H_\ell^s}^2,
						\end{align*}
						and 
						$\mathcal{E}^s[u_\infty](t)$ is absolutely continuous in $t\in [0,T']$ for all $T'>0$.
					\end{prop}
					
					\noindent\textbf{Proof.}
					Let $U=\trans{(\Phi,V)}$ with  $\trans(\nabla \Phi, V) \in  C([0,T']; H_\ell^{s-1} \times H_\ell ^{s})\cap L^2([0,T'];H_\ell^{s-1} \times H_\ell ^{s})$. Then by Lemma \ref{lemP_1 to weightedLinfty}, we see that 
					\begin{align*}
						\|P_1B[\tilde u]U\|_{H_\ell^s} \leq  C\delta \|\trans (\nabla \Phi, V)\|_{ H_\ell^{s-1} \times H_\ell ^{s}}.
					\end{align*}
					It then follows from the Propositions \ref{E0} and \ref{proR1} that there exists a unique solution $U_\infty\in C([0,T'];H_\ell^s\times H_\ell^s)\cap L^2(0,T';H_\ell^s\times H_\ell^s)$ of 
					\begin{align}
						\partial_tU_\infty+AU_\infty+B[\tilde{u}]U_\infty=F_\infty+P_1\big(B[\tilde u]U\big),\quad U_\infty|_{t=0}=u_{0\infty},	\label{92101}
					\end{align}
					and $U_\infty$ satisfies  
					\begin{align}
						&\|U_\infty(t)\|_{H_\ell^s}^2+\kappa\int_0^t(\|\nabla \Phi_\infty(\tau)\|_{H_\ell^{s-1}}^2+\|V_\infty\|_{H_\ell^{s}}^2)d\tau\nonumber\\
						&\leq C\Big\{
						\|u_{0\infty}\|_{H_\ell^s}^2+\int_0^t\|F_\infty\|_{H_\ell^s}^2d\tau\nonumber\\
						&+\delta^2\int_0^t  \|\trans (\nabla \Phi, V)\|_{ H_\ell^{s-1} \times H_\ell ^{s}}^2  d\tau+\int_0^t b(\tau)\|U_\infty\|_{H^s}^2d\tau
						\Big\}.	\label{92102}
					\end{align}
					Here $b(\tau)=1+(\|\partial_t\tilde{\phi}\|_{H^{s-1}}+\|\nabla\tilde \phi\|_{H^{s-1}}+\|\nabla\tilde \phi\|_{H^{s-1}}^2+\|\tilde v\|_{H^s}+\|\tilde v\|_{H^s}^2)$.
					
					We set $U_\infty^{(0)}=0$ and define $U_\infty^{(n)}, (n=1,2,...)$ inductively by the solution of \eqref{92101} with $U=U_\infty^{(n-1)}$. Applying the Gr\"{o}nwall inequality to  \eqref{92102} with $U_\infty=U_\infty^{(1)}$ and $U=0$, we have 
					\begin{align*}
						\|U_\infty^{(1)}(t)\|_{L^2}\leq M_4
					\end{align*}
					for $t\in[0,T']$, where 
					\begin{align*}
						M_4=C\Big\{\|u_{0\infty}\|_{H_1^s}^2+\int_0^{T'}\|F_\infty\|_{H_\ell^s}^2d\tau\Big\}e^{C\|b\|_{L^1(0,T')}}.
					\end{align*}
					Similarly, using \eqref{92102} with $U_\infty=U_\infty^{(n)}-U_\infty^{(n-1)}$ and $U=U_\infty^{(n-1)}-u_\infty^{(n-2)}$ for $n=2,3,...$, one can inductively show that
					\begin{align*}
						\|U_\infty^{(n)}(t)-U_\infty^{(n-1)}\|_{H_\ell^s}&\leq \frac{M_4(CK_0\delta^2t)^{n-1}}{(n-1)!},\\
						\int_0^t\|U_\infty^{(n)}(t)-U_\infty^{(n-1)}\|_{H_\ell^s}d\tau&\leq \frac{M_4}{K_0}\Big\{
						\frac{(CK_0\delta^2t)^{n-1}}{(n-1)!}+\frac{\|b\|_{L^1(0,T')}}{\delta^2}\frac{(CK_0\delta^2t)^n}{n!}.
						\Big\}
					\end{align*}
					Here $K_0=1+\|b\|_{L^1(0,T')}Ce^{C\|b\|_{L^1(0,T')}}$. It then follows that $U_\infty^{(n)}$ converges to a function $U_\infty$ in $C([0,T'];H_\ell^s)\cap L^2([0,T'];H_\ell^s)$ as $n\rightarrow\infty$. One can easily see that $U_\infty$ satisfies \eqref{92101} with $U=U_\infty$, i.e., $U_\infty$ is a solution of \eqref{eq:(2)} and $U_\infty\in H_{(\infty)}^s$ for all $t\in [0,T']$. By the uniqueness of solutions of \eqref{eq:(2)}, we see that $U_\infty=u_\infty$.
					
					It holds 
					\begin{align}
						&\frac{\mathrm{d}}{\mathrm{d}t}E_{\ell}^s[u_{\infty}](t)+\kappa_2 D_{\ell}^s[u_{\infty}](t)\nonumber\\
						&\leq C(\|\partial_t\tilde{\phi}\|_{H^{s-1}}+\|\nabla\tilde{\phi}\|_{H^{s-1}}+\|\nabla\tilde{\phi}\|_{H^{s-1}}^2+\|\tilde{v}\|_{H^s}+\|\tilde{v}\|_{H^s}^2)|u_{\infty}|_{H_\ell^{s}}^2\nonumber\\
						&+C|F_{\infty}|_{H_\ell^s}^2+C|u_\infty|_{H_{\ell-1}^s}^2,\label{112703}
					\end{align}
					and 
					\begin{align}\label{112702}
						&\frac{\mathrm{d}}{\mathrm{d}t}E_0^s[u_\infty](t)+\kappa_0 D_0^s[u_\infty](t)\nonumber\\
						&\leq C(\|\partial_t\tilde{\phi}\|_{H^{s-1}}+\|\nabla\tilde{\phi}\|_{H^{s-1}}+\|\nabla\tilde \phi\|_{H^{s-1}}^2+\|\tilde{v}\|_{H^{s}}+\|\tilde v\|_{H^{s}}^2)\|u_\infty\|_{H^{s}}^2+C\|F_{\infty}\|^2_{H^s}.
					\end{align}
					
					We now prove \eqref{112701} by induction on $\ell$. When $\ell=0$, it is true by \eqref{112702}. Assume that \eqref{112701} holds for $\ell=j-1$. Then by adding 
					$\frac{\kappa}{2C}\times \eqref{112703}$ to \eqref{112701} with $\ell=j-1$, we obtain that there exists a positive constant $\kappa$ such that the desired inequality \eqref{112701} for $\ell=j$ holds, where  $\mathcal{E}^s[u_\infty](t)=E_{\ell}^s[u_{\infty}](t)+\frac{\kappa}{2C}E_0^s[u_\infty](t)$. 
					
					It is directly shown that  $\mathcal{E}^s[u_\infty]$ is equivalent to $\|u_\infty\|_{H_{\ell}^s}^2$, i.e.,
					\begin{align*}
						C^{-1}\|u_\infty\|_{H_\ell^s}\leq \mathcal{E}^s[u_\infty]\leq C\|u_\infty\|_{H_\ell^s}^2.
					\end{align*}
					This completes the proof.
					$\hfill\square$
					
					\vspace{2ex}

					The operators $S_{{\infty},\tilde{u}}(t)$  and $\scr{S}_{\infty,\tilde{u}}(t)$ have the following properties.
					
					\vspace{2ex}
					\begin{prop}\label{Sinftyproperty} 
						Let $d\geq 3$ and $s$ be a nonnegative integer satisfying $s\geq [\frac{d}{2}]+2$. 
						Let $\ell$ be a nonnegative integer. 
						Assume that $\tilde u=\trans(\tilde a,\tilde v)$ satisfies \eqref{periodic-assumption}. 
						Then there exists a constant $\delta>0$ such that the following assertions hold true 
						if 
						\begin{align}
							\|\partial_t\tilde{\phi}\|_{C([0,T];H^{s-1})}+\|\nabla\tilde{\phi}\|_{C([0,T];H^{s-1})}+\|\tilde{v}\|_{C([0,T];H^{s})}\leq \delta,
						\end{align}

						\vspace{2ex}
						{\rm (i)} 
						It holds that $S_{{\infty},\tilde{u}}(\cdot)u_{0\infty}\in C([0,\infty);H_{(\infty),\ell}^{s})$ 
						for each $u_{0\infty}=\trans(a_{0\infty},v_{0\infty})\in H_{(\infty),\ell}^{s}$ 
						and there exist constants $c_1>0$ and $C>0$ such that 
						$S_{{\infty},\tilde{u}}(t)$ satisfies the estimate
						\begin{eqnarray}
							\|S_{{\infty},\tilde{u}}(t)u_{0\infty}\|_{H_{(\infty),\ell}^{s}}
							\leq Ce^{-c_1t}\|u_{0\infty}\|_{H_{(\infty),\ell}^{s}}
							\nonumber
						\end{eqnarray}
						for all $t\geq 0$ and $u_{0\infty}\in H_{(\infty),\ell}^{s}$.
						
						\vspace{2ex}
						{\rm (ii)} 
						It holds that $\scr{S}_{\infty,\tilde{u}}(\cdot)F_\infty\in C([0,T];H_{(\infty),\ell}^{s})$ 
						for each 
						$F_\infty=\trans(F_\infty^0,F_\infty^1)
						\in L^{2}(0,T;H_{(\infty),\ell}^{s})\cap C([0,T'];H^{s-1}_{(\infty),\ell})$ 
						and $\scr{S}_{\infty,\tilde{u}}(t)$ satisfies the estimate 
						\begin{eqnarray}
							\|\scr{S}_{\infty,\tilde{u}}(t)[F_{\infty}]\|_{H_{(\infty),\ell}^{s}}
							\leq 
							C\left\{\int_{0}^{t}e^{-c_1(t-\tau)}\|F_{\infty}\|_{H_{(\infty),\ell}^{s} }^{2}d\tau\right\}^{\frac{1}{2}}\nonumber
						\end{eqnarray}
						for $t\in [0,T]$ and 
						$F_{\infty}\in L^{2}(0,T;H_{(\infty),\ell}^{s})$ 
						with a positive constant $C$ depending on $T$.
						
						\vspace{2ex}
						{\rm (iii)} 
						It holds that $r_{H_{(\infty),\ell}^{s}}(S_{{\infty},\tilde{u}}(T))<1$.
						
						\vspace{2ex}
						{\rm (iv)} 
						$I-S_{{\infty},\tilde{u}}(T)$ has a bounded inverse $(I-S_{{\infty},\tilde{u}}(T))^{-1}$ on $H_{(\infty),\ell}^{s}$ 
						and  $(I-S_{{\infty},\tilde{u}}(T))^{-1}$ satisfies 
						\begin{align*}
							\|(I-S_{{\infty},\tilde{u}}(T))^{-1}u\|_{H_{(\infty),\ell}^{s}}
							\leq C\|u\|_{H_{(\infty),\ell}^{s}}\quad\mbox{for}\quad u\in H_{(\infty),\ell}^{s}.
						\end{align*}
					\end{prop}
					
					\noindent\textbf{Proof.}
					Denote 
					\begin{equation}\label{ZT}
						\begin{aligned}
							\omega_1(T) &=\frac{1}{T}\int_0^T(
							\|\partial_t\tilde\phi\|_{H^{s-1}}+\|\nabla\tilde\phi\|_{H^{s-1}}+\|\nabla\tilde\phi\|_{H^{s-1}}^2+
							\|\tilde{v}\|_{H^s}+\|\tilde{v}\|_{H^s}^2)dt,\\
							\omega_2(t)&=(\|\partial_t\tilde\phi\|_{H^{s-1}}+\|\nabla\tilde\phi\|_{H^{s-1}}+\|\nabla\tilde\phi\|_{H^{s-1}}^2+
							\|\tilde{v}\|_{H^s}+\|\tilde{v}\|_{H^s}^2)-\omega_1(T),\\
							z(t)&=\int_0^t\omega_2(\tau)\mathrm{d}\tau.
						\end{aligned}
					\end{equation}
					Note $z(t+T)=z(t)$, which implies 
					\begin{align*}
						\sup_{t\in \mathbb{R}}|z(t)|\leq \sup_{\tau \in [0,T]}|z(\tau)|\leq C\delta.
					\end{align*}
					According to Proposition  \ref{p8}  with $F_\infty=0$, we verify that 
					\begin{align*}
						&\frac{\mathrm{d}}{\mathrm{d}t}\mathcal{E}_{\ell}^s[u_{\infty}](t)+\kappa(\|\nabla a_\infty(t)\|_{H_{\ell}^{s-1}}+\|v_\infty(t)\|_{H_\ell^{s}})\\
						&\leq C(\|\partial_t\tilde\phi\|_{H^{s-1}}+\|\nabla\tilde{\phi}\|_{H^{s-1}}+\|\nabla\tilde{\phi}\|_{H^{s-1}}^2+\|\tilde{v}\|_{H^s}+\|\tilde{v}\|_{H^{s}}^2)\|u_{\infty}\|_{H_\ell^{s}}^2\\
						&\leq C\omega_2(t)\mathcal{E}_\ell^s[u_\infty](t)+C\omega_1(T) \mathcal{E}_\ell^s[u_\infty](t).
					\end{align*}
					By Lemma \ref{lemPinfty}, we see that there exists a positive constant $\kappa_3$ such that 
					\begin{align*}
						\kappa(\|\nabla a_\infty(t)\|_{H_{\ell}^{s-1}}+\|v_\infty(t)\|_{H_\ell^{s}})\geq \kappa_3\mathcal{E}_\ell^s[u_\infty](t).
					\end{align*}
					In terms of \eqref{ZT} and the above inequality, we achieve that 
					\begin{align*}
						\begin{aligned}
							\frac{\mathrm{d}}{\mathrm{d}t}\mathcal{E}_\ell^s[u_\infty](t)+\kappa_3 \mathcal{E}_\ell^s[u_\infty](t)&\leq 
							C\omega_2(t)\mathcal{E}_\ell^s[u_\infty](t)+C\omega_1(T) \mathcal{E}_\ell^s[u_\infty](t).
						\end{aligned}
					\end{align*}
					Choosing $\delta\leq \frac{\kappa_3}{2C}$, we have $\omega_1(T)\leq \delta\leq \frac{\kappa_3}{2C}$. Then it holds  
					\begin{align*}
						\frac{\mathrm{d}}{\mathrm{d}t}\mathcal{E}_\ell^s[u_\infty](t)+\frac{1}{2}\kappa_3\mathcal{E}_\ell^s[u_\infty](t)\leq C\omega_2(t)\mathcal{E}_\ell^s[u_\infty](t).
					\end{align*}
					We thus obtain 
					\begin{align*}
						\frac{\mathrm{d}}{\mathrm{d}t}
						\Big(
						e^{\frac{\kappa_3}{2}t}e^{-Cz(t)}\mathcal{E}_\ell^s[u_\infty](t)
						\Big)\leq 0.
					\end{align*}
					Integrating it with time yields
					\begin{align*}
						\mathcal{E}_\ell^s[u_\infty](t)
						&\leq \mathcal{E}_\ell^s[u_\infty](0)e^{-\frac{\kappa_3}{2}t}e^{Cz(t)}\\
						&\leq e^{-\frac{\kappa_3}{2}t+C\delta}\mathcal{E}_\ell^s[u_\infty](0)\\
						&\leq  e^{-\frac{\kappa_3}{4}t}\mathcal{E}_\ell^s[u_\infty](0).
					\end{align*}
					Consequently, choosing $c_1=\frac{\kappa_3}{4}$, we obtain 
					\begin{align*}
						\|S_{\infty,\tilde{u}}(t) u_{0\infty}\|_{H_{(\infty),\ell}^s}\leq C
						e^{-c_1t}
						\|u_{0\infty}\|_{H_{(\infty),\ell}}^s.
					\end{align*}
					This proves {\rm (i)}. The assertion {\rm (ii)} is proved similarly, and we omit the proof. 
					
					As for {\rm (iii)}, since $\tilde{u}=\trans(\tilde{a},\tilde{v})
					\in C_{per}(\mathbb{R};H^s)$, it follows from {\rm (i)} that, for each $j\in N$,
					$$
					\|(S_{\infty,\tilde{u}}(T))^ju\|_{H_{(\infty),\ell}^s}=\|S_{(\infty),\tilde{u}}(jT)u\|_{H_{(\infty),\ell}^s}\leq 
					Ce^{-c_1jT}\|u\|_{H_{(\infty),\ell}}.
					$$
					Hence, we have 
					\begin{align*}
						\|(S_{(\infty),\tilde{u}}(T))^j\|\leq Ce^{-c_1jT}.
					\end{align*}
					We thus obtain 
					\begin{align*}
						\lim_{j\rightarrow{\infty}}\|(S_{(\infty),\tilde{u}}(T))^j\|^{\frac{1}{j}}\leq \lim_{j\rightarrow{\infty}}C^{\frac{1}{j}}e^{-c_1T}=e^{-c_1T}<1.
					\end{align*}
					This show {\rm (iii)}. The assertion {\rm (iv)} is an immediate consequence of {\rm (iii)}.
					$\hfill\square$
					\vspace{2ex}
					
					Applying Proposition \ref{Sinftyproperty}, we easily obtain the following estimate for a solution $u_\infty$ of  
					(\ref{eq:(2)}) satisfying $u_\infty(0)=u_\infty(T)$. 
					
					\vspace{2ex}
					\begin{prop}\label{conclusionhighpart}
						Let $d\geq 3$ and $s$ be a nonnegative integer satisfying $s\geq [\frac{d}{2}]+2$. 
						Assume that 
						\begin{align*}
							F_{\infty}=\trans (F_{\infty}^0,F_{\infty}^1)
							\in L^{2}(0,T;H^{s}_{(\infty),d-1}) \cap C([0,T']; H^{s-1}_{(\infty),d-1}). 
						\end{align*}
						
						Assume also that $\tilde{u}=\trans(\tilde a,\tilde v)$ satisfies \eqref{periodic-assumption}. Then there exists a positive constant $\delta$ such that the following assertion holds true if
						\begin{align*}
							\|\partial_t\tilde{\phi}\|_{C([0,T];H^{s-1})}+\|\nabla\tilde{\phi}\|_{C([0,T];H^{s-1})}+\|\tilde{v}\|_{C([0,T];H^{s})}\leq \delta,
						\end{align*}
						
						The function 
						\begin{align}
							u_{\infty}(t)
							:=
							S_{\infty,\tilde{u}}(t)(I-S_{{\infty},\tilde{u}}(T))^{-1}\scr{S}_{\infty,\tilde{u}}(T)[F_{\infty}]
							+\scr{S}_{\infty,\tilde{u}}(t)[F_{\infty}]
							\label{uinftyestimate1}
						\end{align}
						is a solution of \eqref{eq:(2)} in $\scr Z^s_{(\infty),d-1}(0,T)$ 
						satisfying $u_\infty(0)=u_\infty(T)$ and the estimate
						\begin{align*}
							\|u_{\infty}\|_{\scr Z^s_{(\infty),d-1}(0,T)}
							\leq 
							C\|F_{\infty}\|_{L^{2}(0,T;H^{s}_{(\infty),d-1}\times H^{s-1}_{(\infty),d-1})}.
						\end{align*}
					\end{prop}

					%F_{1}(u_{1}+u_{\infty},g)
					%&=&
					%P_1\trans(0,\tilde{f}_{1}(u_{1}+u_{\infty},g)),\\
					%\tilde{f}_{1}(u_{1}+u_{\infty},g)&=& \mu \triangle(\phi w)+\tilde{\mu} \nabla \div (\phi w)
					%+\frac{\rho_{\ast}}{\gamma}\nabla (P^{(1)}(\rho_{\ast}\phi)\phi ^{2})\\
					%&\quad& +\gamma \div ((1+\phi)w \otimes w)-\frac{1}{\gamma}((1+\phi)g).
					
					\section{Proof of Theorem \ref{Theorem 3.1}}\label{S7}
					In this section we give a proof of Theorem \ref{Theorem 3.1}.

					\vspace{2ex}
					We first establish the estimates for the nonlinear and inhomogeneous terms 
					$F_{1}(u,g)$ and $F_\infty(u,g)$:
					\begin{align*}
						F_{1}(u,g)&=P_1\begin{pmatrix}-v\cdot\nabla a\\
							-(v\cdot\nabla v+g_3(\phi)\nabla a)+g
						\end{pmatrix}:=\begin{pmatrix}
							F_1^1(u)\\
							{F}_{1}^2(u,g)
						\end{pmatrix}
						,\\\\
						F_{\infty}(u,g)
						&=P_{\infty}\begin{pmatrix}
							-v\cdot \nabla a_{1}\\
							-(v\cdot\nabla v_1+g_3(\phi)\nabla a_1)+g
						\end{pmatrix}
						:=\begin{pmatrix}
							F_{\infty}^{1}(u) \\
							F_\infty^{2}(u,g)
						\end{pmatrix},
					\end{align*}
					where $u=\trans(a,v)$ is a function given by $u_{1}=\trans(a_1,v_1)$ and $u_{\infty}=\trans(a_\infty,v_{\infty})$ through the relation 
					\begin{eqnarray*}
						&&a=a_1+a_\infty
						,  \ \ v=v_1+v_{\infty}.
					\end{eqnarray*}
					We first state the estimates for $F_{1}(u,g)$ and $F_{\infty}(u,g)$. We note that $$g_3(\phi)=P'(1+\phi)-P'(1)=P'(e^a)-1$$ and we can rewrite the term $g_3(\phi)\nabla a_1$ as 
					$$
					g_3(\phi)\nabla a=(1+g_2(\phi)\phi)\nabla(P^{(2)}(\phi)(\phi)^2), 
					$$
					where 
					$$
					g^{(2)}(\phi)=\displaystyle\int_0^1 h(1+\theta \phi) d\theta, \ \ h(t)=\frac{1}{t}$$
					and 
					$$
					P^{(2)}(\phi) =\displaystyle\int_0^1 P'' (1+\theta \phi) d\theta
					$$
					for $\phi =e^a -1$. 
					For the estimates of the low frequency part, we recall that
					$$
					\Gamma[{F}_1](t):=S_1(t)\scr{S}_{1}(T)(I-S_1(T))^{-1}
					\begin{pmatrix}
						F^1_1 \\
						{F}^2_{1}
					\end{pmatrix}
					+\scr{S}_1(t)
					\begin{pmatrix}
						F^1_1 \\
						{F}_{1}^2
					\end{pmatrix}
					$$
					%and $\Gamma[F_1]:=\Gamma_2[F_1]+\Gamma_{\infty}[F_1]$, 
					%$\Gamma_p[F_1]\|_{L^{2}(0,T;\|F_1\|_{{\scr Y}_{(1),L^p}})}$ $p=2,\infty$.
					
					We first show the estimate of $\|  \Gamma[{F}_{1}(u,g)]\|_{\scr{Z}(0,T)}$.

					\vspace{2ex}
					\begin{prop}\label{nonlinearestimatelowpart}
						Let $u_{1}=\trans(a_1,v_1)$ and $u_{\infty}=\trans(a_\infty,v_{\infty})$ satisfy 
						$$
						\sup_{0\leq t \leq T}\|u_{1}(t)\|_{{\scr X}_{(1)}\times {\scr Y}_{(1)}}+\sup_{0\leq t \leq T}\|u_{\infty}(t)\|_{H^s_{d-1}}  \leq \frac{1}{2},
						$$
						Then it holds that
						\begin{eqnarray*}
							\| \Gamma[{F}_{1}(u,g)]\|_{\scr{Z}^{(1)}(0,T)}
							\leq 
							C\|\{u_{1},u_{\infty}\}\|_{X^s(0,T)}^2 +C\Bigl(1+\|\{u_{1},u_{\infty}\}\|_{X^s(0,T)}\Bigr)[g]_s
						\end{eqnarray*}
						uniformly for $u_{1}$ and $u_{\infty}$.
					\end{prop}

					\vspace{2ex}
					
					\noindent\textbf{Proof.} 
					For $u^{(j)}=\trans(a^{(j)},v^{(j)})$ with 
					$\phi^{(j)} =e^{a^{(j)}} -1$
					$(j=1,\infty)$, we set 
					\begin{eqnarray*}
						G_1 (u^{(1)},u^{(2)}) &=&-P_1 ( v^{(1)} \cdot \nabla a^{(2)}),\\
						G_2 (u^{(1)},u^{(2)}) &=& -P_1 ( v^{(1)} \cdot \nabla v^{(2)}),\\
						G_3(\phi,u^{(1)},u^{(2)})&=& -P_1(\nabla (P^{(1)}(\phi)\phi^{(1)}\phi^{(2)}),\\
						G_4(\phi,u^{(1)},u^{(2)})&=& -P_1((g^{(2)}(\phi))\nabla (P^{(1)}(\phi)\phi^{(1)}\phi^{(2)}),\\
						H_k (u^{(1)},u^{(2)})&=& G_k (u^{(1)},u^{(2)})+G_k (u^{(2)},u^{(1)}),  \ \ (k=1,2),\\
						H_k (\phi, u^{(1)},u^{(2)}) &=& G_k (\phi, u^{(1)},u^{(2)})+G_k (\phi, u^{(2)},u^{(1)}) \ \ (k=3,4).
					\end{eqnarray*}
					Then, $\Gamma[\tilde{F}_{1}(u,g)]$ is written as 
					\begin{eqnarray*}
						\Gamma[{F}_{1}(u,g)]&=&\sum_{k=1}^2 \left(\Gamma[G_k(u_1,u_1)] +\Gamma [H_k(u_1,u_{\infty})] +\Gamma [G_k(u_{\infty},u_{\infty})]\right) \\
						&\quad& \sum_{k=3}^4 (\Gamma[G_k(\phi,u_1,u_1)] +\Gamma [H_k(\phi,u_1,u_{\infty})] +\Gamma [G_k(\phi,u_{\infty},u_{\infty})])\\
						&\quad &+\Gamma \left[g\right].
					\end{eqnarray*}
					Applying $(\ref{propuiestimate1})$ to $\Gamma [G_1(u_1,u_1)]$, and $\Gamma [G_2(u_1,u_1)]$  we have
					$$
					\|\Gamma [G_1(u_1,u_1)]\|_{\scr{Z}_{(1)}(0,T)}+ \Gamma [G_2(u_1,u_1)] \leq C\|\{u_1,u_{\infty}\}\|_{X^s(0,T)}^2.
					$$
					As for $\Gamma [G_3(\phi, u_1,u_1)]$, we apply $(\ref{propuiestimate2})$ with  $F^{(1)}_1=P^{(1)}(\phi)\phi_1^{2}$ 
					$(|\alpha| =1)$ to obtain
					\begin{eqnarray*}
						&&\|\Gamma [G_3(u_1,u_1)]\|_{\scr{Z}_{(1)}(0,T)} \leq C\|\{u_1,u_{\infty}\}\|_{X^s(0,T)}^2.
					\end{eqnarray*}
					By $(\ref{propuiestimate1})$, we have 
					\begin{eqnarray*}
						&&\|\sum_{k=1}^2\Gamma [G_k(u_{\infty},u_{\infty})]\|_{\scr{Z}_{(1)}(0,T)} \leq C\|\{u_1,u_{\infty}\}\|_{X^s(0,T)}^2,\\
						&&\|\Gamma [G_3(\phi,u_{\infty},u_{\infty})]\|_{\scr{Z}_{(1)}(0,T)} \leq C\|\{u_1,u_{\infty}\}\|_{X^s(0,T)}^2.
					\end{eqnarray*}
					By $(\ref{propuiestimate1})$, we also have 
					\begin{eqnarray*}
						&&\|\sum_{k=1}^2\Gamma [G_k(u_1,u_{\infty})]\|_{\scr{Z}_{(1)}(0,T)} \leq C\|\{u_1,u_{\infty}\}\|_{X^s(0,T)}^2,\\
						&&\|\Gamma [G_3(\phi,u_1,u_{\infty})]\|_{\scr{Z}_{(1)}(0,T)} \leq C\|\{u_1,u_{\infty}\}\|_{X^s(0,T)}^2.
					\end{eqnarray*}
					Concerning $\Gamma [g]$, we see from $(\ref{propuiestimate1})$  that
					\begin{eqnarray*}
						\|\Gamma[g]\|_{\scr{Z}_{(1)}(0,T)}
						\leq C[g]_s.
					\end{eqnarray*}
					Therefore, we find that 
					\begin{eqnarray*}
						\| \Gamma[{F}_{1}(u,g)]\|_{\scr{Z}^{(1)}(0,T)}
						\leq 
						C\|\{u_{1},u_{\infty}\}\|_{X^s(0,T)}^2 +C[g]_s.
					\end{eqnarray*} 
					This completes the proof.
					$\hfill\square$

					\vspace{2ex}

					We next show the estimates  for the nonlinear and inhomogeneous terms of the high frequency part.
					
					\vspace{2ex}
					
					\begin{prop}\label{nonlinearestimateinfty}
						Let $u_{1}=\trans(a_1,v_1)$ and $u_{\infty}=\trans(a_\infty,v_{\infty})$ satisfy 
						$$
						\sup_{0\leq t \leq T}\|u_{1}(t)\|_{{\scr X}_{(1)}\times {\scr Y}_{(1)}}+\sup_{0\leq t \leq T}\|u_{\infty}(t)\|_{H^s_{d-1}} \leq \frac{1}{2}.
						$$
						Then it holds that
						\begin{eqnarray*}
							\lefteqn{\| F_{\infty}(u,g)\|_{L^2(0,T;H^s_{d-1}\times H^{s-1}_{d-1})}}\\
							&\leq &
							C\|\{u_{1},u_{\infty}\}\|_{X^s(0,T)}^2 +C[g]_s
						\end{eqnarray*}
						uniformly for $u_{1}$ and $u_{\infty}$.
					\end{prop}
					
					\vspace{2ex}
					
					\noindent\textbf{Proof.} %We here estimate only $P_{\infty}(\phi\del_x  \phi)$ due to the slow decay of $a_1$ as $|x|\rightarrow \infty$.  
					By a straightforward application of Lemma $\ref{lem2.1.}$, Lemma $\ref{lem2.3.}$, Lemma $\ref{lemP_1 weightedLinfty}$ and Lemma $\ref{lemPinfty}$, we thus arrive at 
					\begin{eqnarray*}
						\lefteqn{\|F_\infty^1(u)\|_{H^s_{d-1}}
						}\\
						&
						\leq &C
						\|(1+|x|)^{d-1}\nabla a_{1}\|_{L^{\infty}}\times 
						(\|(1+|x|)^{d-1}\nabla v_{1}\|_{L^{\infty}}+\|\nabla v_{1}\|_{L^2}+\|v_{\infty}\|_{H^{s}_{d-1}}), \\
						\lefteqn{\|{F}^2_{\infty}(u,g))\|_{H^{s-1}_{d-1}}}
						\\
						&
						\leq &C\Bigl\{
						\Bigl(\sum_{j=0}^{1}\|(1+|x|)^{d-2+j}\nabla^{j}a_{1}\|_{L^{\infty}}+\|a_{\infty}\|_{H^s_{d-1}}\Bigr)
						\Bigl(\sum_{j=1}^{2}\|(1+|x|)^{j-1}\nabla^{j}a_{1}\|_{L^2}+\|a_{\infty}\|_{H^s_{d-1}}\Bigr)\\
						& \quad&
						\quad + \|(1+|x|)^{d-1}\nabla v_{1}\|_{L^{\infty}}(\| v_\infty\|_{H^{s}}+\| v_1\|_{L^2})+\|g\|_{H^{s}_{d-1}}.
					\end{eqnarray*}
					Integrating these inequalities on $(0,T)$, we obtain the desired estimate. 
					This completes the proof.
					$\hfill\square$
					
					\vspace{2ex}
					
					We next estimate $F_{1}(u^{(1)},g)-F_{1}(u^{(2)},g)$.
					
					\begin{prop}\label{nonlinearestimatesalowpart}
						Let $u^{(k)}_{1}=\trans(a^{(k)}_1,v^{(k)}_1)$ and $u^{(k)}_{\infty}=\trans(a^{(k)}_{\infty},v^{(k)}_{\infty})$ satisfy 
						$$\sup_{0\leq t \leq T}\|u^{(k)}_{1}(t)\|_{{\scr X}_{(1)}\times {\scr Y}_{(1)}}
						+\sup_{0\leq t \leq T}\|u^{(k)}_{\infty}(t)\|_{H^s_{n-1}}
						+\sup_{0\leq t \leq T}\|\phi^{(k)}(t)\|_{L^\infty}\leq \frac{1}{2},
						$$ 
						where  $\phi^{(k)}=\phi^{(k)}_1 +\phi^{(k)}_{\infty}$ $(k=1,2)$. Then it holds that
						\begin{eqnarray*}
							\lefteqn{\| \Gamma[{F}_{1}(u^{(1)},g)-{F}_{1}(u^{(2)},g)]\|_{\scr{Z}^{(1)}(0,T)}}
							\\
							&\leq& 
							C\sum_{k=1}^2\|\{u^{(k)}_{1},u^{(k)}_{\infty}\}\|_{X^{s}(0,T)}\|\{u^{(1)}_{1}-u^{(2)}_{1},u^{(1)}_{\infty}-u^{(2)}_{\infty}\}\|_{X^{s-1}(0,T)}\\
							&\quad & \quad +C[g]_s \|\{u^{(1)}_{1}-u^{(2)}_{1},u^{(1)}_{\infty}-u^{(2)}_{\infty}\}\|_{X^{s-1}(0,T)}
						\end{eqnarray*}
						uniformly for $u^{(k)}_{1}$ and $u^{(k)}_{\infty}$.
					\end{prop}

					\vspace{2ex}
					
					Proposition $\ref{nonlinearestimatesalowpart}$ can be proved in a similar manner to the proof of Proposition $\ref{nonlinearestimatelowpart}$; and 
					we omit the proof.
					
					\vspace{2ex}
					
					We next estimate $F_{\infty}(u^{(1)},g)-F_{\infty}(u^{(2)},g)$.

					\begin{prop}\label{nonlinearestimatesahighpart}
						Let $u^{(k)}_{1}=\trans(a^{(k)}_1,v^{(k)}_1)$ and $u^{(k)}_{\infty}=\trans(a^{(k)}_{\infty},v^{(k)}_{\infty})$ satisfy 
						$$\sup_{0\leq t \leq T}\|u^{(k)}_{1}(t)\|_{{\scr X}_{(1)}\times {\scr Y}_{(1)}}
						+\sup_{0\leq t \leq T}\|u^{(k)}_{\infty}(t)\|_{H^s_{-1}}
						+\sup_{0\leq t \leq T}\|\phi^{(k)}(t)\|_{L^\infty}\leq \frac{1}{2},
						$$ 
						where $\phi^{(k)}=\phi^{(k)}_1 +\phi^{(k)}_{\infty}$ $(k=1,2)$. Then it holds that
						\begin{eqnarray*}
							\lefteqn{\| F_{\infty}(u^{(1)},g)-F_{\infty}(u^{(2)},g)]\|_{L^2(0,T;H^{s-1}_{d-1}\times H^{s-2}_{d-1})}}
							\\
							&\leq& 
							C\sum_{k=1}^2\|\{u^{(k)}_{1},u^{(k)}_{\infty}\}\|_{X^{s}(0,T)}\|\{u^{(1)}_{1}-u^{(2)}_{1},u^{(1)}_{\infty}-u^{(2)}_{\infty}\}\|_{X^{s-1}(0,T)}\\
							&\quad & \quad +C[g]_s \|\{u^{(1)}_{1}-u^{(2)}_{1},u^{(1)}_{\infty}-u^{(2)}_{\infty}\}\|_{X^{s-1}(0,T)}
						\end{eqnarray*}
						uniformly for $u^{(k)}_{1}$ and $u^{(k)}_{\infty}$.
					\end{prop} 
					
					\vspace{2ex}
					
					Proposition $\ref{nonlinearestimatesahighpart}$ directly follows from Lemmas $\ref{lem2.1.}$--$\ref{lem2.3.}$, Lemma $\ref{lemP_1 weightedLinfty}$, Lemma $\ref{lemPinfty}$ and Lemma \ref{lemPinftyweight} in a similar manner to  the proof of Proposition $\ref{nonlinearestimateinfty}$.
					
					\vspace{2ex}

					We next show the following estimate which will be used in the proof of Proposition $\ref{iterationfirst}$. 
					
					\vspace{2ex}

				\vspace{2ex}

				\vspace{2ex}
				To prove Theorem \ref{Theorem 3.1},
				we next show the existence of a solution $\{u_{1},u_\infty\}$ 
				of \eqref{eq:(3)} on $[0,T]$ 
				satisfying $u_{1}(0)=u_{1}(T)$ and $u_{\infty}(0)=u_{\infty}(T)$ by an iteration argument.
				
				\vspace{2ex}
				For $N=0,$ 
				we define $u_{1}^{(0)}=\trans(a_{1}^{(0)},v_1^{(0)})$ 
				and $u_{\infty}^{(0)}=\trans(a_\infty^{(0)},v_\infty^{(0)})$ 
				by
				\begin{eqnarray}
					\left\{
					\begin{array}{lll}
						u_{1}^{(0)}(t)
						&=
						S_{1}(t)\scr S_{1}(T)[(I-S_1(T))^{-1}\mathbb{G}_{1}]+{\scr S}_{1}(t)[\mathbb{G}_{1}],\label{iteration,eq(0)}
						\\
						u_{\infty}^{(0)}(t)
						&=
						S_{\infty,0}(t)(I-S_{{\infty},0}(T))^{-1}\scr{S}_{\infty,0}(T)[\mathbb{G}_{\infty}]
						+\scr{S}_{\infty,0}(t)[\mathbb{G}_{\infty}],
					\end{array}
					\right.
				\end{eqnarray}
				where $t\in [0,T]$, 
				$\mathbb{G}=\trans(0,g(x,t))$, $\mathbb{G}_{1}=P_{1}\mathbb{G}$ 
				, $\mathbb{G}_{\infty}=P_{\infty}\mathbb{G}$, $a^{(0)}=a^{(0)}_1+a^{(0)}_{\infty}$ and $v^{(0)}=v^{(0)}_1+v^{(0)}_{\infty}$. 
				Note that $u^{(0)}_{1}(0)=u^{(0)}_{1}(T)$ and $u^{(0)}_{\infty}(0)=u^{(0)}_{\infty}(T)$.
				
				For $N\geq 1$, 
				we define 
				$u_{1}^{(N)}=\trans(a_1^{(N)},v_1^{(N)})$ 
				and $u_{\infty}^{(N)}=\trans(a_\infty^{(N)},v_{\infty}^{(N)})$, inductively, by
				\begin{eqnarray}
					\left\{
					\begin{array}{lll|}
						u_{1}^{(N)}(t)
						&
						=S_{1}(t)\scr S_{1}(T)[(I-S_1(T))^{-1}F_{1}(u^{(N-1)},g)]
						+{\scr S}_{1}(t)[F_{1}(u^{(N-1)},g)],\label{iteration,eq(p)}
						\\
						u_{\infty}^{(N)}(t)
						&
						=
						S_{\infty,u^{(N-1)}}(t)(I-S_{{\infty},u^{(N-1)}}(T))^{-1}
						\scr{S}_{\infty,u^{(N-1)}}(T)[F_{\infty}(u^{(N-1)},g)]
						\\
						&
						\quad
						+\scr{S}_{\infty,u^{(N-1)}}(t)[F_{\infty}(u^{(N-1)},g)],
					\end{array}
					\right.
				\end{eqnarray}
				where $t\in[0,T]$, $u^{(N-1)}=u_{1}^{(N-1)}+u_{\infty}^{(N-1)}$, $u^{(N-1)}_1=\trans(a^{(N-1)}_1,v^{(N-1)}_1)$, $a^{(N)}=a^{(N)}_1+a^{(N)}_{\infty}$ and $v^{(N)}=v^{(N)}_1+v^{(N)}_{\infty}$. 
				Note that $u^{(N)}_{1}(0)=u^{(N)}_{1}(T)$ and $u^{(0)}_{\infty}(0)=u^{(0)}_{\infty}(T)$.

				\vspace{2ex}
				\begin{prop}\label{iterationfirst}
					There exists a constant $\delta_{1}>0$ such that 
					if $[g]_{s}\leq \delta_{1}$, 
					then there holds the estimates 
					$$
					\|\{u^{(N)}_{1},u^{(N)}_{\infty}\}\|_{X^s(0,T)}\leq C_1[g]_{s}
					\leqno{\rm (i)} 
					$$
					for all $N\geq 0$, and 
					$$
					\begin{array}{rcl}
						\lefteqn{
							\|\{u^{(N+1)}_{1}-u^{(N)}_{1},u^{(N+1)}_{\infty}-u^{(N)}_{\infty}\}\|_{X^{s}(0,T)}
						}\\[2ex]
						& \leq &  
						C_1[g]_{s}\|\{u^{(N)}_{1}-u^{(N-1)}_{1},u^{(N)}_{\infty}-u^{(N-1)}_{\infty}\}\|_{X^{s}(0,T)}
					\end{array}
					\leqno{\rm (ii)} 
					$$
					for $N\geq 1$. 
					Here $C_1$ is a constant independent of $g$ and $N$. 
				\end{prop}
				
				\vspace{2ex}
				\noindent\textbf{Proof.} 
				If $[g]_{s}\leq \delta_{1}$ for sufficiently small $\delta_{1}$, the estimate (i) and (ii) follows from Propositions \ref{S1}, \ref{conclusionhighpart}, \ref{nonlinearestimatelowpart}, \ref{nonlinearestimateinfty}. 
					$\hfill\square$ 
					
					\vspace{2ex}
					Before going further, we introduce a notation. We denote by $B_{X^k(a,b)}(r)$ 
					the closed unit ball of $X^k(a,b)$ centered at $0$ with radius $r$, i.e.,
					$$
					B_{X^k(a,b)}(r)
					=
					\left\{
					\{u_{1},u_\infty\}\in X^k(a,b); \|\{u_{1},u_\infty\}\|_{X^k(a,b)}\leq r
					\right\}.
					$$
					
					\vspace{2ex}
					\begin{prop}\label{propfirstsolution}
						There exists a constant $\delta_{2}>0$ such that 
						if $[g]_{s}\leq \delta_{2}$, 
						then the system \eqref{eq:(3)} has a unique solution $\{u_{1},u_{\infty}\}$ 
						on $[0,T]$ in $B_{X^s(0,T)}(C_1[g]_s)$ satisfying $u_{1}(0)=u_{1}(T)$ and $u_{\infty}(0)=u_{\infty}(T)$.
						The uniqueness of solutions of \eqref{eq:(3)} on $[0,T]$ 
						satisfying $u_{1}(0)=u_{1}(T)$ and $u_{\infty}(0)=u_{\infty}(T)$ 
						holds in $B_{X^s(0,T)}(C_1\delta_2)$. 
					\end{prop}
					
					\vspace{2ex}
					\noindent\textbf{Proof.} 
					Let $\delta_2=\min\{\delta_1,\frac{1}{2C_1}\}$ 
					with $\delta_1$ given in Propositions \ref{iterationfirst}. 
					By Propositions \ref{iterationfirst}, 
					we see that 
					if $[g]_s\leq \delta_2$, 
					then $u^{(N)}_{1}=\trans(a_{1}^{(N)},v_1^{(N)})$ and $u^{(N)}_{\infty}=\trans(a_\infty^{(N)},v_{\infty}^{(N)})$
					converge to $u_{1}=\trans(\phi_1,m_1)$ and $u_{\infty}=\trans(a_\infty,v_{\infty})$, respectively, 
					in the sense 
					$$
					\{u^{(N)}_{1},u_\infty^{(N)}\} \ \rightarrow \ \{u_{1},u_\infty\} 
					\ \mbox{\rm in } \ X^{s}(0,T). 
					$$
					%$$
					%u_\infty^{(N)}=\trans(a_\infty^{(N)},v_\infty^{(N)}) 
					%	\ \rightarrow \ u_\infty=\trans(a_\infty,v_\infty) 
					%	\ \mbox{\rm $*$-weakly \ in} \ L^\infty(0,T;H^s_{(\infty),d-1}), 
					%$$
					%$$
					%w_\infty^{(N)} \ \rightarrow \ w_\infty 
					%\ \mbox{\rm weakly \ in} \ 
					%L^2(0,T;H^{s}_{(\infty),d-1})\cap H^1(0,T;H^{s-1}_{(\infty),d-1}).
					%$$
					It is not difficult to see that $\{u_{1},u_\infty\} $ is a solution of \eqref{eq:(3)}
					satisfying $u_{1}(0)=u_{1}(T)$ and $u_{\infty}(0)=u_{\infty}(T)$. This completes the proof. 
					$\hfill\square$
					
					%It remains to prove $u_\infty=\trans(v_\infty,w_{\infty})\in C([0,T];H_{d-1}^{s})$, 
					%which implies $\{u_{1},u_\infty\}\in B_{X^s(0,T)}(C_1[g]_s)$ 
					%with $u_{1}(0)=u_{1}(T)$ and $u_{\infty}(0)=u_{\infty}(T)$. 
					%But this can be shown in the same way as in the proof of \cite[Proposition. 8.4]{Kagei-Tsuda}. 

					\vspace{2ex}

					\vspace{2ex}
					
					We can now construct a time periodic solution of \eqref{eq:(4.2.2)}-\eqref{eq:(4.2.3)} in the same argument as that in \cite{Kagei-Tsuda}. 
					As in \cite{Kagei-Tsuda}, based on the estimates in sections 6-8, one can show the following proposition on the unique existence of solutions of the 
					initial value problem. 
					
					\vspace{2ex}
					
					\begin{prop}\label{iterationinitialvaluemodosu}
						Let $h\in \mathbb{R}$ and let $u_{0}=u_{01}+u_{0\infty}$ 
						with $u_{01}\in {\scr X}_{(1)}\times {\scr Y}_{(1)}$ and $u_{0\infty}\in H_{(\infty),d-1}^s$. 
						Then there exist constants $\delta_{4}>0$ and $C_3>0$ such that 
						if 
						$$
						M(U_{01},U_{0\infty},g)
						:=
						\|u_{01}\|_{\scr{X}^{(1)}\times \scr{Y}^{(1)}}+\|u_{0\infty}\|_{H^{s}_{(\infty),d-1}}+[g]_s 
						\leq \delta_{4},
						$$
						there exists a solution $\{u_{1},u_{\infty}\}$ 
						of the initial value problem for \eqref{eq:(4.2.2)}-\eqref{eq:(4.2.3)} 
						on $[h,h+T]$ in $B_{ X^s(h,h+T)}(C_3M(u_{01},u_{0\infty},g))$ 
						satisfying the initial condition $u_j|_{t=h}=u_{0j}$ $(j=0,\infty)$. 
						The uniqueness for this initial value problem holds in $B_{X^s(h,h+T)}(C_3\delta_4)$. 
					\end{prop}
					
					\vspace{2ex}
					
					By using Proposition \ref{iterationinitialvaluemodosu}, one can extend $\{u_1,u_{\infty}\}$ periodically on 
					$\mathbb{R}$ as a time periodic solution of  \eqref{eq:(4.2.2)}-\eqref{eq:(4.2.3)}. Since the argument for extension is the same as that given in \cite{Kagei-Tsuda}, 
					we here omit the details. Consequently, we obtain Theorem \ref{Theorem 3.1}. This completes the proof.

					\vspace{2ex}
					
					\section{Proof of Theorem \ref{Theorem 3.2}.} \label{S8}       
					The local existence theorem is verified by a similar iteration argument to that in Proposition \ref{solvabilityhighpart} based on  Lemma \ref{lemA2-0}. Indeed, we consider the following system; 
					\begin{eqnarray}
						\left\{
						\begin{array}{ll}
							\partial_{t} {\psi}^{(0)}+\div w^{(0)} + v_{per}\cdot\nabla {\psi}^{(0)} + w^{(0)} \cdot \nabla a_{per} =0,
							\\
							\partial_{t} {w}^{(0)}+ w^{(0)} + \nabla \psi^{(0)} +v_{per}\cdot\nabla {w}^{(0)}+ {w}^{(0)}\cdot\nabla v_{per}
							%\\
							%\quad
							%+g_3 (a_{per}) \nabla \psi^{(0)} 
							= 0,
							\label{stability-2}
						\end{array}
						\right.
					\end{eqnarray}
					and for $1 \leq N$ 
					\begin{eqnarray}
						\left\{
						\begin{array}{ll}
							\partial_{t} {\psi}^{(N)}+\div w^{(N)} + v_{per}\cdot\nabla {\psi}^{(N)} + w^{(N)} \cdot \nabla a_{per} + w^{(N-1)} \cdot \nabla \psi^{(N)}=0,
							\\
							\partial_{t} {w}^{(N)}+ w^{(N)} + \nabla \psi^{(N)} +v_{per}\cdot\nabla {w}^{(N)}+ {w}^{(N)}\cdot\nabla v_{per}
							\\
							\quad
							+ g_4 (a_{per}, \psi^{(N)}) \nabla \psi^{(N)} + g_5(a_{per}, \psi^{(N)}) \nabla a_{per} \BLACK + w^{(N-1)}\cdot\nabla {w}^{(N)}%+g_3(\psi^{(N-1)}) \nabla \psi^{(N)}
							= 0.
							\label{stability-3}
						\end{array}
						\right.
					\end{eqnarray}
					Then the local existence of a solution to \eqref{stability-2} is directly obtained by Lemma \ref{lemA2-0}. A similar iteration argument to that in Proposition \ref{solvabilityhighpart} based on  Lemma \ref{lemA2-0} derives a fixed point of \eqref{stability-3}, 
					{\rm i.e., } a solution to \eqref{stability} in $C([0,T']; H^s)$. We note that by $\eqref{stability}_1$ $\del_t \psi \in C([0,T]; H^{s-1})$ and 
					$$
					\|\trans(\psi, w)\|_{C([0,T]; H^s)} + \|\del_t \psi\|_{C([0,T]; H^{s-1})} \leq C\|\trans(\psi_0, w_0)\|_{H^s}, 
					$$
					where $T>0$ is a fixed time and $C$ depends on $T$.  
					
					We note that Remark \ref{note-a} and Lemmas \ref{lem2.1.} imply that if $$
					\|u\|_{C([0,T]; H^s)} \leq \delta, 
					$$
					then $$
					\|g_j (a_{per}, \psi )\|_{L^\infty} + \|g_j (a_{per}, \psi)\|_{H^s} \leq C_\delta \|\nabla \psi\|_{H^{s-1}} 
					$$
					for $j=4,5$
					\BLACK 
					
					Hence we show the energy estimate $\eqref{energy-stability}_2$ which implies $\eqref{energy-stability}_1$  and $\eqref{energy-stability}_3$. 
					To prove $\eqref{energy-stability}_2$, we divide it into three steps.

					\begin{lem} (Step 1: $0$-th order energy inequality) 
						We assume that there exists a solution $u=\trans(\psi, w)$ to \eqref{stability} satisfying 
						$$
						\|u\|_{C([0,T]; H^s)} + \|\del_t \psi \|_{C([0,T]; H^{s-1})} \leq \delta_1 \leq 1.  
						$$ 
						Then it holds that 
						\begin{align}
							\begin{aligned}\label{zeroenergy-stability}
								&\frac{1}{2}\frac{\mathrm{d}}{\mathrm{d}t}\|u\|_{L^2}^2 +\|w\|_{L^2}^2\\
								&\leq \epsilon \|w\|_{L^2}^2 + C\delta_0^2\|\nabla w\|_{L^2}^2+ C\delta_1^2 \|\nabla \psi\|_{L^2}^2.
							\end{aligned}
						\end{align} 
						
					\end{lem}
					
					\begin{proof} We assume that 
						$$
						\|u\|_{C([0,T]; H^s)} + \|\del_t \psi \|_{C([0,T]; H^{s-1})} \leq \delta_1 \leq 1.  
						$$

						Multiplying \eqref{stability} by $\psi$ and $w$ respectively, and integrating  over $\mathbb{R}^d$ by parts then adding them together, we obtain
						\begin{align*}
							&\frac{1}{2}\frac{\mathrm{d}}{\mathrm{d}t}\|\psi\|_{L^2}^2+(\div w, \psi)\nonumber\\
							&=-(v_{per}\cdot\nabla \psi, \psi) -( w\cdot\nabla a_{per}, \psi)+ (f^1(u), \psi)\\
							&:=\sum_{k=1}^3 \tilde{I}_k. 
						\end{align*} 
						\
						By the Hardy inequality and  Sobolev inequality, we obtain that 
						\begin{align}
							\begin{aligned}
								\tilde{I}_1 & \leq C\|v_{per}\|_{L^\infty_{d-2}}\|\nabla \psi\|_{L^2}, \\
								\tilde{I}_2 & \leq C\|\nabla a_{per}\|_{L^\infty_{d-1}}\|\nabla \psi\|_{L^2}\|w\|_{L^2}\\
								&\leq \epsilon \|w\|_{L^2}^2 + \delta_0^2 \|\nabla \psi\|_{L^2}^2, \\
								\tilde{I}_3 &\leq C\|\psi\|_{L^\infty}\|w\|_{L^2}\|\nabla \psi\|_{L^2}\\
								& \leq \epsilon \|w\|_{L^2}^2 + C\delta_1^2 \|\nabla \psi\|_{L^2}^2.
							\end{aligned}
						\end{align}
						On the other hand, 
						\begin{align}
							\frac{1}{2}\frac{\mathrm{d}}{\mathrm{d}t}\|w\|_{L^2}^2+(\nabla\psi, w)+\|w\|_{L^2}^2
							=\sum_{k=1}^5 \tilde{J}_k, 
						\end{align} 
						where 
						\begin{align}
							\begin{aligned}
								\tilde{J}_1 & = -(v_{per}\cdot \nabla w, w), \\
								\tilde{J}_2 &  = -(w\cdot \nabla v_{per}, w), \\
								\tilde{J}_3 &= -(g_4 (a_{per}, \psi)\nabla \psi, w), \\
								\tilde{J}_4 &= -(g_5(a_{per}, \psi)\nabla a_{per}, w), \\
								\tilde{J}_5 &= -(w\cdot \nabla w, w).
								%J_6 &= -(g_3 (\psi) \nabla \psi, w). 
							\end{aligned}
						\end{align}
						The Hardy inequality and  Sobolev inequality yield similarly that 
						\begin{align*}
							\sum_{k=1}^5\tilde{J}_k  \leq \epsilon \|w\|_{L^2}^2 + C\delta_0^2\|\nabla w\|_{L^2}^2 + C\delta_1^2 \|\nabla\psi\|_{L^2}^2. 
						\end{align*}
						Consequently, we get that 
						\begin{align}\label{zero-energy-result-stability}
							&\frac{1}{2}\frac{\mathrm{d}}{\mathrm{d}t}\|u\|_{L^2}^2+\|w\|_{L^2}^2\nonumber\\
							&\leq \epsilon \|w\|_{L^2}^2 + C\delta_0^2\|\nabla w\|_{L^2}^2+ C\delta_1^2 \|\nabla \psi\|_{L^2}^2. 
						\end{align} 
						
					\end{proof} 
					
					\vspace{2ex}
					\begin{lem} (Step 2: Higher order energy inequality) 
						We assume that there exists a solution $u=\trans(\psi, w)$ to \eqref{stability} satisfying 
						$$
						\|u\|_{C([0,T]; H^s)} + \|\del_t \psi \|_{C([0,T]; H^{s-1})} \leq \delta_1 \leq 1.  
						$$ 
						Then there holds that for $1 \leq |\alpha| \leq s$ 
						\begin{align}
							\begin{aligned}\label{energy-higer-order-stability}
								&\frac{1}{2}\frac{\mathrm{d}}{\mathrm{d}t}\|\del_x^\alpha \psi \|_{L^2}^2+(\del_x^\alpha \psi, \del_x^\alpha\div w)\\
								&\leq \epsilon \|w\|_{L^2}^2 + C(\delta_0^2+\delta_1^2) \|\nabla \psi\|_{L^2}^2, \\
								&\frac{1}{2}\frac{\mathrm{d}}{\mathrm{d}t}\|\del_x^\alpha w \|_{L^2}^2-(\del_x^\alpha \psi, \del_x^\alpha\div w)+\|\partial_x^\alpha w \|_{L^2}^2 + \frac{1}{2}\frac{d}{dt} \int_{\mathbb{R}^d} g_4(a_{per}, \psi)|\del_x^\alpha \psi|^2 dx \\
								& \leq \epsilon \|w\|_{H^s}^2  +C(\delta_0+\delta_1) \|\nabla \psi\|_{H^{s-1}}^2. 
							\end{aligned}
						\end{align} 
					\end{lem}
					\vspace{2ex}
					
					\begin{proof}
						
						For $1 \leq |\alpha| \leq s $, applying $\partial_x^\alpha $ to \eqref{stability} and multiplying by $\partial_x^\alpha u$, and integrating over $\mathbb{R}^3$, then adding them together, we obtain 
						
						\begin{align}
							\begin{aligned}
								&\frac{1}{2}\frac{\mathrm{d}}{\mathrm{d}t}\|\del_x^\alpha \psi \|_{L^2}^2+(\del_x^\alpha \psi, \del_x^\alpha\div w)+\|\partial_x^\alpha\psi \|_{L^2}^2 =\sum_{k=1}^3 \tilde{I}_{\alpha, k},\\
								&\frac{1}{2}\frac{\mathrm{d}}{\mathrm{d}t}\|\del_x^\alpha w \|_{L^2}^2-(\del_x^\alpha \psi, \del_x^\alpha\div w)+\|\partial_x^\alpha w \|_{L^2}^2 =\sum_{k=1}^5 \tilde{J}_{\alpha, k},\\
							\end{aligned}
						\end{align}	
						where 
						\begin{align}
							\begin{aligned}
								\tilde{I}_{\alpha, 1} & = -(\del_x^\alpha (v_{per}\cdot\nabla \psi), \del_x^\alpha\psi),\\ 
								\tilde{I}_{\alpha, 2} & = -( \del_x^\alpha (w\cdot\nabla a_{per}), \del_x^\alpha\psi), \\
								\tilde{I}_{\alpha, 3} &=  (\del_x^\alpha f^1(u), \del_x^\alpha\psi),
							\end{aligned}
						\end{align}
						and 
						\begin{align}
							\begin{aligned}
								\tilde{J}_{\alpha, 1} & = -(\del_x^\alpha(\del_x^\alpha(v_{per}\cdot \nabla w)), \del_x^\alpha w), \\
								\tilde{J}_{\alpha, 2} &  = -(\del_x^\alpha(w\cdot \nabla v_{per}), \del_x^\alpha w), \\
								\tilde{J}_{\alpha,3} &= -(\del_x^\alpha (g_4 (a_{per}, \psi)\nabla \psi),\del_x^\alpha w), \\
								\tilde{J}_{\alpha,4} &= -(\del_x^\alpha (g_5(a_{per}, \psi)\nabla a_{per}), \del_x^\alpha w), \\
								\tilde{J}_{\alpha,5} &= -(\del_x^\alpha (w\cdot \nabla w), \del_x^\alpha w).
							\end{aligned}
						\end{align}
						We see from the Sobolev inequality with Lemma \ref{lem2.3.}  that 
						\begin{align*}
							\tilde{I}_{\alpha, 1}+\tilde{I}_{\alpha, 3} \leq \epsilon \|w\|_{H^s}^2 + C(\delta_0^2+\delta_1^2) \|\nabla \psi\|_{H^{s-1}}^2. 
						\end{align*}
						On $\tilde{I}_{\alpha, 2}$, since 
						\begin{align}
							\begin{aligned}
								\tilde{I}_{\alpha,2} &= (w \cdot \nabla \del_x^\alpha a_{per}, \del_x^\alpha \psi)\\
								&\quad + (\del_x^\alpha (w\cdot \nabla a_{per})-w\cdot \nabla \del_x^\alpha a_{per}, \del_x^\alpha \psi ), 
							\end{aligned}
						\end{align}
						the Sobolev inequality with Lemma \ref{lem2.3.} verifies that 
						\begin{align}
							\begin{aligned}
								\tilde{I}_{\alpha,2} &\leq C\|w\|_{L^\infty}\|\nabla \del_x^\alpha a_{per}\|_{L^2}\|\del_x^\alpha \psi\|_{L^2} + C\|\nabla a_{per}\|_{H^{s-1}}\|w\|_{H^s}\|\del_x^\alpha \psi\|_{L^2}\\
								& \leq \epsilon \|w\|_{H^s}^2 + C\delta_0^2 \|\nabla \psi\|_{H^{s-1}}^2.
							\end{aligned}
						\end{align}
						Here we used that $\nabla a_{per} \in H^s$. On $\tilde{J}_{\alpha, k}$ $(k=1, \cdots, 5)$, the Sobolev inequality with Lemma \ref{lem2.3.} similarly yield that 
						\begin{align}
							\begin{aligned}
								\tilde{J}_{\alpha,1}+\tilde{J}_{\alpha,2}+\tilde{J}_{\alpha,5} &\leq 
								\epsilon \|w_{H^{s}}\|^2 + C(\delta_0^2 +\delta_1^2)\| w\|_{H^{s}}^2, \\
								\tilde{J}_{\alpha,4} &\leq \epsilon \|w\|_{H^s}^2 + C\delta_0^2 \|\nabla \psi\|_{H^{s-1}}^2. 
							\end{aligned}
						\end{align}
						On $\tilde{J}_{\alpha,3}$, set 
						\begin{align}
							\begin{aligned}
								\tilde{J}_{\alpha,3} &= -(g_4(a_{per},\psi) \nabla \del_x^\alpha \psi, \del_x^\alpha w) \\
								&\quad - (\del_x^\alpha (g_4 (a_{per},\psi)\nabla \psi)-g_4(a_{per},\psi) \nabla \del_x^\alpha \psi,\del_x^\alpha w)\\
								& := \tilde{J}_{\alpha,3,1}+ \tilde{J}_{\alpha,3,2}. 
							\end{aligned}
						\end{align}
						The integration by parts with substituting $\eqref{stability}_1$ derive that 
						\begin{align}
							\begin{aligned}
								\tilde{J}_{\alpha,3,1} &= (g_4(a_{per}, \psi) \del_x^\alpha \psi, \del_x^\alpha \div w) \\
								&\quad + (\nabla g_4 (a_{per}, \psi) \psi,\del_x^\alpha w)\\
								& =-(g_4(a_{per}, \psi) \del_x^\alpha \psi, \del_x^\alpha \del_t \psi)\\
								&\quad - (g_4(a_{per}, \psi) \del_x^\alpha \psi, \del_x^\alpha (v_{per}\cdot \nabla \psi ))\\
								&\quad - (g_4(a_{per}, \psi) \del_x^\alpha \psi, \del_x^\alpha (w\cdot \nabla a_{per}))\\
								&\quad - (g_4(a_{per}, \psi) \del_x^\alpha \psi, \del_x^\alpha (w\cdot \nabla \psi))\\
								&\quad +  (\nabla g_4 (a_{per},\psi) \psi,\del_x^\alpha w)\\
								&:= \sum_{\ell=1}^5\tilde{J}_{\alpha, 3,1, \ell}. 
							\end{aligned}
						\end{align}
						$\tilde{J}_{\alpha,3,1,1}$ is estimated as 
						\begin{align}
							\begin{aligned}
								\tilde{J}_{\alpha,3,1,1} &\leq -\frac{1}{2}\frac{d}{dt} \int_{\mathbb{R}^d} g_4(a_{per}, \psi)|\del_x^\alpha \psi|^2 dx \\
								& \quad + \frac{1}{2}\int_{\mathbb{R}^d}\del_t g_4(a_{per}, \psi)|\del_x^\alpha \psi|^2 dx\\
								&\leq - \frac{1}{2}\frac{d}{dt} \int_{\mathbb{R}^d} g_4(a_{per}, \psi)|\del_x^\alpha \psi|^2 dx \\
								& \quad + C(\|\del_t \phi_{per}\|_{L^\infty} + \|\del_t \psi\|_{L^\infty})\|\del_x^\alpha \psi\|_{L^2}^2\\ 
								& \leq - \frac{1}{2}\frac{d}{dt} \int_{\mathbb{R}^d} g_4(a_{per}, \psi)|\del_x^\alpha \psi|^2 dx \\
								& \quad + C(\delta_0+\delta_1)\|\del_x^\alpha \psi\|_{L^2}^2. 
							\end{aligned}
						\end{align}
						Here we used that $\sup_{t\in [0,T]}\|\del_t \phi_{per}\|_{H^{s-1}} \leq C\delta_0$ and   $\sup_{t\in [0,T]}\|\del_t \psi\|_{H^{s-1}} \leq \delta_1$. 
						
						$\tilde{J}_{\alpha,3,1,2}$ is estimated by the Sobolev inequality with Lemma \ref{lem2.3.} as 
						\begin{align}
							\begin{aligned}\label{aboid-loss-stability}
								\tilde{J}_{\alpha,3,1,2} &\leq -(g_4(a_{per}, \psi) \del_x^\alpha \psi, v_{per}\cdot \nabla \del_x^\alpha \psi )\\
								& \quad - (g_4(a_{per}, \psi) \del_x^\alpha \psi, \del_x^\alpha(v_{per}\cdot \nabla \psi )-v_{per}\cdot \nabla \del_x^\alpha \psi)\\
								&\leq C\|g_4(a_{per}, \psi)\|_{L^\infty}\|\del_x^\alpha \psi\|_{L^2} (\|\nabla v_{per}\|_{L^2}\|\del_x^\alpha \psi\|_{L^2}+\|\del_x^\alpha v_{per}\|_{L^2}\|\nabla \psi\|_{L^\infty})\\
								&\quad + C\|\div (g_4(a_{per}, \psi)v_{per})\|_{L^\infty}\|\del_x^\alpha \psi\|_{L^2}^2 \\
								&\leq C\|\nabla \psi_{per}\|_{H^{s-1}}\|\nabla v_{per}\|_{H^{s-1}}\|\nabla \psi \|_{H^{s-1}}^2\\
								&\leq C\delta_0^2 \|\nabla \psi \|_{H^{s-1}}^2. 
							\end{aligned}
						\end{align}
						It follows a similar argument as that in \eqref{aboid-loss-stability} that 
						\begin{align}
							\begin{aligned}
								\tilde{J}_{\alpha,3,1,3}+\tilde{J}_{\alpha,3,1,4} \leq \epsilon\|w\|_{H^s}^2 + C\delta_0^2 \|\nabla \psi\|_{H^s}^2.  
							\end{aligned}
						\end{align}
						The Sobolev inequality with Lemma \ref{lem2.3.} again verifies that 
						\begin{align}
							\begin{aligned}
								\tilde{J}_{\alpha,3,1,5}+\tilde{J}_{\alpha,3,2} \leq \epsilon \|w\|_{H^s}^2 + C\delta_0^2 \|\nabla \psi\|_{H^{s-1}}^2. 
							\end{aligned}
						\end{align}
						Therefore we obtain that 
						\begin{align}
							\begin{aligned}
								\tilde{J}_{\alpha,3} &\leq - \frac{1}{2}\frac{d}{dt} \int_{\mathbb{R}^d} g_4(a_{per}, \psi)|\del_x^\alpha \psi|^2 dx\\
								& \quad + \epsilon \|w\|_{H^s}^2 + C\delta_0^2\|\nabla \psi\|_{L^2}^2. 
							\end{aligned}
						\end{align}   
						
						$\tilde{J}_{\alpha,4} $ can be estimated directly by the Sobolev inequality with Lemma \ref{lem2.3.} as 
						\begin{align*}
							\tilde{J}_{\alpha,4} \leq \epsilon \|w\|_{H^s}^2 + C\delta_0^2\|\nabla \psi\|_{H^s}^2.
						\end{align*}
				Combining the above estimates together, we can prove \eqref{energy-higer-order-stability}. 
					\end{proof}
					
					\vspace{2ex}
					
					\begin{lem} 
						(Step 3: Deriving the dissipation term of density)  
						We assume that there exists a solution $u=\trans(\psi, w)$ to \eqref{stability} satisfying 
						\begin{align*}
							\|u\|_{C([0,T]; H^s)} \leq \delta_1 \leq 1.  
						\end{align*}
						Then there holds that 
						\begin{align}
							\begin{aligned}\label{energy-disspation-stability}
								&\frac{\del}{\del t}(\del_x^{\alpha-1}w, \del_x^\alpha \psi) + \|\del_x^\alpha 
								\psi\|_{L^2}^2 \\
								&\quad \leq C_1(\frac{1}{\epsilon_1}+\delta_0^2 +\delta_1^2)\|w\|_{H^s}^2+ \epsilon_1 \|\nabla\psi\|_{H^{s-1}}^2+C(\delta_0+\delta_1)\|\nabla \psi\|_{H^s}^2. 
							\end{aligned}
						\end{align}
						
					\end{lem}
					
					\vspace{2ex}
					
					\begin{proof} 
						We  consider that for $1 \leq |\alpha| \leq s $, applying $\partial_x^{\alpha-1} $ to $\eqref{stability}_2$ and multiplying by $\partial_x^\alpha \psi$, and integrating over $\mathbb{R}^3$ to get 
						\begin{align}
							\frac{\del}{\del t}(\del_x^{\alpha-1}w, \del_x^\alpha \psi) + \|\del_x^\alpha \psi\|_{L^2}^2 
							= \sum_{k=1}^8\tilde{G}_{\alpha,k}, 
						\end{align}
						where 
						\begin{align}
							\begin{aligned}
								\tilde{G}_{\alpha, 1} & = (\del_x^{\alpha-1} w, \del_x^\alpha \psi), \\
								\tilde{G}_{\alpha, 2} &  = -(\del_x^{\alpha-1} w, \del_x^\alpha \psi), \\
								\tilde{G}_{\alpha, 3} & = -(\del_x^{\alpha-1} (v_{per} \cdot \nabla w), \del_x^\alpha \psi), \\
								\tilde{G}_{\alpha, 4} &  = -(\del_x^{\alpha-1}(w\cdot \nabla v_{per}), \del_x^\alpha \psi), \\
								\tilde{G}_{\alpha,5} &= -(\del_x^{\alpha-1} (g_4 (a_{per}, \psi)\nabla \psi),\del_x^\alpha \psi), \\
								\tilde{G}_{\alpha,6} &= -(\del_x^{\alpha-1} (g_5(a_{per}, \psi)\nabla a_{per}), \del_x^\alpha \psi), \\
								\tilde{G}_{\alpha,7} &= -(\del_x^{\alpha-1} (w\cdot \nabla w), \del_x^\alpha \psi).  
								%	G_{\alpha, 8} &= -(\del_x^{\alpha-1} (g_4 (\psi) \nabla \psi),\del_x^\alpha \psi). 
							\end{aligned}
						\end{align}
						The Sobolev inequality with Lemma \ref{lem2.3.} directly show that 
						\begin{align}
							\begin{aligned}
								\sum_{k=2}^8 \tilde{G}_{\alpha,k} \leq \frac{C}{\epsilon_1}\|w\|_{H^s}^2 +\epsilon_1 \|\nabla \psi\|_{H^{s-1}} +C(\delta_0+\delta_1)\|\nabla \psi\|_{H^s}^2. 
							\end{aligned}
						\end{align}
						On $\tilde{G}_{\alpha,1}$, substituting $\eqref{stability}_1$ to $\del_t \psi$,  the Sobolev inequality with Lemma \ref{lem2.3.} verifies that 
						\begin{align*}
							\tilde{G}_{\alpha,1} \leq C_1(\frac{1}{\epsilon_1}+\delta_0^2 +\delta_1^2)\|w\|_{H^s}^2+ \epsilon_1 \|\nabla\psi\|_{H^{s-1}}^2.
						\end{align*}
					Combining the above estimates together, we prove \eqref{energy-disspation-stability}.
					\end{proof}
					
					\vspace{2ex}
					We take $\epsilon =\frac{1}{4}$, $\epsilon_1 =\frac{1}{2}$ and a large constant $C_2$ satisfying that 
					$C_2 > 4C_1$ independent $\delta_0$ and $\delta_1$.  
					When we take summation on $\alpha$ and 
					add \eqref{zeroenergy-stability}, $C_2 \times \eqref{energy-higer-order-stability}$ and  \eqref{energy-disspation-stability} with taking small  $\delta_0$ and $\delta_1$, we get $\eqref{energy-stability}_2$. This completes the proof.
					
					\
					\
					\
					\
					\
					\

					\noindent {\bf Acknowledgements.} 
					The first author  is supported by Anhui Provincial Natural Science Foundation (Grant No. 2408085QA031). The second author is supported by JSPS grant whose number is 22K13946.

					%------------------------------------------------------------
					%     ---------------------------------

					\end{document}